\documentclass[leqno,11pt]{article}
\usepackage[utf8]{inputenc}
\usepackage[T1]{fontenc}
\usepackage{microtype}

\usepackage[dvipsnames]{xcolor}
\usepackage{xcolor-solarized}

\usepackage{calc}
\usepackage[a4paper]{geometry}

\ifdefined\screenLayout
  \pagecolor{solarized-base3}
  \color{solarized-base03}
  \usepackage{fancyhdr}
\fi

\usepackage{amsmath}
\usepackage{amsthm}

\usepackage{tikz-cd}
\usetikzlibrary{arrows} 
\tikzset{
  commutative diagrams/.cd, 
  arrow style=tikz, 
  diagrams={>=stealth}
}
\usetikzlibrary{matrix,decorations.pathreplacing,calc}
\usetikzlibrary{graphs,graphs.standard}
\usepackage{pgfplots}
\usepgfplotslibrary{external}
\usepackage{afterpage}
\usepackage{pdflscape}

\usepackage{colortbl}
\usepackage{adforn}
\usepackage{nameref}

\usepackage{textcomp}
\usepackage[sb]{libertine}
\usepackage[varqu,varl]{zi4}%
\usepackage[libertine,bigdelims,vvarbb]{newtxmath}

\IfPackageAtLeastTF{superiors}{2.0}{%
  \usepackage[supsfam=LibertinusSerif-Sup,supscaled=1.2,raised=-.13em]{superiors}
}{%
  \usepackage[supstfm=libertinesups,supscaled=1.2,raised=-.13em]{superiors}
}

\useosf
\usepackage[scr=boondox,cal=euler]{mathalfa}
\usepackage{marvosym}
\usepackage{slashed}
\usepackage{esint} % SLASHED INTEGRAL

\usepackage[ngerman,english]{babel}
\usepackage{imakeidx}
\makeindex[intoc]
\usepackage{csquotes}
\usepackage[
  backend=biber,
  hyperref=true,
  backref=true,
  isbn=false,
  doi=true,
  natbib=true,
  eprint=true,
  useprefix=true,
  maxcitenames=99,
  maxbibnames=99,  
  maxalphanames=99, 
  minalphanames=99,
  safeinputenc,
  style=alphabetic,
  citestyle=alphabetic,
  block=space,
  datamodel=preamble/ext-eprint,
  sorting=nyt
]{biblatex}
\usepackage[
  bookmarksnumbered = true,
  hypertexnames = false,
  colorlinks    = true,
  citecolor     = solarized-blue,
  linkcolor     = solarized-blue,
  urlcolor      = solarized-blue,
  breaklinks
]{hyperref}

\DeclareFieldFormat{url}{%
  \href{#1}{\ComputerMouse}
}
\DeclareFieldFormat{doi}{%
  \mkbibacro{DOI}\addcolon\space\href{https://doi.org/#1}{#1}
}
\makeatletter
\DeclareFieldFormat{arxiv}{%
  arXiv\addcolon\space\href{http://arxiv.org/\abx@arxivpath/#1}{#1}
}
\makeatother
\DeclareFieldFormat{mr}{%
  MR\addcolon\space\href{http://www.ams.org/mathscinet-getitem?mr=MR#1}{#1}
}
\DeclareFieldFormat{zbmath}{%
  zbMATH\addcolon\space\href{http://zbmath.org/?q=an:#1}{#1}
}
\renewbibmacro*{eprint}{%
  \printfield{arxiv}%
  \newunit\newblock
  \printfield{mr}%
  \newunit\newblock
  \printfield{zbmath}%
  \newunit\newblock
  \iffieldundef{eprinttype}
  {\printfield{eprint}}
  {\printfield[eprint:\strfield{eprinttype}]{eprint}}
}

\AtEveryBibitem{%
  \clearlist{address}%
}
\DeclareFieldFormat[article,inproceedings,inbook,incollection,phdthesis,thesis]{title}{\textit{#1}}
\renewbibmacro{in:}{}
\newcommand{\printreferences}{\raggedright\printbibliography}

\usepackage[inline,shortlabels]{enumitem}

\usepackage{subcaption}

\usepackage[yyyymmdd]{datetime}

\usepackage{etoolbox}
\ifundef{\abstract}{}{\patchcmd{\abstract}%
    {\quotation}{\quotation\noindent\ignorespaces}{}{}}

\usepackage[super]{nth}

\usepackage{thmtools}
\usepackage[framemethod=TikZ]{mdframed}

\numberwithin{equation}{section}

\renewcommand{\qedsymbol}{$\blacksquare$}

\newcommand{\CorollaryQED}{\qedsymbol}
\newcommand{\ConjectureQED}{$\square$}
\newcommand{\SituationQED}{$\times$}
\newcommand{\DefinitionQED}{$\bullet$}
\newcommand{\NotationQED}{$\circ$}
\newcommand{\ExampleQED}{$\spadesuit$}
\newcommand{\RemarkQED}{$\clubsuit$}
\newcommand{\ExerciseQED}{?!}

\ifdefined\screenLayout
  \declaretheoremstyle[
  bodyfont=\itshape,
  mdframed={    
    backgroundcolor=solarized-base3!90!solarized-blue,
    linewidth=0,
    innerleftmargin=.5em,
    innerrightmargin=.5em,
    innertopmargin=.5em,
    innerbottommargin=.5em,
    leftmargin=-.5em,
    rightmargin=-.5em,
  }
]{theorem}

\declaretheoremstyle[
mdframed={
  backgroundcolor=solarized-base3!90!solarized-green,
  linewidth=0,
  innerleftmargin=.5em,
  innerrightmargin=.5em,
  innertopmargin=.5em,
  innerbottommargin=.5em,
  leftmargin=-.5em,
  rightmargin=-.5em,
  }
]{definition}

\declaretheoremstyle[
  mdframed={
    backgroundcolor=solarized-base3!90!solarized-yellow,
    linewidth=0,
    innerleftmargin=.5em,
    innerrightmargin=.5em,
    innertopmargin=.5em,
    innerbottommargin=.5em,
    leftmargin=-.5em,
    rightmargin=-.5em,
  }
]{example}

\declaretheoremstyle[
  mdframed={
    backgroundcolor=solarized-base3!90!solarized-orange,
    linewidth=0,
    innerleftmargin=.5em,
    innerrightmargin=.5em,
    innertopmargin=.5em,
    innerbottommargin=.5em,
    leftmargin=-.5em,
    rightmargin=-.5em,
  }
]{remark}
\else
  \declaretheoremstyle[
  bodyfont=\itshape
  ]{theorem}
  \declaretheoremstyle[]{definition}
  \declaretheoremstyle[]{example}
  \declaretheoremstyle[]{remark}
\fi

\declaretheorem[numberlike=equation,style=theorem]{theorem}
\declaretheorem[numbered=no,name=Theorem,style=theorem]{theorem*}
\declaretheorem[numberlike=equation,name=Lemma,style=theorem]{lemma}
\declaretheorem[numberlike=equation,name=Proposition,style=theorem]{prop}
\declaretheorem[numberlike=equation,name=Corollary,qed=\CorollaryQED,style=theorem]{cor}

\declaretheorem[numberlike=equation,name=Definition,style=definition,qed=\DefinitionQED]{definition}
\declaretheorem[numbered=no,name=Definition,style=definition,qed=\DefinitionQED]{definition*}

\declaretheorem[numberlike=equation,style=definition,qed=\ExampleQED]{example}

\declaretheorem[numberlike=equation,style=remark,qed=\RemarkQED]{remark}
\declaretheorem[numbered=no,style=remark,name=Remark,qed=\RemarkQED]{remark*}

\def\makeautorefname#1#2{\AtBeginDocument{\expandafter\def\csname#1autorefname\endcsname{#2}}}
\makeautorefname{table}{Table}        
\makeautorefname{chapter}{Chapter}
\makeautorefname{section}{Section}
\makeautorefname{subsection}{Section}
\makeautorefname{subsubsection}{Section}
\makeautorefname{footnote}{Footnote}
\AtBeginDocument{\def\itemautorefname~#1\null{(#1)\null}}
\AtBeginDocument{\def\equationautorefname~#1\null{(#1)\null}}

\newtheorem{step}{Step}

\numberwithin{substep}{step}
\makeautorefname{step}{Step}
\makeautorefname{substep}{Step}

\makeautorefname{case}{Case}
\makeautorefname{substep}{Step}
\AddToHook{env/proof/begin}{%
  \setcounter{step}{0}%
  \setcounter{substep}{0}%
}

\setlist[description]{leftmargin=!,labelindent=1em}
\setlist[enumerate]{label={\rm (\arabic*)},ref=\arabic*}
\setlist[enumerate,2]{label={\rm (\alph*)},ref=\theenumi.\alph*}
\setlist[enumerate,3]{label={\rm (\roman*)},ref=\theenumii.\roman*}

\let\C\undefined

\usepackage{bm}
\usepackage{mathtools} % FOR PAIREDDELIMITERS
\usepackage{stmaryrd} % FOR SUPER LIE BRACKET

\DeclareFontFamily{U}{mathx}{\hyphenchar\font45}
\DeclareFontShape{U}{mathx}{m}{n}{
      <5> <6> <7> <8> <9> <10>
      <10.95> <12> <14.4> <17.28> <20.74> <24.88>
      mathx10
      }{}
\DeclareSymbolFont{mathx}{U}{mathx}{m}{n}
\DeclareFontSubstitution{U}{mathx}{m}{n}
\DeclareMathAccent{\widecheck}{0}{mathx}{"71}
\DeclareMathAccent{\wideparen}{0}{mathx}{"75}

\DeclareMathOperator{\Aut}{Aut}

\DeclareMathOperator{\Diff}{Diff}
\DeclareMathOperator{\End}{End}

\DeclareMathOperator{\ext}{ext}

\DeclareMathOperator{\Fr}{Fr}

\DeclareMathOperator{\HF}{\HF}

\DeclareMathOperator{\Hom}{Hom}

\DeclareMathOperator{\Lie}{Lie}

\DeclareMathOperator{\PD}{PD}

\DeclareMathOperator{\Sym}{Sym}

\DeclareMathOperator{\coker}{coker}

\DeclareMathOperator{\im}{im}
\DeclareMathOperator{\ind}{index}

\DeclareMathOperator{\res}{res}
\DeclareMathOperator{\rk}{rk}

\DeclareMathOperator{\spec}{spec}
\DeclareMathOperator{\supp}{supp}

\DeclareMathOperator{\tr}{tr}

\DeclarePairedDelimiter\paren{\lparen}{\rparen}

\DeclarePairedDelimiter{\Abs}{\|}{\|}

\DeclarePairedDelimiter{\Inner}{\langle}{\rangle}

\DeclarePairedDelimiter{\abs}{\lvert}{\rvert}
\DeclarePairedDelimiter{\bracket}{\langle}{\rangle}

\DeclarePairedDelimiter{\set}{\lbrace}{\rbrace}
\def\({\left(}
\def\){\right)}
\def\<{\left\langle}
\def\>{\right\rangle}

\newcommand{\CP}{{\C P}} 

\newcommand{\Cl}{\mathrm{C\ell}}
\newcommand{\C}{{\mathbf{C}}}

\newcommand{\N}{{\mathbf{N}}}

\newcommand{\Or}{\mathrm{or}}

\newcommand{\R}{\mathbf{R}}

\newcommand{\Span}[1]{\bracket{#1}}

\newcommand{\Uh}{\mathrm{Uh}}

\newcommand{\Vect}{\mathrm{Vect}}

\newcommand{\Z}{\mathbf{Z}}

\newcommand{\ch}{\mathrm{ch}}

\newcommand{\co}{\mskip0.5mu\colon\thinspace}

\newcommand{\dR}{\mathrm{dR}}
\newcommand{\defined}[2][\key]{\def\key{#2}\textbf{#2}\index{#1}}

\newcommand{\del}{\partial}
\newcommand{\ev}{\mathrm{ev}}

\newcommand{\id}{\mathrm{id}}
\newcommand{\incl}{\hookrightarrow}

\newcommand{\into}{\hookrightarrow}
\newcommand{\iso}{\cong}

\newcommand{\loc}{\mathrm{loc}}

\newcommand{\ob}{\mathrm{ob}}

\newcommand{\one}{\mathbf{1}}
\newcommand{\onto}{\twoheadrightarrow}

\newcommand{\pr}{\mathrm{pr}}
\newcommand{\qandq}{\quad\text{and}\quad}

\newcommand{\qand}{\quad\text{and}}

\newcommand{\qwithq}{\quad\text{with}\quad}
\newcommand{\qwith}{\quad\text{with}}

\newcommand{\vol}{\mathrm{vol}}

\renewcommand{\H}{\mathbf{H}}
\renewcommand{\Im}{\operatorname{Im}}
\renewcommand{\O}{\mathrm{O}}
\renewcommand{\P}{\mathbf{P}}
\renewcommand{\Re}{\operatorname{Re}}

\renewcommand{\emptyset}{\varnothing}
\renewcommand{\epsilon}{\varepsilon}
\renewcommand{\setminus}{{\backslash}}
\renewcommand{\sp}{\mathfrak{sp}}

\renewcommand{\leq}{\leqslant}
\renewcommand{\geq}{\geqslant}

\newcommand{\Wedge}{\Lambda}

\makeatletter
\renewcommand*\env@matrix[1][*\c@MaxMatrixCols c]{%
  \hskip -\arraycolsep
  \let\@ifnextchar\new@ifnextchar
  \array{#1}}

\renewcommand\xleftrightarrow[2][]{%
  \ext@arrow 9999{\longleftrightarrowfill@}{#1}{#2}}
\newcommand\longleftrightarrowfill@{%
  \arrowfill@\leftarrow\relbar\rightarrow}
\makeatother

\newcommand{\tw}{\mathrm{tw}}

\newcommand{\Stab}{\mathrm{Stab}}

\newcommand{\rd}{{\rm d}}

\newcommand{\rC}{{\rm C}}

\newcommand{\rH}{{\rm H}}

\newcommand{\rR}{{\rm R}}

\newcommand{\bp}{{\mathbf{p}}}

\newcommand{\bt}{{\mathbf{t}}}

\newcommand{\bA}{{\mathbf{A}}}
\newcommand{\bB}{{\mathbf{B}}}

\newcommand{\bD}{{\mathbf{D}}}
\newcommand{\bE}{{\mathbf{E}}}
\newcommand{\bF}{{\mathbf{F}}}

\newcommand{\bH}{{\mathbf{H}}}

\newcommand{\bK}{{\mathbf{K}}}
\newcommand{\bL}{{\mathbf{L}}}

\newcommand{\bN}{{\mathbf{N}}}

\newcommand{\bP}{{\mathbf{P}}}
\newcommand{\bQ}{{\mathbf{Q}}}
\newcommand{\bR}{{\mathbf{R}}}
\newcommand{\bS}{{\mathbf{S}}}

\newcommand{\bU}{{\mathbf{U}}}
\newcommand{\bV}{{\mathbf{V}}}

\newcommand{\bX}{{\mathbf{X}}}
\newcommand{\bY}{{\mathbf{Y}}}
\newcommand{\bZ}{{\mathbf{Z}}}

\newcommand{\sA}{\mathscr{A}}

\newcommand{\sF}{\mathscr{F}}
\newcommand{\sG}{\mathscr{G}}
\newcommand{\sH}{\mathscr{H}}

\newcommand{\sL}{\mathscr{L}}
\newcommand{\sM}{\mathscr{M}}

\newcommand{\sO}{\mathscr{O}}
\newcommand{\sP}{\mathscr{P}}

\newcommand{\sS}{\mathscr{S}}

\newcommand{\sU}{\mathscr{U}}
\newcommand{\sV}{\mathscr{V}}
\newcommand{\sW}{\mathscr{W}}

\newcommand{\fa}{{\mathfrak a}}

\newcommand{\fe}{{\mathfrak e}}

\newcommand{\fl}{{\mathfrak l}}

\newcommand{\fo}{{\mathfrak o}}

\newcommand{\fs}{{\mathfrak s}}
\newcommand{\ft}{{\mathfrak t}}
\newcommand{\fu}{{\mathfrak u}}

\newcommand{\fL}{{\mathfrak L}}

\newcommand{\fU}{{\mathfrak U}}

\newcommand{\slS}{\slashed S}

\newcommand{\BG}{\mathrm{BG}}

\newcommand{\dom}{\mathrm{dom}}
\newcommand{\Dmin}{D_{\mathrm{min}}}
\newcommand{\Dmax}{D_{\mathrm{max}}}

\newcommand{\APS}{\mathrm{APS}}
\newcommand{\GelfandRobbinQuotient}{\check\bH}

\newcommand{\Spectrum}{\sigma}
\newcommand{\Br}{\mathrm{Br}}
\newcommand{\UniversalBr}{\mathbf{Br}}
\newcommand{\SpaceOfBranchingLoci}{\mathbf{Sub}}
\newcommand{\SpaceOfMetrics}{\mathbf{Met}}
\newcommand{\SpaceOfDiracBundles}{\mathbf{Dir}}
\newcommand{\SpaceOfRamifiedLineBundles}{\mathbf{Ram}}
\newcommand{\SpaceOfSingularRamifiedLineBundles}{\mathfrak{Ram}}
\newcommand{\ProjectiveSpaceOfHarmonicSpinors}{\bP\mathbf{Harm}}
\newcommand{\BundleOfNonDegenerateAdmissibleZTwoZSpinors}{\BundleOfAdmissibleZModTwoZSpinors^*}
\newcommand{\BundleOfNonDegenerateAdmissibleChiralZTwoZSpinors}{\BundleOfAdmissibleZModTwoZSpinors^{\pm,*}}
\newcommand{\SpaceOfNonDegenerateHarmonicSpinors}{\mathbf{Harm}^*}
\newcommand{\SpaceOfHarmonicSpinors}{\mathbf{Harm}}
\newcommand{\ProjectiveSpaceOfNonDegenerateHarmonicSpinors}{\bP\mathbf{Harm}^*}
\newcommand{\ProjectiveSpaceOfNonDegenerateHarmonicOneForms}{\bP\mathbf{Harm}_{\Omega^1}^*}
\newcommand{\SpaceOfNonDegenerateHarmonicOneForms}{\mathbf{Harm}_{\Omega^1}^*}
\newcommand{\SpaceOfNonDegenerateHarmonicSelfDualTwoForms}{\mathbf{Harm}_{\Omega^+}^*}
\newcommand{\ProjectiveSpaceOfNonDegenerateHarmonicSelfDualTwoForms}{\bP\mathbf{Harm}_{\Omega^+}^*}
\newcommand{\ProjectiveThickenedSpaceOfNonDegenerateHarmonicSpinors}{\bP\mathbf{Harm}^\ft}
\newcommand{\ProjectiveSpaceOfNonDegenerateEigenSpinors}{\bP\mathbf{Eig}^*}
\newcommand{\SpaceOfNonDegenerateEigenSpinors}{\mathbf{Eig}^*}
\newcommand{\ThickenedSpaceOfNonDegenerateEigenSpinors}{\mathbf{Eig}^\ft}
\newcommand{\ProjectiveSpaceOfNonDegeneratePositiveHarmonicSpinors}{\bP\mathbf{Harm}^{+,*}}
\newcommand{\ProjectiveSpaceOfNonDegenerateNegativeHarmonicSpinors}{\bP\mathbf{Harm}^{-,*}}
\newcommand{\ProjectiveSpaceOfChiralHarmonicSpinors}{\bP\mathbf{Harm}^{\pm}}
\newcommand{\ProjectiveSpaceOfNonDegenerateChiralHarmonicSpinors}{\bP\mathbf{Harm}^{\pm,*}}
\newcommand{\SpaceOfNonDegenerateChiralHarmonicSpinors}{\mathbf{Harm}^{\pm,*}}

\newcommand{\BundleOfSingularSpinors}{\mathbf{F}}
\newcommand{\BundleOfAdmissibleSingularSpinors}{\mathbf{E}}
\newcommand{\BundleOfZModTwoZSpinors}{\bE_{0}}
\newcommand{\BundleOfAdmissibleZModTwoZSpinors}{\bE_{00}}
\newcommand{\BundleOfCoAdmissibleZModTwoZSpinors}{\bF_{00}}
\newcommand{\BundleOfResidues}{\mathbf{G}}
\newcommand{\DeformationBundle}{\mathrm{Def}}
\newcommand{\ObstructionBundle}{\mathrm{Ob}}
\newcommand{\ObstructionMap}{\mathrm{ob}}
\newcommand{\ThickenedSpaceOfNonDegenerateHarmonicSpinors}{\SpaceOfHarmonicSpinors^\bt}

\newcommand{\GaugeGroup}[1]{\sG(#1)}

\newcommand{\ExtendedGaugeGroup}[1]{\widehat\sG(#1)}

\newcommand{\Slice}{\mathbf{Sl}}
\newcommand{\DiffOp}{\mathrm{DiffOp}}
\newcommand{\lift}{\mathrm{lift}}
\newcommand{\VerticalTangentBundle}[2]{T#1/#2}
\newcommand{\UniversalDiracOperator}{\bD}
\newcommand{\UniversalResidueMap}{\mathbf{res}}
\newcommand{\UniversalExtensionMap}{\mathbf{ext}}
\newcommand{\HdR}{\iota_{\mathrm{HdR}}}

\newcommand{\WeightLog}{w_{\log}}
\newcommand{\Riem}{\mathrm{Riem}}

\newcommand{\AlgebraBundle}{\mathbb{A}}
\newcommand{\ResidueCondition}{R}
\newcommand{\ResidueBundle}{\check S}
\newcommand{\UniversalResidueBundle}{\check \bS}
\newcommand{\PeriodMap}{\mathbf{Per}}
\newcommand\fakeqed{\pushQED{\qed}\qedhere}
\renewcommand{\Z}{\mathbb{Z}}
\renewcommand{\R}{\mathbb{R}}
\renewcommand{\C}{\mathbb{C}}
\renewcommand{\H}{\mathbb{H}}
\newcommand{\RegVal}{\mathrm{RegVal}}

\author{
  Siqi He
  \and
  Gregory J.~Parker
  \and
  Thomas Walpuski
}
\title{%
  The universal moduli space of non-degenerate $\Z/2\Z$~harmonic spinors
}
\hypersetup{
  pdfauthor={Siqi He, Gregory J.~Parker, Thomas Walpuski},
  pdftitle={The universal moduli space of non-degenerate $\Z/2\Z$~harmonic spinors},
  pdflang={English}}
\date{2026-09-23}

\begin{document}

\maketitle

\begin{abstract}
  This article equips the universal moduli spaces of non-degenerate $\Z/2\Z$ harmonic spinors and eigenspinors with tame Fréchet manifold structures and proves that their natural projections to the corresponding parameter spaces are uniformly Fredholm.
  A unified construction yields the deformation theories of harmonic spinors and eigenspinors in dimension three and of chiral harmonic spinors in dimension four,
  recovering earlier work of \citet{Donaldson2021,Parker2023:Deformation,Takahashi2015}.
  As an application, a deformation theory for non-degenerate $\Z/2\Z$ harmonic self-dual $2$--forms in dimension four is established.
\end{abstract}

\tableofcontents

\section{Introduction}
\label{Sec_Introduction}

More than a decade ago,
Taubes and his disciples began to observe that the failure of compactness for a wide variety of ``new'' low-dimensional gauge theories leaves behind evidence in the form of a \defined{$\Z/2\Z$ harmonic spinor}
\cite{Taubes2012,Taubes2013,Taubes2016:SeibergWittenCompactness,Taubes2017,Taubes2022:NahmPoleKapustinWitten,Haydys2014,Walpuski2019}.
The latter have since appeared in related questions in gauge theory and calibrated geometry \cite{Taubes2014,Donaldson2016,,Doan2017c,He2022:DeformBranchedSLags},
and might play a role in the conjectural construction of invariants
of manifolds with special holonomy \cite{Donaldson2009,Haydys2017,Joyce2016,Doan2017d,Doan2024}.

In light of the above it seems pressing to understand what kind of a phenomenon the existence of $\Z/2\Z$ harmonic spinors is;
in particular, whether it persists under (small) deformations of the background parameters.
This question of the deformation theory of $\Z/2\Z$ harmonic spinors has been approached before in the work of  \citet{Donaldson2021,Parker2023:Deformation,Takahashi2015}.
The purpose of this article is to develop this deformation theory in a unified framework, building upon \cite{BeraWalpuski2025} and a suggestion by \citet{Donaldson2021}.
This framework recovers the existing deformation theories through a single universal construction and also yields a deformation theory for non-degenerate $\Z/2\Z$ harmonic self-dual $2$--forms in dimension four.

\medskip

Throughout the remainder of this article,
let $X$ be a connected closed manifold of dimension $d$,
and $S$ a Euclidean vector bundle of rank $r$ over $X$.
Denote by $\SpaceOfMetrics$ the tame Fréchet manifold of Riemannian metrics $g$ on $X$, by $\SpaceOfDiracBundles$ the tame Fréchet manifold of Dirac bundle structures $\bp = (g;\gamma,\nabla)$ on $S$ (as in \cite[Part II Definition 5.2]{Lawson1989} and \autoref{Sec_DiracOperatorsTwistedByRamifiedLineBundles}),
and by
$\SpaceOfSingularRamifiedLineBundles$ the set of isometry classes of ramified Euclidean line bundles $[\fl]$ over $X$.
The goal is to equip the universal \emph{set} of (non-zero) $\Z/2\Z$ harmonic spinors (up to scaling)
\begin{equation*}
  \ProjectiveSpaceOfHarmonicSpinors
  \coloneq
  \coprod_{(\bp,[\fl]) \in \SpaceOfDiracBundles \times \SpaceOfSingularRamifiedLineBundles}
  \bP\ker\paren[\big]{D_\bp^\fl \co H^1\Gamma\paren{X\setminus \Br(\fl),S \otimes \fl} \to L^2\Gamma\paren{X\setminus \Br(\fl),S \otimes \fl} }
\end{equation*}
with some geometric structure and understand the nature of the projection map $\pr_\SpaceOfDiracBundles \co \ProjectiveSpaceOfHarmonicSpinors \to \SpaceOfDiracBundles$.
Here $D_\bp^\fl$ denotes the Dirac operator associated with $\bp$ twisted by $\fl$, and $\bP$ denotes the real projectivisation;
that is: $\P V \coloneq (V \setminus \set{0})/\R^\times$.
This structure would upgrade the set $\ProjectiveSpaceOfHarmonicSpinors$ to the \defined{universal moduli space of $\Z/2\Z$ harmonic spinors}.

The present article does not deal with $\ProjectiveSpaceOfHarmonicSpinors$ itself,
but only with the subset
\begin{equation*}
  \ProjectiveSpaceOfNonDegenerateHarmonicSpinors \subseteq \ProjectiveSpaceOfHarmonicSpinors
\end{equation*}
consisting of $\Z/2\Z$ harmonic spinors $(\bp,[\fl];[\Phi])$ which are \defined{non-degenerate} in the following sense:
\begin{enumerate}
\item
  \label{It_NonDegenerate_SmoothBranchingLocus}
  the branching locus $\Br(\fl) \subseteq X$ is a submanifold, henceforth also denoted by $Z$, and
\item
  \label{It_NonDegenerate_VanishingOrder}
  $\abs{\Phi}$ vanishes exactly to order $\frac12$ at every $x \in \Br(\fl)$;
  that is: $\abs{\Phi} \asymp_\Phi r^{1/2}$.
  Here $r$ denotes the distance to $\Br(\fl)$.
\end{enumerate}
The first restriction is imposed because $\SpaceOfSingularRamifiedLineBundles$ is rather wild and it is far from obvious what structure to put on it,
unless $d = 2$, see \cite[§3.1]{TaubesWu2021:Z2ZEigenfunctionsTopologicalAspects}.
The subset $\SpaceOfRamifiedLineBundles \subseteq \SpaceOfSingularRamifiedLineBundles$ of isometry classes of ramified Euclidean line bundles $[\fl]$ with smooth branching locus $\Br(\fl)$, however, can be equipped with the structure of a tame Fréchet manifold in a straightforward manner.
With this in mind, it is natural to try to equip $\ProjectiveSpaceOfNonDegenerateHarmonicSpinors$ with the structure of a tame 
Fréchet manifold.
As explained above, the crucial difficulty is that $D_\bp^\fl \co H^1\Gamma\paren{X\setminus \Br(\fl), S\otimes \fl} \to L^2\Gamma\paren{X\setminus \Br(\fl), S\otimes \fl}$ is left semi-Fredholm but has an $\infty$-dimensional cokernel, and the naive expectation is that this can be compensated by allowing $\Br(\fl)$ to vary.
Indeed, provided $\abs{\Phi}$ vanishes exactly to order $\frac12$ along $\Br(\fl)$,
this additional variation essentially corresponds to passing to a closed extension of $D_\bp^\fl$.
If $\rk S = 4$,
then this extension is Fredholm.
This observation and a suitable framework enable Nash--Moser theory to prove the following.

\begin{theorem}
  \label{Thm_ProjectiveSpaceOfNonDegenerateHarmonicSpinors}
  If $\dim X=3$ and $\rk S = 4$,
  then $\ProjectiveSpaceOfNonDegenerateHarmonicSpinors$ can be given the structure of a tame Fréchet manifold such that $\pr_\SpaceOfDiracBundles \co \ProjectiveSpaceOfNonDegenerateHarmonicSpinors \to \SpaceOfDiracBundles$ is uniformly Fredholm of index $-1$.
\end{theorem}

The notion of a uniform Fredholm map between tame Fréchet manifolds is defined in \autoref{Def_UniformlyFredholmMap}.
Crucially, by \autoref{Prop_UniformFredholmMapsAreStableUnderTransverseBaseChange}, it is stable under transverse base change.
In particular,
if $B$ is a (finite-dimensional) manifold and $\Gamma \co B \to \SpaceOfDiracBundles$ is a smooth map which is transverse to $\pr_\SpaceOfDiracBundles$,
then the pullback $\Gamma^*\ProjectiveSpaceOfNonDegenerateHarmonicSpinors$ is a manifold of dimension $\dim B - 1$.
In this sense,
the appearance of $\Z/2\Z$ harmonic spinors is a codimension one phenomenon in dimension three.

\begin{remark}
  \label{Rmk_PriorAndRelatedWork_2D}
  If $\dim X = 2$,
  then $\SpaceOfRamifiedLineBundles$ is finite-dimensional,
  which drastically simplifies the issue.
  The analogue of  \autoref{Thm_ProjectiveSpaceOfNonDegenerateHarmonicSpinors} in this setting is essentially \cite[Theorem 3.40]{Doan2024}.    
\end{remark}

\begin{remark}
  \label{Rmk_PriorAndRelatedWork_3D}
  \autoref{Thm_ProjectiveSpaceOfNonDegenerateHarmonicSpinors} is essentially \cite[Theorem 1.5]{Parker2023:Deformation}; see also \cite[Theorem 1.5]{Takahashi2015}.
  The proof of \autoref{Thm_ProjectiveSpaceOfNonDegenerateHarmonicSpinors} presented in \autoref{Sec_Z2ZHarmonicSpinorsInDimensionThree} is based on a suggestion by \citet{Donaldson2021:Talk} and conceptually cleaner and more direct than the argument in \cite{Parker2023:Deformation}.
\end{remark}

\begin{remark}
  \label{Rmk_RelaxingNonDegeneracy}
  The restriction to $\ProjectiveSpaceOfNonDegenerateHarmonicSpinors$ is not unproblematic,
  since it prevents $\pr_\SpaceOfDiracBundles$ from being proper;
  cf.~\cite[§1.4]{TaubesWu2021:Z2ZEigenfunctionsTopologicalAspects}.
  Therefore,
  it would be interesting to analyse what happens if the non-degeneracy condition is relaxed either by allowing $\Br(\fl)$ to be singular or $\abs{\Phi}$ to vanish to a higher order at certain points in $\Br(\fl)$.
  \cite{HaydysTakahashiMazzeo2023:Index} contains some initial work in the direction of allowing $\Br(\fl)$ to be a graph.
\end{remark}

\begin{remark}
  \label{Rmk_BetterThanTameFréchetManifoldStructure}
  The tame Fréchet manifold structure on $\ProjectiveSpaceOfNonDegenerateHarmonicSpinors$ is obtained by exhibiting it as a tame Fréchet submanifold of the projectivisation $\bP\BundleOfAdmissibleZModTwoZSpinors$ of a tame Fréchet space bundle $\BundleOfAdmissibleZModTwoZSpinors$ over $\SpaceOfDiracBundles \times \SpaceOfRamifiedLineBundles$.
  Moreover, locally, $\ProjectiveSpaceOfNonDegenerateHarmonicSpinors$ can be thickened to a submanifold $\ProjectiveThickenedSpaceOfNonDegenerateHarmonicSpinors \subseteq \bP\BundleOfAdmissibleZModTwoZSpinors$ such that $\sV \coloneq \pr_\SpaceOfDiracBundles(\ProjectiveThickenedSpaceOfNonDegenerateHarmonicSpinors)$ is open and $\pr_\sV \co \ProjectiveThickenedSpaceOfNonDegenerateHarmonicSpinors \to \sV$ is a uniform submersion.
  This additional information can be useful to determine pullbacks of $\ProjectiveSpaceOfNonDegenerateHarmonicSpinors$ along tame smooth maps that fail to be transverse to $\pr_\SpaceOfDiracBundles$.
\end{remark}

The argument in the proof of \autoref{Thm_ProjectiveSpaceOfNonDegenerateHarmonicSpinors} is somewhat robust to the addition of mild lower-order terms.
As an application one obtains the following result regarding
\begin{equation*}
  \ProjectiveSpaceOfNonDegenerateEigenSpinors
  \coloneq
  \coprod_{(\bp,[\fl],\lambda) \in \SpaceOfDiracBundles \times \SpaceOfRamifiedLineBundles \times \R}
  \bP \set[\big]{
    \Phi \in H^1\Gamma\paren{X\setminus\Br(\fl), S \otimes \fl} : \Phi ~\text{satisfies \autoref{It_NonDegenerate_VanishingOrder} and~} D_\bp^\fl\Phi = \lambda \Phi
  },
\end{equation*}
the \defined{universal moduli space of non-degenerate $\Z/2\Z$ eigenspinors}.

\begin{theorem}[{cf.~\cite[Corollary 1.6]{Parker2023:Deformation}}]
  \label{Thm_ProjectiveSpaceOfNonDegenerateEigenSpinors}
  If $\dim X=3$ and $\rk S = 4$,
  then $\ProjectiveSpaceOfNonDegenerateEigenSpinors$ can be given the structure of a tame Fréchet manifold such that $\pr_\SpaceOfDiracBundles \co \ProjectiveSpaceOfNonDegenerateEigenSpinors \to \SpaceOfDiracBundles$ is uniformly Fredholm of index $0$ and the projection map $\lambda \co \ProjectiveSpaceOfNonDegenerateEigenSpinors \to \R$ is smooth.
\end{theorem}

Of course, $\ProjectiveSpaceOfNonDegenerateHarmonicSpinors = \lambda^{-1}(0) \subseteq \ProjectiveSpaceOfNonDegenerateEigenSpinors$.

\medskip %%% 4D

The hypothesis $\rk S = 4$ is essential in the above.
As is explained at the end of \autoref{Sec_UniversalBundleOfAdmissibleZModTwoZSpinors},
the infinitesimal deformation theory of non-degenerate $\Z/2\Z$ harmonic spinors at $(\bp,[\fl];\Phi)$ is governed by the Dirac operator $D_\bp^\fl$ subject to a residue condition arising from the inclusion $NZ \incl \ResidueBundle$
of the normal bundle of $Z$ into the bundle of possible residue values,
encoding the leading order behaviour of $\Phi$ at $Z$.
For this residue condition to be elliptic it is necessary that $2 = \rk NZ = \frac12 \rk \ResidueBundle = \frac12 \rk S$.

If $\dim X \geq 4$,
then $\rk S \geq 8$ for every Dirac bundle $S$;
however, if $X$ is oriented and $\dim X = 0 \pmod 4$,
then every choice of $\bp \in \SpaceOfDiracBundles$ orthogonally decomposes $S$ into
\begin{equation*}
  S = S^+ \oplus S^-
  \qwithq
  S^\pm \coloneq \ker\paren{\epsilon \mp \one} \qandq \epsilon \coloneq \gamma(\vol_g);
\end{equation*}
moreover, for every $[\fl] \in \SpaceOfSingularRamifiedLineBundles$,
$D_\bp^\fl$ decomposes accordingly as
\begin{equation*}
  D_\bp^\fl
  =
  \begin{pmatrix}
    0 &  D_\bp^{\fl,-} \\
    D_\bp^{\fl,+} & 0
  \end{pmatrix}.
\end{equation*}
The universal set of (non-zero) chiral $\Z/2\Z$ harmonic spinors (up to scaling)
\begin{equation*}
  \ProjectiveSpaceOfChiralHarmonicSpinors
  \coloneq
  \coprod_{(\bp,[\fl]) \in \SpaceOfDiracBundles \times \SpaceOfSingularRamifiedLineBundles}
  \bP\ker\paren[\big]{D_\bp^{\fl,\pm} \co H^1\Gamma\paren{X\setminus \Br(\fl),S^\pm \otimes \fl} \to L^2\Gamma\paren{X\setminus \Br(\fl),S^\mp \otimes \fl} }
\end{equation*}
contains the subset 
\begin{equation*}
  \ProjectiveSpaceOfNonDegenerateChiralHarmonicSpinors \subseteq  \ProjectiveSpaceOfChiralHarmonicSpinors
\end{equation*}
of non-degenerate chiral $\Z/2\Z$ harmonic spinors (up to scaling).
The restriction of the closed extension of $D_\bp^\fl$ considered above leads to a closed extension of $D_\bp^{\fl,\pm}$ which is Fredholm if $\rk S = 8$.
The proof of \autoref{Thm_ProjectiveSpaceOfNonDegenerateHarmonicSpinors} adapts with minor cosmetic modifications to establish the following result.

\begin{theorem}
  \label{Thm_ProjectiveSpaceOfNonDegenerateChiralHarmonicSpinors}
  If $X$ is oriented, $\dim X = 4$, and $\rk S = 8$,
  then $\ProjectiveSpaceOfNonDegenerateChiralHarmonicSpinors$ can be given the structure of a tame Fréchet manifold such that $\pr_{\SpaceOfDiracBundles} \co \ProjectiveSpaceOfNonDegenerateChiralHarmonicSpinors \to \SpaceOfDiracBundles$ is uniformly Fredholm of index
  \begin{equation*}
    \mp\frac14 \Inner{p_1(S),[X]} \pm \sigma(X) + \frac12\chi(Z) \mp \frac54\Inner{e(NZ),[Z]} - 1.
  \end{equation*}
\end{theorem}

\begin{remark}
  \label{Rmk_UniquenessOfChirality}
  If $X$ is oriented, $\dim X = 4$, and $\rk S = 8$,
  then $\epsilon \coloneq \gamma(\vol_g)$ is the unique chirality operator on $S$ up to sign.
  Indeed,
  if $\delta$ is another chirality operator,
  then it commutes with $\epsilon$;
  therefore, $\delta\epsilon = \epsilon\delta$ is Clifford linear, self-adjoint, and parallel;
  hence, by Schur's Lemma $\delta = \pm \epsilon$.
  In particular, there is no additional loss of generality in the above setup.
\end{remark}

\medskip

The above results (and their proof) can be brought to bear on $\Z/2\Z$ harmonic $1$--forms as follows.

Assume that $\dim X = 3$, $X$ is oriented, and $S = \R \oplus T^*X$.
For every $g \in \SpaceOfMetrics$ there is a Dirac bundle structure on $S$ whose associated Dirac operator is
\begin{equation*}
  D_g
  \coloneq
  \begin{pmatrix}
    0 & \rd^* \\
    \rd & *\rd
  \end{pmatrix}
  \co \Omega^0(X) \oplus \Omega^1(X) \to \Omega^0(X) \oplus \Omega^1(X),
\end{equation*}
a truncation of the \defined{Hodge--de Rham operator} $\rd + \rd^* \co \Omega^\bullet(X) \to \Omega^\bullet(X)$.
This defines a tame smooth map $\HdR \co \SpaceOfMetrics \to \SpaceOfDiracBundles$.
By integration by parts,
if $(f,\alpha) \in \Omega^0(X,\fl) \oplus \Omega^1(X,\fl)$ is $\Z/2\Z$ harmonic, then $f = 0$---except in the edge case $\Br(\fl) = \emptyset$, which shall be disregarded in the remaining discussion.
Therefore,
the pullback
\begin{equation*}
  \ProjectiveSpaceOfNonDegenerateHarmonicOneForms \coloneq \HdR^*\ProjectiveSpaceOfNonDegenerateHarmonicSpinors
\end{equation*}
is the \defined{universal moduli space of non-degenerate $\Z/2\Z$ harmonic $1$--forms}.
Although $\HdR$ is \emph{not transverse} to $\pr \co \SpaceOfNonDegenerateHarmonicSpinors \to \SpaceOfDiracBundles$,
the theory underlying the proof of \autoref{Thm_ProjectiveSpaceOfNonDegenerateHarmonicSpinors}, see \autoref{Rmk_BetterThanTameFréchetManifoldStructure}, provides detailed information on $\ProjectiveSpaceOfNonDegenerateHarmonicOneForms$ and the projection map $\pr_\SpaceOfMetrics \co \ProjectiveSpaceOfNonDegenerateHarmonicOneForms \to \SpaceOfMetrics$.

There is a finite rank vector bundle $\sH^1 \to \SpaceOfMetrics\times\SpaceOfRamifiedLineBundles$ whose fibre over $(g,[\fl])$ is $\sH^1(g,\fl)$, the space of $L^2$ harmonic $1$--forms $\alpha \in L^2\Omega^1(X,\fl)$.
A singular version of Hodge theory \cite{Teleman1983:SignatureLipschitz} induces an isomorphism $\sH^1(g,\fl) \iso \rH_{\dR}^1(\tilde X)^-$.
Here $\tilde X \to X$ denotes the branched double cover induced by $\fl$ and the superscript indicates taking the anti-invariant part with respect to the sheet-swapping involution.
Evidently, there is a canonical inclusion map $\jmath \co \ProjectiveSpaceOfNonDegenerateHarmonicOneForms \incl \bP\sH^1$.

\begin{theorem}
  \label{Thm_ProjectiveSpaceOfNonDegenerateHarmonicOneForms}
  If $\dim X=3$ and $X$ is oriented,
  then $\ProjectiveSpaceOfNonDegenerateHarmonicOneForms$ can be given the structure of a tame Fréchet manifold such that:
  \begin{enumerate}
  \item
    \label{Thm_ProjectiveSpaceOfNonDegenerateHarmonicOneForms_Submersion}
    $\pr_\SpaceOfMetrics \co \ProjectiveSpaceOfNonDegenerateHarmonicOneForms \to \SpaceOfMetrics$ is a uniform submersion, and
  \item
    \label{Thm_ProjectiveSpaceOfNonDegenerateHarmonicOneForms_VerticalTangenBundle}
    $\jmath$ is smooth and induces an isomorphism $\VerticalTangentBundle{\ProjectiveSpaceOfNonDegenerateHarmonicOneForms}{\SpaceOfMetrics} \iso \jmath^* \VerticalTangentBundle{\bP\sH^1}{\paren{\SpaceOfMetrics \times \SpaceOfRamifiedLineBundles}}$
    of the vertical tangent bundles.
  \end{enumerate}
  In particular,
  the index of $\pr_\SpaceOfMetrics$ at $(g,[\fl];[\alpha]) \in \ProjectiveSpaceOfNonDegenerateHarmonicOneForms$ agrees with $\dim \rH_{\dR}^1(\tilde X)^- - 1$.
\end{theorem}

\begin{remark}
  \label{Rmk_ProjectiveSpaceOfNonDegenerateHarmonicOneForms_PriorWork}
  \autoref{Thm_ProjectiveSpaceOfNonDegenerateHarmonicOneForms} recovers \cite[Theorem 1.1]{Donaldson2021} in dimension three;
  see also \cite{HeParker2024} for an adaptation of \cite{Parker2023:Deformation} to $\Z/2\Z$ harmonic $1$--forms.  
  The proof of \autoref{Thm_ProjectiveSpaceOfNonDegenerateHarmonicOneForms} presented in \autoref{Sec_Z2ZHarmonicOneForms} is different from the proof in \cite{Donaldson2021},
  but, of course, related to the alternative proof of \cite[Theorem 1.1]{Donaldson2021} mentioned in \cite{Donaldson2021:Talk}.
\end{remark}

\begin{remark}
  \label{Rmk_ProjectiveSpaceOfNonDegenerateHarmonicOneForms_PeriodMap}
  The vector bundle $\sH^1$ has a flat Gauß--Manin connection.  
  In particular, if $\sU \subseteq \SpaceOfMetrics \times \SpaceOfRamifiedLineBundles$ is a simply-connected neighbourhood of $(g,[\fl])$,
  then it induces a trivialisation $\sH^1|_\sU \iso \rH_{\dR}^1(\tilde X)^-$.
  As a consequence of \autoref{Thm_ProjectiveSpaceOfNonDegenerateHarmonicOneForms},
  every $(g,[\fl];[\alpha]) \in \ProjectiveSpaceOfNonDegenerateHarmonicOneForms$ has an open neighbourhood $\sU \subseteq \ProjectiveSpaceOfNonDegenerateHarmonicOneForms$ such that there is a well-defined \defined{local period map} $\PeriodMap \co \sU \to \bP\rH_{\dR}^1(\tilde X)^-$ and $\pr_\SpaceOfMetrics \times \PeriodMap \co \sU \to \SpaceOfMetrics \times \bP\rH_{\dR}^1(\tilde X)^-$ is an open embedding.
  This formulation is closer to the one found in \cite[Theorem 1.1]{Donaldson2021}.
\end{remark}

There is a variant of the preceding discussion assuming that $\dim X = 4$, $X$ is oriented, and $S = \R \oplus \Wedge^{+} T^*X \oplus T^*X$.
For every $g \in \SpaceOfMetrics$ there is a Dirac bundle structure on $S$ whose associated Dirac operator is
\begin{equation*}
  D_g
  \coloneq
  \begin{pmatrix}
    0 & \rd & \sqrt{2}\rd^* \\
    \rd^* & 0 & 0 \\
    \sqrt{2}\rd^+ & 0 & 0
  \end{pmatrix}
  \co \Omega^1(X) \oplus \Omega^0(X) \oplus \Omega^+(X) \to \Omega^1(X) \oplus \Omega^0(X) \oplus \Omega^+(X).
\end{equation*}
This defines a tame smooth map $\HdR \co \SpaceOfMetrics \to \SpaceOfDiracBundles$.
A brief computation shows that
\begin{equation*}
   \ProjectiveSpaceOfNonDegenerateHarmonicOneForms \coloneq \HdR^*\ProjectiveSpaceOfNonDegeneratePositiveHarmonicSpinors
   \qandq
   \ProjectiveSpaceOfNonDegenerateHarmonicSelfDualTwoForms \coloneq 
   \HdR^*\ProjectiveSpaceOfNonDegenerateNegativeHarmonicSpinors
\end{equation*}
are the \defined{universal moduli spaces of non-degenerate $\Z/2\Z$ harmonic $1$--forms} and \defined{self-dual $2$--forms} respectively.
There are finite rank vector bundles $\sH^1 \to \SpaceOfMetrics\times\SpaceOfRamifiedLineBundles$ and $\sH^+ \to \SpaceOfMetrics\times\SpaceOfRamifiedLineBundles$ whose fibres over $(g,[\fl])$ are $\sH^1(g,\fl)$, as above, and $\sH^+(g,\fl)$, the space of $L^2$ harmonic self-dual $2$--forms $\alpha \in L^2\Omega^+(X,\fl)$.
As before,
$\sH^1(g,\fl) \iso \rH_{\dR}^1(\tilde X)^-$;
moreover,
$\sH^+(g,\fl)$ is isomorphic to $\rH_{\dR}^+(\tilde X)^-$,
a maximal subspace of $\rH_{\dR}^2(\tilde X)^-$ on which the intersection form is positive definite.
There are canonical inclusions
$\jmath \co \ProjectiveSpaceOfNonDegenerateHarmonicOneForms \incl \bP\sH^1$ and
$\jmath \co \ProjectiveSpaceOfNonDegenerateHarmonicSelfDualTwoForms \incl \bP\sH^+$.

\begin{theorem}
  \label{Thm_ProjectiveSpaceOfNonDegenerateHarmonicSelfDualTwoAndOneForms}
  If $\dim X=4$ and $X$ is oriented,
  then:
  \begin{enumerate}
  \item
    \label{Thm_ProjectiveSpaceOfNonDegenerateHarmonicSelfDualTwoAndOneForms_1}
    $\ProjectiveSpaceOfNonDegenerateHarmonicOneForms$ can be given the structure of a tame Fréchet manifold such that:
    \begin{enumerate}
    \item 
      $\pr_\SpaceOfMetrics \co \ProjectiveSpaceOfNonDegenerateHarmonicOneForms \to \SpaceOfMetrics$ is a uniform submersion, and
    \item
      $\jmath$ is smooth and induces an isomorphism $\VerticalTangentBundle{\ProjectiveSpaceOfNonDegenerateHarmonicOneForms}{\SpaceOfMetrics} \iso \jmath^* \VerticalTangentBundle{\bP\sH^1}{\paren{\SpaceOfMetrics \times \SpaceOfRamifiedLineBundles}}$.
    \end{enumerate}
    In particular, the index of $\pr_\SpaceOfMetrics$ at $(g,[\fl];[\alpha])$ agrees with $\dim \rH_{\dR}^1(\tilde X)^- - 1$.
  \item
    \label{Thm_ProjectiveSpaceOfNonDegenerateHarmonicSelfDualTwoAndOneForms_+}
    $\ProjectiveSpaceOfNonDegenerateHarmonicSelfDualTwoForms$ can be given the structure of a tame Fréchet manifold such that:
    \begin{enumerate}
    \item 
      $\pr_\SpaceOfMetrics \co \ProjectiveSpaceOfNonDegenerateHarmonicSelfDualTwoForms \to \SpaceOfMetrics$ is a uniform submersion, and
    \item
      $\jmath$ is smooth and induces an isomorphism $\VerticalTangentBundle{\ProjectiveSpaceOfNonDegenerateHarmonicSelfDualTwoForms}{\SpaceOfMetrics} \iso \jmath^* \VerticalTangentBundle{\bP\sH^+}{\paren{\SpaceOfMetrics \times \SpaceOfRamifiedLineBundles}}$.
    \end{enumerate}
    In particular, the index of $\pr_\SpaceOfMetrics$ at $(g,[\fl];[\alpha])$ agrees with $\dim \rH_{\dR}^+(\tilde X)^- - 1$.
  \end{enumerate}  
\end{theorem}

\begin{remark}
  The assertion about $\ProjectiveSpaceOfNonDegenerateHarmonicOneForms$ is again nothing but \cite[Theorem 1.1]{Donaldson2021} in dimension four.
  However,
  the assertion about $\ProjectiveSpaceOfNonDegenerateHarmonicSelfDualTwoForms$ is novel.
  In fact,
  for self-dual $2$--forms the reduction to a scalar equation used in \cite{Donaldson2021} is not available.
\end{remark}

\paragraph{Preliminaries}
This article is based on the theory of tame Fréchet manifolds as developed by \citet{Hamilton1982:NashMoser}.
Readers are assumed either to be familiar with this subject or to be willing to accept it as a black box.
It should be pointed out that, as a small deviation from \cite{Hamilton1982:NashMoser},
instead of \cite[Part II Definition 1.3.2]{Hamilton1982:NashMoser} the following definition is in effect.
\begin{definition}
  \label{Def_SmoothingOperatorProperty}
  A graded Fréchet space $\paren{F,\paren{\Abs{-}_k}_{k \in \N_0}}$ is \defined{tame} if it has the \defined{smoothing operator property};
  \label{It_SmoothingOperatorProperty}
  that is:
  there is a family of smoothing operators $\paren{ S_\epsilon : \epsilon \in (0,1) }$ on $F$ such that:
  \begin{enumerate}[label=\rm{(S\arabic*)},ref=S\arabic*]
  \item
    \label{It_S1}
    For every $\epsilon \in (0,1)$, $k,\ell \in \N_0$ with $k \geq \ell$, and $x \in F$
    \begin{equation*}
      \Abs{S_\epsilon x}_k \lesssim \epsilon^{\ell-k}\Abs{x}_\ell.
    \end{equation*}
  \item
    \label{It_S2}
    For every $\epsilon \in (0,1)$, $k,\ell \in \N_0$ with $k \leq \ell$, and $x \in F$
    \begin{equation*}
      \Abs{\paren{S_\epsilon-\one} x}_k \lesssim \epsilon^{\ell-k}\Abs{x}_\ell.
    \end{equation*}
  \item
    \label{It_S3}
    For every $\epsilon \in (0,1)$, $k,\ell \in \N_0$, and $x \in F$
    \begin{equation*}
      \epsilon\Abs{\del_\epsilon S_\epsilon x}_k \lesssim \epsilon^{\ell-k}\Abs{x}_\ell.
      \qedhere
    \end{equation*}
  \end{enumerate}
\end{definition}
This is a weaker condition than \cite[Part II Definition 1.3.2]{Hamilton1982:NashMoser},
but sufficient for the proof of the Nash--Moser inverse function theorem \cite[Part III Theorem 1.1.1]{Hamilton1982:NashMoser}.

\paragraph{Acknowledgements}
The authors thank
Fabian Lehmann for his explanations of \cite{DonaldsonLehmann2025:CY3Boundary},
Thibault Langlais for discussions concerning \cite{Teleman1983:SignatureLipschitz} and the proof of \autoref{Prop_KerDRD}, and
both Thibault Langlais and Jacek Rzemieniecki for numerous helpful comments and suggestions on an earlier version of this article.
%%%
This material is based in part upon work carried out while the authors were in residence at the Simons Laufer Mathematical Sciences Institute (previously known as MSRI) in Berkeley, California, during the Fall 2022 semester.
%%%
G.P. is partially supported by an NSF Mathematical Sciences Postdoctoral Research Fellowship (Award No.~2303102).

\paragraph{Statement on the use of AI}
The main mathematical ideas and arguments of this paper were developed by the authors without AI assistance.
The first arXiv version was also prepared without AI assistance.
This project began at MSRI in 2022 and has been developed over the past four years, and the manuscript was written by the authors.
In preparing the second arXiv version, the authors used GPT-6 Astra (OpenAI) to speed up the process of proving the index formula in \autoref{Sec_IndexFormula} and to proofread the entire article.
The authors take full responsibility for the mathematical correctness and final content of the article.

%%% Local Variables:
%%% mode: latex
%%% TeX-master: "UniversalModuliSpaceOfZ2ZHarmonicSpinors"
%%% ispell-local-dictionary: "british"
%%% End:

\section{Dirac operators twisted by ramified line bundles}
\label{Sec_DiracOperatorsTwistedByRamifiedLineBundles}

The technical foundation of this article is \cite{BeraWalpuski2025}.
The latter is, of course, inspired by \cite{BaerBallmann2012:BVP} and most of its results could plausibly be cobbled together from the powerful machinery developed, e.g., in \cite{Mazzeo1991,MazzeoVertman2014,AlbinGellRedman2016,AlbinGellRedman2023}.

For the readers' convenience this section provides an \emph{informal} overview of some of the material in \cite{BeraWalpuski2025}.
Readers who are familiar with it can skim this section or skip it entirely.
%%% 
Throughout this section,
the following data are assumed to be fixed:
\begin{enumerate}
\item
  a Riemannian metric $g$ on $X$;
\item
  a \defined{Dirac bundle structure} $(\gamma,\nabla)$ on $S$ with respect to $g$;
  that is:
  \begin{enumerate}
  \item
    a linear map $\gamma \co TX \to \fo(S)$, the \defined{Clifford multiplication}, such that for every $v \in TX$
    \begin{equation*}      
      \gamma(v)^2 = -\abs{v}^2 \one_S,
    \end{equation*}
    and
  \item
    \label{Def_DiracBundle_SpinConnection}
    a covariant derivative $\nabla \co \Gamma\paren{X,S} \to \Omega^1\paren{X,S}$,
    the \defined{spin connection},
    such that for every $v,w \in \Vect(X)$ and $\phi \in \Gamma\paren{X,S}$
    \begin{equation*}
      \nabla_v\paren{\gamma(w)\phi} = \gamma(\nabla_v w)\phi + \gamma(w)\nabla_v\phi;
    \end{equation*}
  \end{enumerate}
  and
\item    
  a \defined{ramified Euclidean line bundle} over $X$;
  that is:
  \begin{enumerate}
  \item
    a closed subset $Z \eqcolon \Br(\fl) \subseteq X$, the \defined{branching locus}, and
  \item
    a Euclidean line bundle $\fl$ over $X \setminus Z$
  \end{enumerate}
  such that
  \begin{enumerate}[resume]
  \item
    if $W \subseteq Z$ is closed and $\fl$ extends over $X \setminus W$,
    then $W = Z$.
  \end{enumerate}  
  Moreover,
  it is \emph{assumed} that $Z$ is a submanifold,
  necessarily of codimension two.
\end{enumerate}
Associated with these data is a \defined{Dirac operator}
\begin{equation*}
  D \co \Gamma\paren{X\setminus Z, S\otimes\fl} \to \Gamma\paren{X\setminus Z, S\otimes\fl}
\end{equation*}
defined by
\begin{equation*}
  D \phi \coloneq \sum_{i=1}^d \gamma(e_i) \nabla_{e_i} \phi.
\end{equation*}
Here $(e_1,\ldots,e_d)$ is an arbitrary local orthonormal frame and $\nabla$ is the twist of the spin connection by the unique orthogonal connection on $\fl$.

\subsection{The Gelfand--Robbin quotient}
\label{Sec_GelfandRobbinQuotient}

To begin putting the Dirac operator into a functional analytic context,
consider the \defined{minimal extension}
\begin{equation*}
  \Dmin = D \co \dom(\Dmin) \coloneq H^1\Gamma\paren{X\setminus Z, S\otimes \fl} \to L^2\Gamma\paren{X\setminus Z, S \otimes \fl}.
\end{equation*}
$\Dmin$ is left semi-Fredholm but its cokernel is $\infty$--dimensional \cite[Proposition 2.3 and Lemma 3.3]{BeraWalpuski2025}.
Its adjoint, in the sense of unbounded operators, is the \defined{maximal extension}
\begin{equation*}
  \Dmax = \Dmin^* \co \dom(\Dmax) \to L^2\Gamma\paren{X\setminus Z, S\otimes \fl};
\end{equation*}
in fact, by \cite[Remark 2.10]{BeraWalpuski2025},
\begin{equation*}
  \dom(\Dmax)
  =
  \set[\big]{
    \phi \in H_\loc^1\Gamma\paren{X\setminus Z, S\otimes\fl}
    :
    \phi,D\phi \in L^2\Gamma\paren{X\setminus Z, S\otimes \fl}
  }.
\end{equation*}
Of course, $\Dmax$ is right semi-Fredholm but its kernel is $\infty$--dimensional.

The intermediate extensions of $\Dmin$ are controlled by the \defined{Gelfand--Robbin quotient}
\begin{equation*}
  \GelfandRobbinQuotient
  \coloneq
  \frac{\dom(\Dmax)}{\dom(\Dmin)}.
\end{equation*}
This is a symplectic Hilbert space;
its symplectic form, the \defined{Green's} form, is defined by
\begin{equation*}
  G([\phi],[\psi]) \coloneq \Inner{D\phi,\psi}_{L^2} - \Inner{\phi,D\psi}_{L^2}.
\end{equation*}
Every closed extension of $\Dmin$ is of the form
\begin{equation*}
  D_R \coloneq \Dmax|_{\dom(D_R)}
  \qwithq
  \dom(D_R) = \set{ \phi \in \dom(\Dmax) : [\phi] \in R };
\end{equation*}
where $R \subseteq \GelfandRobbinQuotient$ is a closed subspace \cite[Proposition 2.14]{BeraWalpuski2025};
moreover,
\begin{equation*}
  D_R^* = D_{R^G};
\end{equation*}
that is: the adjoint of $D_R$ is obtained by forming the symplectic complement $R^G$ of $R$ \cite[Proposition 2.17]{BeraWalpuski2025}.
As in \cite[Definition 2.13]{BeraWalpuski2025},
a \defined{residue condition} is a closed subspace $R \subseteq \GelfandRobbinQuotient$.

%%% Local Variables:
%%% mode: latex
%%% TeX-master: "UniversalModuliSpaceOfZ2ZHarmonicSpinors"
%%% ispell-local-dictionary: "british"
%%% End:

\subsection{The residue map}
\label{Sec_ResidueMap}

In an abstract way,
$[\phi] \in \GelfandRobbinQuotient$ encodes the leading-order behaviour of $\phi \in \dom(\Dmax)$ near $Z$.
This leads to the following more geometric realisation of the Gelfand--Robbin quotient.

It is instructive to begin by considering an untwisted spinor $\phi \in \Gamma\paren{X,S}$ vanishing along $Z$.
The Taylor expansion of $\phi$ at $Z$ has a leading-order term $\phi^{(1)} \in \Gamma\paren{Z,\Hom\paren{NZ,S|_Z}}$.
A brief computation shows that if $D\phi = 0$, then $\phi^{(1)}$ satisfies the algebraic constraint
\begin{equation*}
  \phi^{(1)} \circ \Or = - \gamma(\Or) \circ \phi^{(1)}
\end{equation*}
for every $x \in Z$ and $\Or \in \Wedge^2 N_xZ$.
Here $\Wedge^2 N_xZ$ is identified with $\fo(N_xZ)$ using the Riemannian metric $g$ so that $(u \wedge v)(w) = \Inner{u,w}v - \Inner{v,w}u$ as in \cite[Part I §6 (6.3)]{Lawson1989}.

If $Z$ is cooriented,
then there is a preferred $\Or \in \Wedge^2 NZ$ of unit length which defines complex structures on $S|_Z$ and $NZ$.
The algebraic constraint is then simply that $\phi^{(1)}$ is complex anti-linear.
However,
$Z$ might not be coorientable.
Nevertheless,
the Riemannian metric defines an isomorphism $(\Wedge^2NZ)^{\otimes 2} \iso \R$ and this produces a flat bundle
\begin{equation*}
  \AlgebraBundle = \R \oplus i\Wedge^2NZ
\end{equation*}
of normed $\R$--algebras over $Z$ whose fibres are isomorphic to $\C$,
but canonically only up to complex conjugation.
The preceding discussion exhibits both $NZ$ and $S|_Z$ as $\AlgebraBundle$--modules.
If $D\phi = 0$, then $\phi^{(1)}$ is $\AlgebraBundle$--anti-linear.

The analogous construction for $\phi \in \dom(\Dmax)$ is quite a bit more involved.
Its outcome is very briefly stated as follows.
The Dirac bundle structure on $S$ restricts to a Dirac bundle structure on $S|_Z$ \cite[Definition 3.9]{BeraWalpuski2025}.
As explained in \cite[Definition 3.21 and Remark 3.22]{BeraWalpuski2025},
the Euclidean line bundle $\fl$ defines an $\AlgebraBundle$--module $NZ^{-1/2}$ together with an $\AlgebraBundle$--linear orthogonal connection, and an isomorphism
\begin{equation*}
  NZ^{-1/2} \otimes_\AlgebraBundle NZ^{-1/2} \iso \Hom_\AlgebraBundle(NZ,\AlgebraBundle).
\end{equation*}
This endows the \defined{residue bundle}
\begin{equation*}
  \ResidueBundle \coloneq S|_Z \otimes_\AlgebraBundle NZ^{-1/2}
\end{equation*}
with a Dirac bundle structure;
in particular,
$\ResidueBundle$ has a Dirac operator $D_{\ResidueBundle}$.
Moreover, the restriction of the Clifford multiplication to $NZ$ induces a $\AlgebraBundle$--anti-linear almost complex structure
\begin{equation*}
  J \in \Gamma\paren{Z,\overline{\End}_{\AlgebraBundle}(\ResidueBundle)}.
\end{equation*}
This together with the Euclidean inner product induces a symplectic structure $\check\Omega$ on $\ResidueBundle$.
Roughly speaking, after identifying $N_xZ = \C$, an element of $S|_Z \otimes_\AlgebraBundle NZ^{-1/2}$ can be viewed as a term of the form $\phi \bar z^{-1/2}$.
This should make it plausible that the Gelfand--Robbin quotient $\GelfandRobbinQuotient$ can be realised as a suitable completion of $\Gamma\paren{Z,\ResidueBundle}$.

The norm with respect to which $\Gamma\paren{Z,\ResidueBundle}$ has to be completed is defined in terms of the \defined{branching locus operator}
\begin{equation*}
  A \coloneq -JD_{\ResidueBundle}.
\end{equation*}
Since $J$ is skew-adjoint, $D_{\ResidueBundle}$ is (formally) self-adjoint, and $J$ and $D_{\ResidueBundle}$ anti-commute,
$A$ is (formally) self-adjoint.
Denote by $\one_{(-\infty,0)}(A)$ and $\one_{[0,\infty)}(A)$ the orthogonal projections to the negative and non-negative eigenspaces of $A$ respectively.
Define the norm $\Abs{-}_{\check H} \co \Gamma\paren{Z,\ResidueBundle} \to [0,\infty)$ by
\begin{equation*}
  \Abs{\phi}_{\check H}
  \coloneq
  \Abs{\one_{(-\infty,0)}(A)\phi}_{H^{1/2}}
  +
  \Abs{\one_{[0,\infty)}(A)\phi}_{H^{-1/2}},
\end{equation*}
and denote by $\check H\Gamma\paren{Z,\ResidueBundle}$ the completion of $\Gamma\paren{Z,\ResidueBundle}$ with respect to $\Abs{-}_{\check H}$.
The symplectic form $\check\Omega$ on $\ResidueBundle$ induces a symplectic form on $\check H\Gamma\paren{Z,\ResidueBundle}$ which shall also be denoted by $\check \Omega$.

\cite[Definition 3.40 (1)]{BeraWalpuski2025} constructs the \defined{residue map}
\begin{equation*}
  \res \co \dom(\Dmax) \to \check H\Gamma\paren{Z,\ResidueBundle}
\end{equation*}
which satisfies
\begin{equation*}
  \ker \res = \dom(\Dmin)
\end{equation*}
and is surjective;
in fact, \cite[Definition 3.40 (2)]{BeraWalpuski2025} constructs a right inverse,
the \defined{extension map}
\begin{equation*}
  \ext \co \check H\Gamma\paren{Z,\ResidueBundle} \to \dom(\Dmax).
\end{equation*}
As a consequence of this,
$\res$ induces an isomorphism
\begin{equation*}
  \GelfandRobbinQuotient \iso \check H\Gamma\paren{Z,\ResidueBundle};
\end{equation*}
moreover, through this isomorphism the Green's form corresponds to $\check\Omega$.
In combination with the discussion in \autoref{Sec_GelfandRobbinQuotient},
this shows that every closed extension of $\Dmin$ corresponds to a residue condition considered as a closed subspace $R \subseteq \check H\Gamma\paren{Z,\ResidueBundle}$.

Although the residue map is constructed analytically in \cite[§3]{BeraWalpuski2025},
it can be seen to arise from a suitable restriction map as follows.
\emph{Choose} a tubular neighbourhood of $Z$ and use it to construct a blow-up $\beta \co \hat X \to X$ of $X$ along $Z$.
The restriction of $\beta$ to the boundary $\del \hat X$ is identified with the frame bundle
\begin{equation*}
  \pi \co F \coloneq \set{ v \in NZ : \abs{v} = 1 } \to Z.
\end{equation*}
Set $\hat S \coloneq \beta^*S$.
Since the inclusion $X\setminus Z \subseteq \hat X$ is a homotopy equivalence,
$\fl$ extends to a Euclidean line bundle $\hat \fl$ over $\hat X$.
\cite[Proposition 3.23]{BeraWalpuski2025} constructs an isomorphism
\begin{equation*}
  \pi^*NZ^{-1/2} \iso \hat \fl|_{\del \hat X} \otimes \pi^*\AlgebraBundle.
\end{equation*}
In particular, this defines an inclusion
\begin{equation*}
  \pi^* \co \Gamma\paren{Z,\ResidueBundle} \incl \Gamma\paren{\del \hat X,\hat S \otimes \hat \fl}
\end{equation*}
whose image should be understood to correspond to the fibrewise $-\frac12$ Fourier modes.
If
\begin{equation*}
  r \co X \setminus Z \to (0,\infty)
\end{equation*}
denotes a \defined{regularised distance} to $Z$,
that is: $r$ is smooth, positive, and agrees with the distance to $Z$ on a neighbourhood of $Z$,
then it turns out that $\dom(\Dmax) \cap r^{-1/2}\Gamma\paren{\hat X,\hat S \otimes \hat\fl}$ is dense in $\dom(\Dmax)$ and that for every $r^{-1/2}\phi \in \dom(\Dmax) \cap r^{-1/2}\Gamma\paren{\hat X,\hat S \otimes \hat\fl}$
\begin{equation*}
  \phi|_{\del \hat X} = \pi^*\res\paren{r^{-1/2}\phi}.
\end{equation*}

%%% Local Variables:
%%% mode: latex
%%% TeX-master: "UniversalModuliSpaceOfZ2ZHarmonicSpinors"
%%% ispell-local-dictionary: "british"
%%% End:

\subsection{Conormal and adapted Sobolev spaces}
\label{Sec_ConormalAndAdaptedSobolevSpaces}

A vector field $v \in \Vect(\hat X)$ is \defined{conormal} if $v|_{\del \hat X} \in \Vect(\del \hat X)$.
Denote the subspace of conormal vector fields by $\Vect_b(\hat X)$.
Denote by $\DiffOp^\bullet(S\otimes\fl)$ the $\N_0$--filtered ring of differential operators acting on $S\otimes\fl$.
The space $\DiffOp_b^\bullet(S \otimes \fl) \subseteq \DiffOp^\bullet(S \otimes \fl)$ of \defined{conormal}
differential operators is the filtered subring generated by $\Gamma\paren{\hat X,\End(\hat S \otimes \hat\fl)}$ and differential operators of the form $\nabla_v$ with $v \in \Vect_b(\hat X)$.

For every $k \in \N_0$ the \defined{conormal Sobolev space} $H_b^k\Gamma\paren{X\setminus Z, S\otimes\fl}$ is defined by
\begin{equation*}
  H_b^k\Gamma\paren{X\setminus Z, S\otimes\fl}
  \coloneq
  \set*{
    \phi \in H_\loc^k\Gamma\paren{X\setminus Z, S\otimes\fl}
    :
    \begin{aligned}
      & P \phi \in L^2\Gamma\paren{X\setminus Z, S\otimes\fl} \text{ for} \\
      & \text{every }
        P \in \DiffOp_b^k(S\otimes\fl)
    \end{aligned}
  }.
\end{equation*}
\emph{Choose} a finite subset $\sP_b^k \subseteq \DiffOp_b^k(S\otimes\fl)$ which spans $\DiffOp_b^k(S\otimes\fl)$ over $\Gamma\paren{\hat X,\End(\hat S\otimes\hat\fl)}$.
Define the norm $\Abs{-}_{H_b^k} \co  H_b^k\Gamma\paren{X\setminus Z, S\otimes\fl} \to [0,\infty)$ by
\begin{equation*}
  \Abs{\phi}_{H_b^k}^2
  \coloneq
  \sum_{P \in \sP_b^k} \Abs{P\phi}_{L^2}^2.
\end{equation*}
Of course, these are the norms that appear in \citeauthor{Melrose1991:Kyoto}'s $b$-calculus;
see, e.g., \cite{Melrose1991:Kyoto,Melrose1993}.
They should be regarded as higher regularity analogues of $L^2\Gamma\paren{X\setminus Z,S\otimes \fl}$.
The core of the scale of conormal Sobolev spaces $\paren{H_b^k\Gamma\paren{X\setminus Z,S\otimes \fl},\paren{\Abs{-}_{H_b^k}}_{k \in \N_0}}$ is a tame Fréchet space \cite[Proposition 4.32]{BeraWalpuski2025}.

For every $k \in \N_0$ the \defined{adapted Sobolev space} $H_a^{k+1}\Gamma\paren{X\setminus Z, S\otimes\fl}$ is defined by
\begin{equation*}
  H_a^{k+1}\Gamma\paren{X\setminus Z, S\otimes\fl}
  \coloneq
  \set*{
    \phi \in H_\loc^{k+1}\Gamma\paren{X\setminus Z, S\otimes\fl}
    :
    \begin{aligned}
      &\phi \in H_b^{k+1}\Gamma\paren{X\setminus Z, S\otimes\fl} ~\text{and} \\
      &D\phi \in H_b^k\Gamma\paren{X\setminus Z, S\otimes\fl} 
    \end{aligned}
  }.
\end{equation*}
Define the norm $\Abs{-}_{H_a^{k+1}} \co  H_a^{k+1}\Gamma\paren{X\setminus Z, S\otimes\fl} \to [0,\infty)$ by
\begin{equation*}
  \Abs{\phi}_{H_a^{k+1}}
  \coloneq
  \Abs{\phi}_{H_b^{k+1}} + \Abs{D\phi}_{H_b^k}.
\end{equation*}
These should be regarded as higher regularity analogues of $\dom(\Dmax)$.
Again, the core of the scale of adapted Sobolev spaces $\paren{H_a^{k+1}\Gamma\paren{X\setminus Z,S\otimes \fl},\paren{\Abs{-}_{H_a^{k+1}}}_{k \in \N_0}}$ is a tame Fréchet space \cite[Proposition 4.37]{BeraWalpuski2025}.

By \cite[Lemma 4.12 and Lemma 4.13]{BeraWalpuski2025},
for every $k \in \N_0$
the residue map $\res \co \dom(\Dmax) \to \check H\Gamma\paren{Z,\ResidueBundle}$ restricts to a bounded surjective linear map
\begin{equation*}
  \res \co H_a^{k+1}\Gamma\paren{X\setminus Z,S\otimes\fl} \to H^{k+1/2}\Gamma\paren{Z,\ResidueBundle}.
\end{equation*}
The extension map $\ext \co \check H\Gamma\paren{Z,\ResidueBundle} \to \dom(\Dmax)$ restricts to a right-inverse
\begin{equation*}
  \ext \co H^{k+1/2}\Gamma\paren{Z,\ResidueBundle} \to H_a^{k+1}\Gamma\paren{X\setminus Z,S\otimes\fl}.
\end{equation*}
Moreover,
if $\phi \in H_a^{k+1}\Gamma\paren{X\setminus Z,S\otimes\fl}$ satisfies $\res(\phi) = 0$,
then
\begin{equation*}
  \nabla_v\phi, r^{-1}\phi \in H_b^k\Gamma\paren{X\setminus Z,S\otimes\fl}
\end{equation*}
for every $v \in \Vect(X)$ \cite[Proposition 4.27]{BeraWalpuski2025}.
Therefore,
$\ker \res \subseteq H_a^{k+1}\Gamma\paren{X\setminus Z,S\otimes\fl}$ should be regarded as a higher regularity analogue of $\dom(\Dmin) = H^1\Gamma\paren{X\setminus Z,S\otimes \fl}$.

%%% Local Variables:
%%% mode: latex
%%% TeX-master: "UniversalModuliSpaceOfZ2ZHarmonicSpinors"
%%% ispell-local-dictionary: "british"
%%% End:

\subsection{Elliptic residue conditions}
\label{Sec_ElliticResidueConditions}

For every residue condition $R \subseteq \check H\Gamma\paren{Z,\ResidueBundle}$ and every $k \in \N_0$ set
\begin{equation*}
  H_a^{k+1}\Gamma\paren{X\setminus Z,S\otimes\fl;R} \coloneq H_a^{k+1}\Gamma\paren{X\setminus Z,S\otimes\fl}
  \cap \dom(D_R)
\end{equation*}
and consider the restriction of $D_R$ to
\begin{equation*}
  D_R^{(k)} \co H_a^{k+1}\Gamma\paren{X\setminus Z,S\otimes\fl;R} \to H_b^k\Gamma\paren{X\setminus Z,S\otimes\fl}.
\end{equation*}

A residue condition $R \subseteq \check H\Gamma\paren{Z,\ResidueBundle}$ is \defined{$\infty$--regular} if for every $k \in \N_0$ and every $\phi \in R$
\begin{equation*}
  \Abs{\phi}_{H^{k+1/2}}
  \lesssim_{R,k}
  \Abs{\one_{(-\infty,0)}(A)\phi}_{H^{k+1/2}}
  +
  \Abs{\phi}_{\check H}.
\end{equation*}
It is \defined{elliptic} if both $R$ and $R^G$ are $\infty$--regular.
In this case, $\smash{D_R^{(k)}}$ and $D_{R^G}^{(k)}$ satisfy elliptic estimates and elliptic regularity \cite[Theorem 4.17]{BeraWalpuski2025};
in particular,
$D_R^{(k)}$ is Fredholm, and
$\smash{\ker D_R^{(k)}}$ and $\smash{\coker D_R^{(k)}}$ are independent of $k \in \N_0$ \cite[Theorem 4.17 and Proposition 4.21]{BeraWalpuski2025}.

Evidently,
the \defined{APS residue condition} defined by
\begin{equation*}
  R_\APS
  \coloneq
  \one_{(-\infty,0)}(A)H^{1/2}\Gamma\paren{Z,\ResidueBundle}
  \subset
  \check H\Gamma\paren{Z,\ResidueBundle}
\end{equation*}
is elliptic.
For a \defined{local residue condition}, that is: a residue condition of the form
\begin{equation*}
  R_V
  \coloneq
  \check H\Gamma(Z,V) \subseteq \check H\Gamma\paren{Z,\ResidueBundle}
\end{equation*}
for a subbundle $V \subseteq \ResidueBundle$,
there is the following symbolic criterion for ellipticity \cite[Theorem 4.40]{BeraWalpuski2025}:
$R_V$ is elliptic if for every $v \in TZ \setminus \set{0}$
\begin{equation*}
  V \oplus J\gamma(v) V = \ResidueBundle.
\end{equation*}
If $\rk_\AlgebraBundle \ResidueBundle = 2$ and $V \subseteq \ResidueBundle$ is an $\AlgebraBundle$--linear subbundle with $\rk_\AlgebraBundle V = 1$,
then $R_V$ satisfies this criterion because $J\gamma(v)$ is skew-adjoint and $\AlgebraBundle$--anti-linear;
in particular, $R_V$ is elliptic \cite[Example 4.44]{BeraWalpuski2025}.

%%% Local Variables:
%%% mode: latex
%%% TeX-master: "UniversalModuliSpaceOfZ2ZHarmonicSpinors"
%%% ispell-local-dictionary: "british"
%%% End:

\subsection{Once more, with chirality}
\label{Sec_OnceMoreWithChirality}

Suppose that $(S,\gamma,\nabla)$ is equipped with a \defined{chirality operator};
that is: a self-adjoint parallel isometry $\epsilon \in \Gamma\paren{X,\O(S)}$ such that
\begin{equation*}
  \gamma\epsilon + \epsilon\gamma = 0.
\end{equation*}
This induces an orthogonal decomposition of $S$ into $\nabla$--parallel subbundles
\begin{equation*}
  S = S^+ \oplus S^-
  \qwithq
  S^\pm \coloneq \ker\paren{\epsilon \mp \one}.
\end{equation*}
Consequently, the Dirac operator $D$ decomposes as
\begin{equation*}
  D
  =
  \begin{pmatrix}
    0 & D^- \\
    D^+ & 0
  \end{pmatrix}
\end{equation*}
with $D^\pm \co \Gamma\paren{X\setminus Z, S^\pm \otimes\fl} \to \Gamma\paren{X\setminus Z, S^\mp \otimes\fl}$ denoting the \defined{chiral Dirac operators}.

The minimal and maximal extensions decompose analogously.
Therefore,
the Gelfand--Robbin quotient $\GelfandRobbinQuotient$ orthogonally decomposes as
\begin{equation*}
  \GelfandRobbinQuotient = \GelfandRobbinQuotient^+ \oplus \GelfandRobbinQuotient^-
  \qwithq
  \GelfandRobbinQuotient^\pm \coloneq \frac{\dom(\Dmax^\pm)}{\dom(\Dmin^\pm)};
\end{equation*}
moreover, $\GelfandRobbinQuotient^\pm \subseteq \GelfandRobbinQuotient$ are Lagrangian.
A \defined{chiral residue condition} is a closed subspace $\ResidueCondition^\pm \subseteq \GelfandRobbinQuotient^\pm$.
For every chiral residue condition $\ResidueCondition^\pm \subseteq \GelfandRobbinQuotient^\pm$ there is a unique complementary chiral residue condition $\ResidueCondition^\mp \subseteq \GelfandRobbinQuotient^\mp$ such that $\ResidueCondition \coloneq \ResidueCondition^+ \oplus \ResidueCondition^- \subseteq \GelfandRobbinQuotient$ is a Lagrangian residue condition.
As before, the chiral residue conditions $R^\pm \subseteq \GelfandRobbinQuotient^\pm$ precisely correspond to the closed extensions $D_{R^\pm}^{\pm}$ of $\Dmin^\pm$.

The chirality operator also induces a parallel orthogonal decomposition $\ResidueBundle = \ResidueBundle^+ \oplus \ResidueBundle^-$.
With respect to it the branching locus operator $A$ decomposes into $A^\pm \coloneq -JD_{\ResidueBundle}^\pm \co \Gamma\paren{Z,\ResidueBundle^\pm} \to \Gamma\paren{Z,\ResidueBundle^\pm}$;
consequently,
\begin{equation*}
  \check H\Gamma\paren{Z,\ResidueBundle}
  =
  \check H\Gamma\paren{Z,\ResidueBundle^+} \oplus \check H\Gamma\paren{Z,\ResidueBundle^-}.
\end{equation*}
The residue map $\res$ correspondingly decomposes into $\res^\pm$, and these induce isomorphisms
\begin{equation*}
  \GelfandRobbinQuotient^\pm \iso \check H\Gamma\paren{Z,\ResidueBundle^\pm}.
\end{equation*}

Of course,
the conormal and adapted Sobolev spaces also decompose according to the chirality operator.
In particular, for every chiral residue condition $R^\pm \subseteq \check H\Gamma\paren{Z,\ResidueBundle^\pm}$ and $k \in \N_0$ there are the following higher regularity versions of $D_{R^\pm}^\pm$:
\begin{equation*}
  D_{R^\pm}^{\pm,(k)} \co H_a^{k+1}\Gamma\paren{X\setminus Z,S^\pm\otimes\fl;R^\pm} \to H_b^k\Gamma\paren{X\setminus Z,S^\mp\otimes\fl}.
\end{equation*}
If the Lagrangian residue condition $R = R^+ \oplus R^-$ is elliptic,
then both operators $D_{R^\pm}^{\pm,(k)}$ satisfy elliptic estimates and elliptic regularity;
in particular, each is Fredholm.

A \defined{local chiral residue condition} is a chiral residue condition of the form
\begin{equation*}
  R_{V^\pm}^\pm
  \coloneq
  \check H\Gamma(Z,V^\pm) \subseteq \check H\Gamma\paren{Z,\ResidueBundle^\pm}
\end{equation*}
for some subbundle $V^\pm \subseteq \ResidueBundle^\pm$.
For every subbundle $V^\pm \subseteq \ResidueBundle^\pm$ there is a unique subbundle $V^\mp \subseteq \ResidueBundle^\mp$ such that $V^+ \oplus V^- \subseteq \ResidueBundle = \ResidueBundle^+ \oplus \ResidueBundle^-$ is Lagrangian.
Of course, $V^\mp$ induces the complementary chiral residue condition:
$R_V \coloneq R_{V^+}^+ \oplus R_{V^-}^-$ is Lagrangian.
A moment's thought shows that $R_V$ is elliptic if $V^\pm$ satisfies the symbolic criterion,
namely for every $v \in TZ\setminus\set{0}$
\begin{equation*}
  V^\pm \oplus J\gamma(v)V^\pm = \ResidueBundle^\pm.
\end{equation*}
This holds in particular if $\rk_\AlgebraBundle \ResidueBundle^\pm = 2$ and $V^\pm \subseteq \ResidueBundle^\pm$ is an $\AlgebraBundle$--linear subbundle with $\rk_\AlgebraBundle V^\pm = 1$.

%%% Local Variables:
%%% mode: latex
%%% TeX-master: "UniversalModuliSpaceOfZ2ZHarmonicSpinors"
%%% End:

%%% Local Variables:
%%% mode: latex
%%% TeX-master: "UniversalModuliSpaceOfZ2ZHarmonicSpinors"
%%% ispell-local-dictionary: "british"
%%% End:

\section{The universal Dirac operator and residue map}
\label{Sec_UniversalDiracOperatorAndResidueMap}

This section sets up the family version of the theory developed in \cite{BeraWalpuski2025} and surveyed in \autoref{Sec_DiracOperatorsTwistedByRamifiedLineBundles}.
After describing
$\SpaceOfDiracBundles$,
the space of Dirac bundle structures on $S$,
and $\SpaceOfRamifiedLineBundles$,
the moduli space of ramified line bundles (with smooth branching locus) over $X$,
the spaces
$H_a^\infty\Gamma\paren{X\setminus \Br(\fl),S\otimes\fl}$,
$H_b^\infty\Gamma\paren{X\setminus \Br(\fl),S\otimes\fl}$, and
$\Gamma\paren{\Br(\fl),\ResidueBundle}$
are assembled into
tame Fréchet space bundles
$\BundleOfAdmissibleSingularSpinors$,
$\BundleOfSingularSpinors$, and
$\BundleOfResidues$ over $\SpaceOfDiracBundles \times \SpaceOfRamifiedLineBundles$.
The Dirac operators and residue maps assemble into the universal Dirac operator $\UniversalDiracOperator \co \BundleOfAdmissibleSingularSpinors \to \BundleOfSingularSpinors$ and the universal residue map $\UniversalResidueMap \co \BundleOfAdmissibleSingularSpinors \to \BundleOfResidues$.
These bundles are equipped with partial covariant derivatives, and the corresponding infinitesimal variation of $\UniversalDiracOperator$ is determined.
Within this framework it is possible to consider residue conditions in families.
Finally, it is explained what additional structure arises from the chirality operator $\epsilon \coloneq \gamma(\vol_g)$ if $X$ is oriented and $\dim X = 0 \pmod 4$.

\subsection{The space of Dirac bundles}
\label{Sec_UniversalSpaceOfDiracBundles}

For the purposes of this article,
the following is the relevant parameter space and half of the base of the tame Fréchet space bundles $\BundleOfAdmissibleSingularSpinors$, $\BundleOfSingularSpinors$, and $\BundleOfResidues$.

\begin{definition}
  \label{Def_UniversalSpaceOfDiracBundles}
  Denote by $\SpaceOfMetrics \subseteq \Gamma\paren{X,\Hom(S^2TX,\R)}$ the space of Riemannian metrics on $X$.
  The \defined{space of Dirac bundle structures} on $S$ is the tame Fréchet submanifold
  \begin{equation*}
    \SpaceOfDiracBundles
    \coloneq
    \set*{
      (g;\gamma,\nabla) \in \SpaceOfMetrics \times \Gamma\paren{X,\Hom(TX,\fo(S))} \times \sA(S)
      :
      \begin{aligned}
        & (\gamma,\nabla)
        ~\text{is a Dirac bundle structure} \\
        & 
        ~\text{with respect to}~
        g              
      \end{aligned}
    }.
  \end{equation*}
  In the following a triple $(g;\gamma,\nabla) \in \SpaceOfDiracBundles$ is often abbreviated to $\bp$
  (which, of course, stands for parameter).
  \qedhere
\end{definition}

The projection map $p \co \SpaceOfDiracBundles \to \SpaceOfMetrics$ exhibits $\SpaceOfDiracBundles$ as a tame smooth Fréchet fibre bundle over $\SpaceOfMetrics$.
\citet{Bourguignon1992} construct an Ehresmann connection on $\SpaceOfDiracBundles$;
see also
\cites[§3.2]{AmmannWeissWitt2016}[Definition 2.9]{MullerNowaczyk2017}[§2.4,§2.5]{AmmannDahl2025}.
Moreover, the \defined{vertical tangent bundle}
\begin{equation*}
  \VerticalTangentBundle{\SpaceOfDiracBundles}{\SpaceOfMetrics}
  \coloneq
  \ker \paren*{Tp \co T\SpaceOfDiracBundles \to T\SpaceOfMetrics}
\end{equation*}
further decomposes as follows.

\begin{prop}
  \label{Prop_UniversalSpaceOfDiracBundles_TangentBundle}
  For every $\bp = (g;\gamma,\nabla) \in \SpaceOfDiracBundles$ the tangent space $T_\bp\SpaceOfDiracBundles$ decomposes into the tame Fréchet subspaces
  \begin{equation*}
    T_\bp\SpaceOfDiracBundles = H_{\mathrm{BG},\bp} \oplus V_{\sG,\bp} \oplus V_{\nabla,\Uh;\bp}
  \end{equation*}
  defined as follows:  
  \begin{enumerate}
  \item
    \label{Prop_UniversalSpaceOfDiracBundles_TangentBundle_HBG}
    $H_{\BG,\bp}$ is the horizontal distribution at $\bp \in \SpaceOfDiracBundles$ of the \defined{Bourguignon--Gauduchon connection} \cite{Bourguignon1992};
    that is:
    the unique Ehresmann connection on $\SpaceOfDiracBundles \to \SpaceOfMetrics$ which horizontally lifts $\dot g \in T_g\SpaceOfMetrics$ to $(\dot g;\dot\gamma^\BG,\dot\nabla^\BG) \in T_\bp\SpaceOfDiracBundles$ defined by
    \begin{equation*}
      \dot\gamma^\BG
      \coloneq
      \frac12 \sum_{i=1}^d \dot g(-,e_i) \gamma(e_i)
      %%%
      \qandq
      %%%
      \dot\nabla^\BG      
      \coloneq \frac18 \sum_{i,j=1}^d \paren{\nabla_{e_i} \dot g}(e_j,-) [\gamma(e_i),\gamma(e_j)].
    \end{equation*}
    Here and throughout $(e_1,\ldots,e_d)$ denotes a local $g$--orthonormal frame.
  \item
    \label{Prop_UniversalSpaceOfDiracBundles_TangentBundle_VG}
    $V_{\sG,\bp}$ is generated by the action of the \defined{gauge group} $\sG(S) \coloneq \Gamma\paren{X,\O(S)}$ on $\SpaceOfDiracBundles$;
    that is:
    \begin{equation*}
      V_{\sG,\bp}
      \coloneq
      \set[\big]{
        (0;[\xi,\gamma],-\nabla\xi)
        :
        \xi \in \Gamma\paren{X,\fo(S)}
      }.
    \end{equation*}    
  \item
    \label{Prop_UniversalSpaceOfDiracBundles_TangentBundle_VNabla}
    $V_{\nabla,\Uh;\bp}$ consists of infinitesimal variations of the spin connection in Uhlenbeck gauge;
    that is:
    \begin{equation*}
      V_{\nabla,\Uh;\bp}
      \coloneqq
      \set[\big]{
        (0;0,\dot\nabla)
        :
        \dot\nabla \in \Omega^1\paren{X,\fo_{\Cl}(S)} ~\textnormal{satisfying}~
        \rd_\nabla^*\dot\nabla = 0
      }.
    \end{equation*}
    Here $\fo_\Cl(S)$ denotes the bundle of skew-adjoint endomorphisms of $S$ which commute with the Clifford multiplication $\gamma$.
  \end{enumerate}
  In fact,
  the vertical tangent space $T_\bp{\SpaceOfDiracBundles}/{\SpaceOfMetrics}$ contains every infinitesimal variation of the spin connection;
  that is:
  \begin{equation*}
    V_{\nabla,\bp}
    \coloneq
    \set*{
      (0;0,\dot\nabla) : \dot\nabla \in \Omega^1\paren{X,\fo_{\Cl}(S)}
    }
    \subseteq V_{\sG,\bp} \oplus V_{\nabla,\Uh;\bp}.
    \qedhere
  \end{equation*}
\end{prop}

The above turns out to be useful in the following to organise various computations.

\begin{proof}[Proof of \autoref{Prop_UniversalSpaceOfDiracBundles_TangentBundle}]
  To prove that $H_{\mathrm{BG}}$ is an Ehresmann connection on $\SpaceOfDiracBundles \to \SpaceOfMetrics$,
  it suffices to verify the formulae
  \begin{align}
    \label{Eq_DotGamma}
    \dot\gamma^\BG(v)\gamma(v) + \gamma(v)\dot\gamma^\BG(v)
    &=
      -\dot g(v,v) \qandq \\
      %%% 
    \label{Eq_DotNabla}
    [\dot\nabla^\BG_v,\gamma(w)]
    &=
      \dot\gamma^\BG(\nabla_v w) + \gamma(\dot\nabla_v w) -[\nabla_v,\dot\gamma^\BG(w)]
  \end{align}
  with $\dot\nabla$ denoting the infinitesimal deformation of the Levi--Civita connection corresponding to $\dot g \in T_g\SpaceOfMetrics$.
  %%%
  The verification of \autoref{Eq_DotGamma} is straightforward from the formula in \autoref{Prop_UniversalSpaceOfDiracBundles_TangentBundle_HBG}.
  %%%
  By differentiating Koszul's formula,
  \begin{equation*}
    2g(\dot\nabla_v w,u)
    =
    (\nabla_v \dot g)(w,u) + (\nabla_w \dot g)(u,v) - (\nabla_u \dot g)(v,w).
  \end{equation*}
  Therefore,
  \begin{align*}
    \gamma(\dot \nabla_v w)
    =
    \frac12  \sum_{j=1}^d \paren[\big]
    {(\nabla_v \dot g)(w,e_j) + (\nabla_w \dot g)(v,e_j) - (\nabla_{e_j} \dot g)(v,w)} \gamma(e_j).
  \end{align*}
  %%%
  A direct computation
  using the identity
  \begin{equation}
    \label{Eq_DoubleCommutator}
    [[\gamma(u),\gamma(v)],\gamma(w)] =  4\paren{\gamma(v)g(u,w)-\gamma(u)g(v,w)}
  \end{equation}
  yields
  \begin{align*}
    [\dot\nabla^\BG_v,\gamma(w)]
    &=
      \frac12 \sum_{j=1}^d \paren[\big]{(\nabla_w \dot g)(e_j,v) - (\nabla_{e_j} \dot g)(v,w)} \gamma(e_j), \\
    %%%
    \dot\gamma^\BG(\nabla_v  w)
    &=
      \frac12 \sum_{j=1}^d \dot g(\nabla_v w,e_j) \gamma(e_j), \qandq \\
    %%%
    [\nabla_v,\dot \gamma^\BG(w)]
    &=
      \frac12\sum_{j=1}^d \paren[\big]{(\nabla_v \dot g)(w,e_j) + \dot g(\nabla_v w,e_j)} \gamma(e_j).
  \end{align*}
  \autoref{Eq_DotNabla} follows directly from the preceding identities.

  Evidently,
  $V_{\sG,\bp}$ and $V_{\nabla,\bp}$ are contained in the vertical tangent bundle $\VerticalTangentBundle{\SpaceOfDiracBundles}{\SpaceOfMetrics}$.
  The Uhlenbeck gauge fixing condition in \autoref{Prop_UniversalSpaceOfDiracBundles_TangentBundle_VNabla} ensures that $V_{\sG,\bp} \cap V_{\nabla,\Uh;\bp} = 0$.
  To prove that $V_{\sG,\bp} \oplus V_{\nabla,\Uh;\bp}$ is the entire vertical tangent space $T_\bp{\SpaceOfDiracBundles}/{\SpaceOfMetrics}$,
  it suffices to show that every infinitesimal deformation $\dot\gamma$ of a Clifford multiplication $\gamma$ is of the form $[\xi,\gamma]$ with $\xi \in \Gamma\paren{X,\fo(S)}$.
  This, however, is a consequence of the fact that Clifford algebras are separable finite-dimensional $\R$--algebras and, therefore,
  their representations are infinitesimally rigid;
  see, e.g., \cite[Theorem 5]{Hochschild1946} or \cite[Lemma 9.1.9 and Theorem 9.2.11]{Weibel1994}.
  In fact, the arguments found loc.~cit.~reveal that
  $\dot\gamma = [\xi,\gamma]$ holds for
  \begin{equation*}
    \xi \coloneq 2^{-d}\sum_{I \subseteq \set{1,\ldots,d}} \xi_I \in \Gamma(X,\fo(S))
  \end{equation*}
  with
  \begin{equation*}
    \xi_I \coloneq \sum_{\ell=1}^k (-1)^\ell \gamma(e_{i_1}) \cdots \gamma(e_{i_{\ell-1}})\dot\gamma(e_{i_\ell}) \gamma(e_{i_\ell}) \gamma(e_{i_{\ell-1}}) \cdots \gamma(e_{i_1})
  \end{equation*}
  for $I = \set{ i_1,\ldots,i_k} \subseteq \set{1,\ldots,d}$ with $i_1 < \ldots < i_k$.
  Of course, this can also be verified by direct computation.
\end{proof}

\begin{remark}
  \label{Rmk_VGNotSubbundle}
  The $V_{\sG,\bp}$ and $V_{\nabla,\Uh;\bp}$ need not form subbundles because there can be $\bp \in \SpaceOfDiracBundles$ for which the map $\Gamma\paren{X,\fo(S)} \to V_{\sG,\bp}$ has a non-trivial kernel.
  Of course,
  the $H_{\BG,\bp}$ and $V_{\nabla,\bp}$ do form subbundles.
\end{remark}

\begin{remark}
  \label{Rmk_OClS}
  Here is how to understand the Lie algebra bundle $\fo_{\Cl}(S)$.
  For every $x \in X$ and choice of orthogonal frame of $T_xX$,
  the fibre $S_x$ is a representation of the Clifford algebra $\Cl_d$.
  The representation $S_x$ decomposes into irreducible representations, the \defined{pinor representations} \cite[Definition 11.10]{Harvey1990}, which can be read off from the classification \cite[Theorem 11.3]{Harvey1990}:
  \begin{enumerate}
  \item
    If $d = 0,1,2 \pmod 4$,
    then there is a unique pinor representation $\bP$ and, therefore,
    $S_x = \bP^m$ for some $m \in \N_0$.
    The commuting algebra $\mathbb{K} \coloneq \End_{\Cl}(\bP)$ of $\bP$ is isomorphic to $\R$; $\C$; $\H$ depending on whether $d = 0,6$; $1,5$; $2,4 \pmod 8$, respectively.
    Therefore,
    \begin{equation*}
      \fo_{\Cl}(S_x) \iso
      \begin{cases}
        \fo(m) & \textnormal{if}~ d = 0,6 \pmod 8 \\
        \fu(m) & \textnormal{if}~ d = 1,5 \pmod 8 \\
        \sp(m) & \textnormal{if}~ d = 2,4 \pmod 8.
      \end{cases}
    \end{equation*}
  \item
    If $d = 3 \pmod 4$,
    then there are two pinor representations $\bP^\pm$, distinguished by whether the volume element acts as $\pm \one$.
    Therefore, $S_x = (\bP^+)^{m_+} \oplus (\bP^-)^{m_-}$.
    The commuting algebra $\mathbb{K} \coloneq \End_{\Cl}(\bP^\pm)$ of $\bP^\pm$ is isomorphic to $\H$, $\R$, depending on whether $d = 3$, $7 \pmod 8$, respectively.
    Therefore,
    \begin{equation*}
      \fo_{\Cl}(S_x) \iso
      \begin{cases}
        \sp(m_+) \oplus \sp(m_-) & \textnormal{if}~ d = 3 \pmod 8 \\
        \fo(m_+) \oplus \fo(m_-) & \textnormal{if}~ d = 7 \pmod 8.
      \end{cases}
    \end{equation*} 
  \end{enumerate}  
  If $X$ is spin and a spin structure $\fs$ on $(X,g)$ is chosen,
  then the pinor representations give rise to pinor bundles $\bP$ or $\bP^+,\bP^-$ over $X$ and,
  consequently,
  $S = M \otimes_{\mathbb{K}} \bP$ or $S = M^+ \otimes_{\mathbb{K}} \bP^+ \oplus M^- \otimes_{\mathbb{K}} \bP^-$ with $M$ or $M^+,M^-$ encoding the multiplicities.
  \qedhere
\end{remark}

\begin{remark}
  \label{Rmk_TwistingCurvature}
  For every $(g;\gamma,\nabla) \in \SpaceOfDiracBundles$ the curvature $F_\nabla \in \Omega^2\paren{X,\fo(S)}$ decomposes into \defined{Riemannian spin curvature} $\Riem_g^S \in \Omega^2\paren{X,\fo(S)}$ and the \defined{twisting curvature} $F_\nabla^\tw \in \Omega^2\paren{X,\fo_\Cl(S)}$ defined by
  \begin{align*}
    \Riem_g^S(u,v)
    &\coloneq \frac18 \sum_{k,\ell=1}^d \Inner{\Riem_g(u,v)e_k,e_\ell} [\gamma(e_k),\gamma(e_\ell)] \qand \\
    %%%
    F_\nabla^\tw
    &\coloneq
    F_\nabla - \Riem_g^S;
  \end{align*}
  see \cite[Proposition 3.43]{BerlineGetzlerVergne1992}.
  This decomposition is useful, because the contribution of $\Riem_g^S$ to various formulae often simplifies quite notably.
\end{remark}

%%% Local Variables:
%%% mode: latex
%%% TeX-master: "UniversalModuliSpaceOfZ2ZHarmonicSpinors"
%%% ispell-local-dictionary: "british"
%%% End:

\subsection{The space of ramified Euclidean line bundles}
\label{Sec_ModuliSpaceOfRamifiedEuclideanLineBundles}

Here is the other half of the base of the tame Fréchet space bundles $\BundleOfAdmissibleSingularSpinors$, $\BundleOfSingularSpinors$, and $\BundleOfResidues$.

\begin{definition}[{cf.~\cite[§3.4.3]{Doan2024}}]
  \label{Def_ModuliSpaceOfRamifiedEuclideanLineBundles}
  The \defined{moduli space of ramified Euclidean line bundles} over $X$ is constructed as follows:
  \begin{enumerate}
  \item
    \label{Def_ModuliSpaceOfRamifiedEuclideanLineBundles_BranchingLoci}
    Denote by $\SpaceOfBranchingLoci$ the \emph{set} of closed codimension-two submanifolds $Z \subseteq X$.
    Equip $\SpaceOfBranchingLoci$ with the tame Fréchet manifold structure described in \cite[Part I Example 4.1.7 and Example 4.4.7, and Part II Corollary 2.3.7]{Hamilton1982:NashMoser}.
    
    Given $Z \in \SpaceOfBranchingLoci$ and a tubular neighbourhood $\jmath \co NZ \supseteq U \to X$ of $Z$,
    define the map $\phi_\jmath^{-1} \co \Gamma\paren{Z,NZ} \supseteq \Gamma\paren{Z,U} \to \SpaceOfBranchingLoci$ by
    \begin{equation*}
      \phi_\jmath^{-1}(v) \coloneq \im \paren{\jmath \circ v}.
    \end{equation*}
    Denote by $\phi_\jmath$ the inverse of $\phi_\jmath^{-1} \co \Gamma\paren{Z,U} \to \im \phi_\jmath^{-1}$.
    Equip $\Gamma\paren{Z,NZ}$ with the usual structure of a tame Fréchet space;
    see \cite[Corollary II.1.3.9]{Hamilton1982:NashMoser}.
    The charts constructed in \cite[Example I.4.1.7]{Hamilton1982:NashMoser} are of the form $\phi_\jmath$. 
  \item
    \label{Def_ModuliSpaceOfRamifiedEuclideanLineBundles_Set}
    Denote by $\SpaceOfRamifiedLineBundles$ the \emph{set} of isometry classes of ramified Euclidean line bundles $(Z,\fl)$ over $X$ with smooth $Z$.
    Consider the map
    \begin{equation*}
      \Br \co \SpaceOfRamifiedLineBundles \to \SpaceOfBranchingLoci
    \end{equation*}
    assigning to $(Z,[\fl])$ the branching locus $Z$.
  \item
    \label{Def_ModuliSpaceOfRamifiedEuclideanLineBundles_Topology}
    For every open subset $U \subseteq X$ denote by $\SpaceOfBranchingLoci_U \subseteq \SpaceOfBranchingLoci$ the open subset of those $Z \in \SpaceOfBranchingLoci$ for which:
    \begin{enumerate}
    \item
      $Z \subseteq U$, and
    \item
      the restriction $\rH^1\paren{X\setminus Z,\set{\pm 1}} \to \rH^1\paren{X\setminus U,\set{\pm 1}}$ is an isomorphism.
    \end{enumerate}
    The subset $\SpaceOfBranchingLoci_U$ is open and the monodromy representations assemble into an injection
    \begin{equation*}
      \tau_U \co \Br^{-1}(\SpaceOfBranchingLoci_U) \to \SpaceOfBranchingLoci_U \times \rH^1\paren{X\setminus U,\set{\pm 1}}.
    \end{equation*}
    Equip $\SpaceOfRamifiedLineBundles$ with the coarsest \emph{topology} with respect to which the maps $\tau_U$ are continuous.
  \item
    \label{Def_ModuliSpaceOfRamifiedEuclideanLineBundles_Manifold}
    With respect to the above topology $\Br \co \SpaceOfRamifiedLineBundles \to \SpaceOfBranchingLoci$ is a locally finite covering map; cf.~\cite[proof of Theorem 3.3.2]{tomDieck2008:AlgebraicTopology}.
    In particular, $\SpaceOfRamifiedLineBundles$ inherits the structure of a tame Fréchet manifold from $\SpaceOfBranchingLoci$.
    Moreover, for every $(Z,[\fl]) \in \SpaceOfRamifiedLineBundles$,
    \begin{equation*}
      T_{(Z,[\fl])}\SpaceOfRamifiedLineBundles = T_Z\SpaceOfBranchingLoci = \Gamma\paren{Z,NZ}.
    \end{equation*}
  \end{enumerate}
  In the following $(Z,[\fl])$ is frequently abbreviated to $[\fl]$.
  This is harmless, because $Z$ can be recovered as $\Br(\fl)$.
  \qedhere
\end{definition}

\begin{remark}
  \label{Rmk_ModuliSpaceOfRamifiedEuclideanLineBundles_Topology}
  Here are some further comments regarding \autoref{Def_ModuliSpaceOfRamifiedEuclideanLineBundles_Topology}.
  Euclidean line bundles $\fl$ over $X\setminus Z$ are classified by their Stiefel--Whitney class $w_1(\fl) \in \Hom\paren{\pi_1(X\setminus Z),\set{\pm 1}} = \rH^1\paren{X\setminus Z,\set{\pm 1}}$.
  The latter cohomology group fits into the exact sequence
  \begin{align*}
    \rH^1\paren{X;\set{\pm 1}} \xhookrightarrow{i^*} \rH^1\paren{X\setminus Z;\set{\pm 1}} \xrightarrow{\delta}  \rH^2\paren{X, X \setminus Z;\set{\pm 1}} \xrightarrow{j^*} \rH^2\paren{X,\set{\pm 1}}.
  \end{align*}
  Excision and the Thom isomorphism identify $\rH^2\paren{X, X \setminus Z;\set{\pm 1}} \iso \rH^0\paren{Z,\set{\pm 1}}$.
  The monodromy condition means that $\delta(w_1(\fl)) = -\one \in \rH^0\paren{Z,\set{\pm 1}}$.
  In particular, its image under the final map vanishes.
  But this is exactly the Poincaré dual of $[Z] \in \rH_{d-2}\paren{X,\set{\pm 1}}$.
  Therefore, the latter must vanish for $Z$ to arise as the branching locus;
  moreover, if it vanishes, then the space of isometry classes of ramified Euclidean line bundles with branching locus $Z$ is an $\rH^1\paren{X,\set{\pm 1}}$--torsor.
\end{remark}

\begin{prop}[{cf.~\cite[Proposition 3.60]{Doan2024}}]
  \label{Prop_TwistedUniversalRamifiedEuclideanLineBundle} 
  $\SpaceOfRamifiedLineBundles$ carries a \defined{twisted universal ramified Euclidean line bundle} in the following sense:
  \begin{enumerate}
  \item
    \label{Prop_TwistedUniversalRamifiedEuclideanLineBundle_NonBranchingLocus} 
    Consider the \defined{universal non-branching locus}
    \begin{equation*}
      \mathring\bX
      \coloneq
      \set[\big]{
        ([\fl],x) \in \SpaceOfRamifiedLineBundles \times X : x \notin \Br(\fl)
      }
      \subseteq
      \SpaceOfRamifiedLineBundles\times X.
    \end{equation*}
    The projection map ${p} \co \mathring\bX \to \SpaceOfRamifiedLineBundles$ is a fibre bundle.
  \item
    \label{Prop_TwistedUniversalRamifiedEuclideanLineBundle_Bundle} 
    Denote by $\fU$ the set of those open subsets $\sU \subseteq \SpaceOfRamifiedLineBundles$ for which ${p} \co \mathring\bX|_\sU \to \sU$ is trivial.
    For every $\sU \in \fU$ there is a Euclidean line bundle $\fL_\sU$ over $\mathring\bX|_\sU$ and for every $[\fl] \in \sU$ an isomorphism
    \begin{equation*}
      \fL_\sU|_{[\fl] \times \paren{X\setminus \Br(\fl)}} \iso \fl
    \end{equation*}
    unique up to $\Aut(\fl) = \set{\pm 1}$.
  \item
    \label{Prop_TwistedUniversalRamifiedEuclideanLineBundle_Cocycle} 
    There is a Čech $2$--cocycle $\lambda \in \check\rC^2(\fU,\underline{\set{\pm 1}})$
    together with, for every $\sU,\sV \in \fU$, an isomorphism
    $\phi_\sV^\sU \co \fL_\sU|_{\mathring\bX|_{\sU\cap \sV}} \iso \fL_\sV|_{\mathring\bX|_{\sU\cap \sV}}$
    such that for every $\sU,\sV,\sW \in \fU$
    \begin{equation*}
      {p}^*\lambda_{\sU\sV\sW} \coloneq \phi_\sU^\sW \phi_\sW^\sV \phi_\sV^\sU
      \in \check\rC^0\paren[\big]{\mathring\bX|_{\sU \cap \sV \cap \sW},\set{\pm 1}}.
    \end{equation*}
    Here $\underline{\set{\pm 1}}$ is the constant sheaf associated with the abelian group $\set{\pm 1}$.
  \end{enumerate}  
\end{prop}

\begin{proof}
  The proof is identical to that of \cite[Proposition 3.60]{Doan2024} and, therefore, omitted.
\end{proof}

The \v Cech cohomology class $\smash{[\lambda] \in \check\rH^2\paren{\SpaceOfRamifiedLineBundles,\underline{\set{\pm1}}}}$ is the obstruction to gluing $\set{ \fL_\sU : \sU \in \fU }$ to an untwisted universal ramified Euclidean line bundle $\fL$ over $\mathring\bX$.
Unfortunately, it does not always vanish.
A pragmatic solution to this problem is to instead work with the Euclidean line bundle
\begin{equation*}
  \fL \coloneq \coprod_{\sU \in \fU} \fL_\sU
  \quad\textnormal{and}\quad
  \SpaceOfRamifiedLineBundles^0 \coloneq \coprod_{\sU \in \fU} \sU.
\end{equation*}
The significance of $\fL$ is that for every $(\sU;[\fl]) \in \SpaceOfRamifiedLineBundles^0$ it selects a choice of representative of $[\fl]$.
Henceforth, this choice shall be denoted by $\fl$,
abusing but simplifying notation.
Moreover, in the upcoming discussion, $(\sU;[\fl])$ is abbreviated to $[\fl]$ and $\SpaceOfRamifiedLineBundles^0$ to $\SpaceOfRamifiedLineBundles$.
This abuse of notation is inconsequential,
because ultimately $\Z/2\Z$ harmonic spinors are considered up to the action of $\R^\times$ and $\set{\pm 1} \leq \R^\times$:
forming the quotient by $\set{\pm 1}$ removes the ambiguity.

\begin{remark}
  The origin of the above issue is the stabiliser $\Aut(\fl)=\set{\pm 1}$
  and,
  in principle,
  one ought to work with orbifolds to deal with it.  
  The Čech $2$--cocycle $\lambda$ produces a proper étale tame Fréchet Lie groupoid which presents a tame Fréchet orbifold $\underline{\SpaceOfRamifiedLineBundles}$:
  the \defined{fine moduli space of ramified Euclidean line bundles};
  cf.~\cite[§3]{Moerdijk2002}.
  The corresponding orbifold counterpart $\underline{\mathring\bX}$ of $\mathring\bX$ does admit a \defined{universal ramified Euclidean line bundle} $\underline{\fL}$;
  cf.~\cite[§5]{Moerdijk2002}.
\end{remark}

%%% Local Variables:
%%% mode: latex
%%% TeX-master: "UniversalModuliSpaceOfZ2ZHarmonicSpinors"
%%% ispell-local-dictionary: "british"
%%% End:

\subsection{The extended gauge group and the Kosmann lift}
\label{Sec_ExtendedGaugeGroup}

The following symmetries of $\SpaceOfDiracBundles \times \SpaceOfRamifiedLineBundles$ are crucial for the upcoming discussion.

\begin{definition}
  \label{Def_ExtendedIsometry}
  The \defined{extended gauge group} of $S$ is the tame Fréchet Lie group $\ExtendedGaugeGroup{S}$ constructed as follows:
  \begin{enumerate}  
  \item
    \label{Def_ExtendedIsometry_Set}
    An \defined{extended isometry} of $S$ is a diffeomorphism $u \in \Diff(S)$ for which there exists a diffeomorphism $\check u \in \Diff(X)$ such that $u$ defines an isometry $S \iso \check u^*S$.
    These form a subgroup $\ExtendedGaugeGroup{S} \leq \Diff(S)$.    
  \item
    A moment's thought shows that extended isometries lift to $\O(r)$--equivariant diffeomorphisms of $\Fr(S)$,
    the orthogonal frame bundle of $S$;
    in fact, $\ExtendedGaugeGroup{S} = \Diff(\Fr(S))^{\O(r)}$, abusing notation.
    By \cite[Proposition 2.6]{Molitor2010:UnimodularAutomorphisms},
    $\ExtendedGaugeGroup{S} \leq \Diff(\Fr(S))$ is a tame Fréchet Lie subgroup.
  \item
    From the above perspective,
    $\ExtendedGaugeGroup{S}$ acts smoothly and tamely on $S = \Fr(S) \times_{\O(r)} \R^r$ and $X = \Fr(S) \times_{\O(r)} \set{*}$ from the left
    by
    \begin{equation*}
      u \cdot [p,\phi] \coloneq [u(p),\phi]
      \qandq
      u \cdot [p,*] \coloneq [u(p),*],
    \end{equation*}
    respectively.
    As a consequence,
    it acts smoothly and tamely on $\Gamma\paren{X,S}$, $\SpaceOfDiracBundles$, and $\SpaceOfRamifiedLineBundles$ from the right by pullback.%
    \footnote{%
      \label{Footnote_ActionLift}
      The reader should be warned that $\ExtendedGaugeGroup{S}$ does not lift to an action on $\SpaceOfRamifiedLineBundles^0$;
      however, the action lifts locally.
      This is sufficient for \autoref{Prop_TransformToNormalForm}.
    }
    \qedhere
  \end{enumerate}
\end{definition}

The subset $\Diff_S(X) \subseteq \Diff(X)$ of those diffeomorphisms that lift to $\ExtendedGaugeGroup{S}$ forms an open subgroup.
In fact,
$\ExtendedGaugeGroup{S}$ fits into the exact sequence of tame Fréchet Lie groups
\begin{equation*}
  \GaugeGroup{S} \into \ExtendedGaugeGroup{S} \onto \Diff_S(X).
\end{equation*}
This induces a corresponding exact sequence of tame Fréchet Lie algebras
\begin{equation*}
  \Gamma\paren{X,\fo(S)} \into \Lie(\ExtendedGaugeGroup{S}) \coloneq T_\one\ExtendedGaugeGroup{S} \onto \Vect(X).
\end{equation*}
As a sequence of tame Fréchet \emph{spaces},
this splits, but not canonically so.
For example,
for every orthogonal connection $\nabla$ on $S$ the \defined{horizontal lift}
\begin{equation*}
  \lift_\nabla \co \Vect(X) \to \Lie(\ExtendedGaugeGroup{S}) \subseteq \Vect(S)
\end{equation*}  
defines a right splitting.
The right Lie algebra action of $\lift_\nabla(v)$ on $\bp = (g;\gamma,\nabla) \in \SpaceOfDiracBundles$ yields $(\sL_vg;\dot \gamma,\dot\nabla) \in T_\bp\SpaceOfDiracBundles$ with $\dot\gamma(w) \coloneq \gamma(\nabla_w v)$.
For the purposes of this article it is convenient to use the following lift that fits better with the Bourguignon--Gauduchon connection defined in \autoref{Prop_UniversalSpaceOfDiracBundles_TangentBundle}~\autoref{Prop_UniversalSpaceOfDiracBundles_TangentBundle_HBG}.

\begin{definition}[cf.~{\cites[Definition III.2.1]{Kosmann1972}[Proposition 17]{Bourguignon1992}}]
  \label{Def_KosmannLift}
  For every $\bp = (g;\gamma,\nabla) \in \SpaceOfDiracBundles$ the \defined{Kosmann lift} is the tame linear map $\lift_\bp^K \co \Vect(X) \to \Lie(\ExtendedGaugeGroup{S})$ defined by
  \begin{align*}
    \lift_\bp^K(v)
    &\coloneq
      \lift_\nabla(v) + \kappa_v
      \qwith \\
    \kappa_v
    &\coloneq
      -\frac18 \sum_{i,j=1}^d g(\nabla_{e_i}v,e_j) [\gamma(e_i),\gamma(e_j)]
      \in
      \Gamma\paren{X,\fo(S)}.
      \qedhere
  \end{align*}
\end{definition}

\begin{prop}[{cf.~\cite[§V.1.2]{Kosmann1972}}]
  \label{Prop_LieAlgebraActionOfKosmannLift}
  Let
  $\bp \in \SpaceOfDiracBundles$ and
  $v \in \Vect(X)$.
  \begin{enumerate}
  \item
    \label{Prop_LieAlgebraActionOfKosmannLift_Spinor}
    The right Lie algebra action of $\lift_\bp^K(v) \in \Lie(\ExtendedGaugeGroup{S})$ on $\phi \in \Gamma\paren{X,S}$ is given by the \defined{Kosmann Lie derivative};
    that is:
    \begin{equation*}
      \phi \cdot \lift_\bp^K(v) = \sL_v^K \phi
      \qwithq
      \sL_v^K \coloneq \nabla_v + \kappa_v.
      \qedhere
    \end{equation*}
  \item
    \label{Prop_LieAlgebraActionOfKosmannLift_DiracBundle}
    The right Lie algebra action of $\lift_\bp^K(v) \in \Lie(\ExtendedGaugeGroup{S})$ on $\bp \in \SpaceOfDiracBundles$ is related to the Bourguignon--Gauduchon lift $(\dot g;\dot\gamma^\BG, \dot\nabla^\BG) \in T_\bp\SpaceOfDiracBundles$ of $\dot g \coloneq \sL_vg \in T_g\SpaceOfMetrics$,
    defined in \autoref{Prop_UniversalSpaceOfDiracBundles_TangentBundle}~\autoref{Prop_UniversalSpaceOfDiracBundles_TangentBundle_HBG},
    by
    \begin{equation}
      \label{Eq_KosmannLift}
      \bp \cdot \lift_\bp^K(v)
      =
      \paren[\big]{\dot g;\dot\gamma^\BG, \dot\nabla^\BG + F_\nabla^\tw(v,-)}.
    \end{equation}
  \end{enumerate}
\end{prop}

\begin{remark}
  \label{Rmk_KosmannLift}
  The following should help to make the Kosmann lift more palatable:
  \begin{enumerate}
  \item
    The requirement that the right Lie algebra action of the lift $\tilde v \in \Lie(\ExtendedGaugeGroup{S})$ of $v \in \Vect(X)$ reproduces $\dot\gamma^\BG$ determines $\tilde v$ up to the addition of an arbitrary
    $\xi_v \in \Gamma\paren{X,\fo_{\Cl}(S)}$.
  \item
    The requirement \autoref{Eq_KosmannLift} determines $\lift_\bp^K(v)$ uniquely up to the addition of an arbitrary $\xi_v\in \ker\paren[\big]{\nabla \co \Gamma\paren{X,\fo_{\Cl}(S)} \to \Omega^1\paren{X,\fo_{\Cl}(S)}}$.
    The latter vanishes for typical $\bp \in \SpaceOfDiracBundles$.
    Therefore, \autoref{Eq_KosmannLift} and a suitable continuity condition determine the Kosmann lift uniquely.
  \item
    For generic $\bp \in \SpaceOfDiracBundles$ and $v \in \Vect(X)$,
    $F_\nabla^\tw(v,-) \notin \im \paren[\big]{\nabla \co \Gamma\paren{X,\fo_{\Cl}(S)} \to \Omega^1\paren{X,\fo_{\Cl}(S)}}$.
    As a consequence, this term cannot be eliminated from \autoref{Eq_KosmannLift}.
  \item
    The reader should be warned that,
    despite its name,
    the Kosmann Lie derivative does not define a Lie algebra homomorphism $\Vect(X) \to \Lie(\ExtendedGaugeGroup{S})$;
    see \cite[Proposition 18]{Bourguignon1992}.
    It does not even agree with the Lie derivative on $\Omega^\bullet(X)$ for the Dirac bundle that gives rise to the Hodge--de Rham operator $\rd + \rd^*$.
    \qedhere
  \end{enumerate}
\end{remark}

\begin{proof}[Proof of \autoref{Prop_LieAlgebraActionOfKosmannLift}]
  \autoref{Prop_LieAlgebraActionOfKosmannLift_Spinor} is obvious.
  From this \autoref{Prop_LieAlgebraActionOfKosmannLift_DiracBundle} is derived by the following direct computation.
  Set
  \begin{equation*}
    (\dot g;\dot\gamma,\dot\nabla) \coloneq \bp \cdot \lift_\bp^K(v).
  \end{equation*}
  Evidently,
  $\dot g = \sL_vg$.  
  Since
  \begin{equation}
    \label{Eq_LieDerivativeOfMetric}
    (\sL_v g)(w,z) = g(\nabla_w v,z) + g(\nabla_z v,w),
  \end{equation}
  the formulae in  \autoref{Prop_UniversalSpaceOfDiracBundles_TangentBundle}~\autoref{Prop_UniversalSpaceOfDiracBundles_TangentBundle_HBG} concretise to
  \begin{align*}
    \dot\gamma^\BG(w)
    &=
      \frac12 \sum_{i=1}^d \paren[\big]{g(\nabla_w v,e_i) + g(\nabla_{e_i} v,w)} \gamma(e_i) \qand \\
    \dot\nabla_w^\BG
    &=
      \frac18 \sum_{i,j=1}^d \paren[\big]{g(\nabla^2_{e_i,e_j}v,w) + g(\nabla^2_{e_i,w} v,e_j)} [\gamma(e_i),\gamma(e_j)].
  \end{align*}
  Therefore,
  using \autoref{Eq_DoubleCommutator},
  \begin{align*}
    \dot \gamma(w)
    &=
      [\sL_v^K,\gamma(w)] - \gamma([v,w])
    =
      [\kappa_v,\gamma(w)] + \gamma(\nabla_w v) \\
    &=
      -\frac18 \sum_{i,j=1}^d g(\nabla_{e_i}v,e_j) [[\gamma(e_i),\gamma(e_j)],\gamma(w)]
      +
      \gamma(\nabla_w v) \\
    &=
      \frac12 \sum_{i=1}^d \paren[\big]{g(\nabla_wv,e_i) + g(\nabla_{e_i}v,w)} \gamma(e_i)
      =
      \dot\gamma^\BG(w).
  \end{align*}
  Moreover,
  \begin{align*}
    \dot\nabla_w
    &=
      [\sL_v^K,\nabla_w] - \nabla_{[v,w]}
    =
      [\nabla_v,\nabla_{w}] - \nabla_{[v,w]} - [\nabla_w,\kappa_v] \\
    &=
      \Riem_g^S(v,w) + F_\nabla^\tw(v,w)
      +  \frac18 \sum_{i,j=1}^d g(\nabla_{w,e_i}^2v,e_j) [\gamma(e_i),\gamma(e_j)] \\
    &=
      \dot\nabla_{w}^\BG + F_{\nabla}^\tw(v,w) 
      + \rR
  \end{align*}
  with $\rR$ denoting the remainder
  \begin{equation*}    
    \rR
    \coloneq
    \frac18\sum_{i,j=1}^d
    \paren{
      g(\Riem_g(v,w)e_i,e_j) + g(\nabla^2_{w,e_i} v,e_j) - g(\nabla^2_{e_i,e_j}v,w) - g(\nabla^2_{e_i,w} v,e_j)
    }
    [\gamma(e_i),\gamma(e_j)].
  \end{equation*}
  As a consequence of the algebraic Bianchi identity,
  \begin{align*}
    \rR
    &=
      \frac1{16} \sum_{i,j=1}^d
      \paren*{g(\Riem_g(v,w)e_i,e_j) + 2g(\Riem_g(w,e_i)v,e_j)}
      [\gamma(e_i),\gamma(e_j)] \\
    &=
      \frac1{16} \sum_{i,j=1}^d
      \paren[\big]{g(\Riem_g(v,w)e_i,e_j) + g(\Riem_g(v,e_j)w,e_i) + g(\Riem_g(v,e_i)e_j,w)}
      [\gamma(e_i),\gamma(e_j)] \\
    &=
      0.
      \qedhere
  \end{align*}
\end{proof}

%%% Local Variables:
%%% mode: latex
%%% TeX-master: "UniversalModuliSpaceOfZ2ZHarmonicSpinors"
%%% ispell-local-dictionary: "british"
%%% End:

\subsection{The universal bundles of (admissible) singular spinors and residues}
\label{Sec_UniversalBundlesOfSingularSpinorsAndResidues}

The following discussion constructs $\BundleOfAdmissibleSingularSpinors$, $\BundleOfSingularSpinors$, and $\BundleOfResidues$ as well as $\UniversalDiracOperator$ and $\UniversalResidueMap$ using the method introduced by \citet[§4]{Donaldson2021}.

\begin{definition}
  \label{Def_Universal_1}
  Here is the first stage of the construction:
  \begin{enumerate}
  \item
    \label{Def_Universal_1_AdmissibleSingularSpinors}
    The \defined{universal bundle of admissible singular spinors} is the \emph{set}
    \begin{align*}
      \BundleOfAdmissibleSingularSpinors
      &\coloneq
        \coprod_{(\bp,[\fl]) \in \SpaceOfDiracBundles \times \SpaceOfRamifiedLineBundles} H_a^\infty\Gamma\paren{X\setminus \Br(\fl),S\otimes\fl;\bp}
    \end{align*}    
    together with the projection map
    $\BundleOfAdmissibleSingularSpinors \to \SpaceOfDiracBundles \times \SpaceOfRamifiedLineBundles$
    and the tame Fréchet space structures on the fibres,
    defined in \autoref{Sec_ConormalAndAdaptedSobolevSpaces} and \cite[Definition 4.2 (b)]{BeraWalpuski2025}.
    Here and throughout the remainder of this article,
    additional decorations with $\bp$ and $\fl$ are used to indicate the dependence on these parameters.
  \item
    \label{Def_Universal_1_SingularSpinors}
    The \defined{universal bundle of singular spinors} is the \emph{set}
    \begin{align*}
      \BundleOfSingularSpinors
      &\coloneq
        \coprod_{(\bp,[\fl]) \in \SpaceOfDiracBundles \times \SpaceOfRamifiedLineBundles} H_b^\infty\Gamma\paren{X\setminus \Br(\fl),S\otimes\fl;\bp}
    \end{align*}    
    together with the projection map
    $\BundleOfSingularSpinors \to \SpaceOfDiracBundles \times \SpaceOfRamifiedLineBundles$
    and the tame Fréchet space structures on the fibres,
    defined in \autoref{Sec_ConormalAndAdaptedSobolevSpaces} and \cite[Definition 4.2 (a)]{BeraWalpuski2025}.
  \item
    \label{Def_Universal_1_Residues}    
    The \defined{universal bundle of residues} is the \emph{set}
    \begin{equation*}
      \BundleOfResidues \coloneq \coprod_{(\bp,[\fl]) \in \SpaceOfDiracBundles \times \SpaceOfRamifiedLineBundles} \Gamma\paren{\Br(\fl),\ResidueBundle_\bp}
    \end{equation*}
    together with the projection map $\BundleOfResidues \to \SpaceOfDiracBundles \times \SpaceOfRamifiedLineBundles$ and the usual tame Fréchet space structures on the fibres.
  \item
    \label{Def_Universal_1_D}
    The \defined{universal Dirac operator} $\UniversalDiracOperator \co \BundleOfAdmissibleSingularSpinors \to \BundleOfSingularSpinors$ is defined by
    \begin{equation*}
      \UniversalDiracOperator(\bp,[\fl];\phi)
      \coloneq
      (\bp,[\fl];D_\bp^\fl \phi).
    \end{equation*}
  \item
    \label{Def_Universal_1_Res}
    The \defined{universal residue map} $\UniversalResidueMap \co \BundleOfAdmissibleSingularSpinors \to \BundleOfResidues$ is defined by
    \begin{equation*}
      \UniversalResidueMap(\bp,[\fl];\phi)
      \coloneq
      (\bp,[\fl];\res \phi).
    \end{equation*}
  \end{enumerate}
  It remains to endow
  $\BundleOfAdmissibleSingularSpinors$,
  $\BundleOfSingularSpinors$, and
  $\BundleOfResidues$
  with the structure of tame Fréchet space bundles such that $\UniversalDiracOperator$ and $\UniversalResidueMap$ become maps of tame Fréchet space bundles.
\end{definition}

\begin{figure}[h]
  \centering
  \begin{equation*}
    \begin{tikzcd}
      \BundleOfAdmissibleSingularSpinors \ar{r}{\UniversalDiracOperator} \ar{d}[swap]{\UniversalResidueMap} & \BundleOfSingularSpinors \\
      \BundleOfResidues
    \end{tikzcd}
  \end{equation*}
  \caption{The universal Dirac operator and residue map.}
  \label{Fig_UniversalDiracOperatorAndResidueMap}
\end{figure}

The above task is not entirely trivial.
Even if $(Z_0,[\fl_0]) \in \SpaceOfRamifiedLineBundles$ is held fixed and $\bp_0,\bp \in \SpaceOfDiracBundles$ are close,
$H_a^\infty\paren{X\setminus Z_0,S\otimes\fl_0;\bp_0}$ and $H_a^\infty\paren{X\setminus Z_0,S\otimes\fl_0;\bp}$ need not agree, even as subspaces of $\Gamma\paren{X\setminus Z_0,S\otimes\fl_0}$, unless a much stronger relation between $\bp_0$ and $\bp$ is assumed.

\begin{definition}
  \label{Def_Slice}
  A set of \defined{normal data} for a codimension two submanifold $Z_0 \subseteq X$ consists of:
  \begin{enumerate}[label=\rm{(N\arabic*)},ref=N\arabic*]
  \item
    a rank $2$ subbundle $N \subseteq TX|_{Z_0}$ complementary to $TZ_0$,
  \item
    a Euclidean metric $h$ on $N$, and
  \item
    \label{Def_NormalData_I}
    an $I \in \Gamma\paren{Z_0,\Hom\paren{\Wedge^2N,\End(S|_{Z_0})}}$ such that
    $I(\Or) \in \End(S_x)$ is an orthogonal complex structure
    for every $x \in Z_0$ and $\Or \in \Wedge^2N_x$ with $\abs{\Or} = 1$.
  \end{enumerate}
  Every $\bp = (g;\gamma,\nabla) \in \SpaceOfDiracBundles$ defines normal data for a codimension two submanifold $Z \subseteq X$ by declaring
  \begin{equation*}
    N \coloneq TZ^\perp, \quad
    h \coloneq g|_N, \qandq
    I \coloneq \gamma|_{\Wedge^2 N}.
  \end{equation*}
  The \defined{slice} associated with a ramified Euclidean line bundle $(Z_0,[\fl_0]) \in \SpaceOfRamifiedLineBundles$ and normal data $(N,h,I)$ for $Z_0$ is the tame Fréchet submanifold
  \begin{equation*}
    \Slice
    \coloneq
    \set[\big]{
      (\bp,[\fl_0]) \in \SpaceOfDiracBundles \times \SpaceOfRamifiedLineBundles
      :
      \bp ~\text{induces the normal data}~ (N,h,I) ~\text{for}~ Z_0
    }
    \subseteq
    \SpaceOfDiracBundles \times \SpaceOfRamifiedLineBundles.
  \end{equation*}
  These slices foliate $\SpaceOfDiracBundles \times \SpaceOfRamifiedLineBundles$ in the sense that every $(\bp,[\fl]) \in \SpaceOfDiracBundles \times \SpaceOfRamifiedLineBundles$ is contained in some slice.
  \qedhere
\end{definition}

\begin{prop}
  \label{Prop_Universal_Slice}
  Let
  $\Slice \subseteq \SpaceOfDiracBundles \times \SpaceOfRamifiedLineBundles$ be a slice and
  $(\bp_0,[\fl_0]) \in \Slice$.
  \begin{enumerate}
  \item
    \label{Prop_Universal_Slice_Bundles}
    For every $(\bp,[\fl]) \in \Slice$ the identity maps induce tame isomorphisms
    \begin{align*}
      H_a^\infty\Gamma\paren{X\setminus \Br(\fl_0),S\otimes\fl_0;\bp_0}
      &\iso
        H_a^\infty\Gamma\paren{X\setminus \Br(\fl),S\otimes\fl;\bp} = \BundleOfAdmissibleSingularSpinors_{\bp,\fl}, \\
        %%% 
      H_b^\infty\Gamma\paren{X\setminus \Br(\fl_0),S\otimes\fl_0;\bp_0}
      &\iso
        H_b^\infty\Gamma\paren{X\setminus \Br(\fl),S\otimes\fl;\bp} = \BundleOfSingularSpinors_{\bp,\fl}, \qand \\
        %%% 
      \Gamma\paren{\Br(\fl_0),\ResidueBundle_{\bp_0}}
      &\iso
        \Gamma\paren{\Br(\fl),\ResidueBundle_{\bp}} = \BundleOfResidues_{\bp,\fl}.
    \end{align*}
  \item
    \label{Prop_Universal_Slice_Maps}
    The maps
    \begin{align*}
      \UniversalDiracOperator|_\Slice \co \Slice \times H_a^\infty\Gamma\paren{X\setminus \Br(\fl_0),S\otimes\fl_0;\bp_0} &\to H_b^\infty\Gamma\paren{X\setminus \Br(\fl_0),S\otimes\fl_0;\bp_0} \qand \\
      {\UniversalResidueMap}|_{\Slice} \co \Slice \times H_a^\infty\Gamma\paren{X\setminus \Br(\fl_0),S\otimes\fl_0;\bp_0} &\to \Gamma\paren{ \Br(\fl_0),\ResidueBundle_{\bp_0}}
    \end{align*}
    induced by $\UniversalDiracOperator$ and $\UniversalResidueMap$ are tame smooth.
  \end{enumerate}
\end{prop}

\begin{proof}%[Proof of \autoref{Prop_ComparableTameFrechetSpaces}]  
  By \cite[Proposition 3.14]{BeraWalpuski2025},
  $D_0 \coloneq D_{\bp_0}^{\fl_0}$ and $D_{\bp}^{\fl_0}$ differ from (cut-offs of) the corresponding model operators $\mathring{D}_{\bp_0}^{\fl_0}$ and $\mathring{D}_{\bp}^{\fl_0}$ \cite[Definition 3.12]{BeraWalpuski2025} by operators in $\DiffOp_b^1(S\otimes\fl_0)$.  
  Direct inspection of the model operators and a glance at \cite[Proof of Lemma 4.29]{BeraWalpuski2025} show that
  \begin{equation*}
    D_\bp^{\fl_0} = D_{\bp_0}^{\fl_0} + B_{\bp,\bp_0}
  \end{equation*}
  with $B_{\bp,\bp_0} \in \DiffOp_b^1(S\otimes\fl_0)$.
  This immediately implies \autoref{Prop_Universal_Slice_Bundles}.
  Inspection of $B_{\bp,\bp_0}$ establishes \autoref{Prop_Universal_Slice_Maps} as in \cite[Part II Theorem 3.3.3]{Hamilton1982:NashMoser}.
\end{proof}

This provides the desired tame Fréchet bundle structures for the restrictions
$\BundleOfAdmissibleSingularSpinors|_\Slice$,
$\BundleOfSingularSpinors|_\Slice$, and
$\BundleOfResidues|_\Slice$
to slices:
in fact, it canonically trivialises them.
These tame Fréchet bundle structures can be extended with the help of the following method due to \citet{Donaldson2021}.

\begin{prop}[cf.~{\cite[Proposition 4.2]{Donaldson2021}}]
  \label{Prop_TransformToNormalForm}
  Let $\Slice$ be a slice.
  For every $(\bp_0,[\fl_0]) \in \Slice$ there is a \defined{local uniformiser};
  that is:
  an open neighbourhood $\sU \subseteq \SpaceOfDiracBundles \times \SpaceOfRamifiedLineBundles$ of $(\bp_0,[\fl_0])$ and a tame smooth map
  \begin{equation*}
    u \co \sU \to \ExtendedGaugeGroup{S}
  \end{equation*}
  such that for every $(\bp,[\fl]) \in \sU$
  \begin{equation*}
    u(\bp,[\fl])^*(\bp,[\fl]) \in \Slice.
  \end{equation*}
  Moreover, given a compact subset $K \subseteq X \setminus \Br(\fl_0)$,
  the local uniformiser can be chosen such that 
  $u(\bp,[\fl])|_K = \one_S$
  for every
  $(\bp,[\fl]) \in \sU$.
\end{prop}

\begin{proof}%[Proof of \autoref{Prop_TransformToNormalForm}]
  \cite[Proof of Proposition 4.2]{Donaldson2021} explains the construction of a tame smooth map $\check u \co \sU \to \Diff_0(X)$ which admits a lift $u \co \sU \to \ExtendedGaugeGroup{S}$ such that,
  for every $(\bp,[\fl]) \in \sU$, 
  $u(\bp,[\fl])|_K = \one_S$ and
  $u(\bp,[\fl])^*\bp$ defines the same normal data as $\bp_0$---except possibly for the almost complex structure $I$,
  because in \cite{Donaldson2021} there is no analogue of \autoref{Def_NormalData_I}.
  After possibly shrinking $\sU$,
  this defect can be repaired by observing the following:
  \begin{enumerate}
  \item
    For every $I$ as in \autoref{Def_NormalData_I} that is sufficiently close to $I_0$, there is a $u_0 \in \sG(S|_{Z_0})$ close to $\one_{S|_{Z_0}}$ with $u_0^*I = I_0$;
  \item
    Since $u_0$ is close to $\one_{S|_{Z_0}}$,
    it can be extended to $\overline u_0 \in \sG(S)$.
  \end{enumerate}
  Evidently, this repair can be carried out in a way that produces the desired tame smooth map $u \co \sU \to \ExtendedGaugeGroup{S}$.
\end{proof}

The following two observations are verified by direct inspection.

\begin{prop}
  \label{Prop_UniversalBundles_Atlases_Chart}
  Let
  $(\bp,[\fl]) \in \SpaceOfDiracBundles \times \SpaceOfRamifiedLineBundles$. 
  Every extended gauge transformation
  $u \in \ExtendedGaugeGroup{S}$
  induces tame isomorphisms
  \begin{align*}
    u \co H_a^\infty\Gamma\paren{X\setminus \Br(\fl),S\otimes \fl;\bp} &\to H_a^\infty\Gamma\paren{X\setminus \Br(\check u^*\fl),S\otimes \check u^*\fl;u^*\bp}, \\
    %%% 
    u \co H_b^\infty\Gamma\paren{X\setminus \Br(\fl),S\otimes \fl;\bp} &\to H_b^\infty\Gamma\paren{X\setminus \Br(\check u^*\fl),S\otimes \check u^*\fl;u^*\bp}, \qand \\
    %%% 
    u \co \Gamma\paren{\Br(\fl),\ResidueBundle_{\bp}} &\to \Gamma\paren{\Br(\check u^*\fl),\ResidueBundle_{u^*\bp}}.
  \end{align*}
  In particular,
  these lift the action of
  $\ExtendedGaugeGroup{S}$
  on the right of
  $\SpaceOfDiracBundles \times \SpaceOfRamifiedLineBundles$
  to
  $\BundleOfAdmissibleSingularSpinors$,
  $\BundleOfSingularSpinors$, and
  $\BundleOfResidues$.
  \qed
\end{prop}

\begin{prop}  
  \label{Prop_UniversalBundles_Atlases_Transition}
  Let $\Slice$ be a slice and $(\bp_0,[\fl_0]) \in \Slice$.
  \begin{enumerate}
  \item
    The subgroup
    \begin{equation*}
      \Stab_{\ExtendedGaugeGroup{S}}(\Slice)
      \coloneq
      \set[\big]{
        u \in \ExtendedGaugeGroup{S}
        :
        u^*\Slice = \Slice
      }
      \subset
      \ExtendedGaugeGroup{S}
    \end{equation*}
    is a tame Fréchet Lie subgroup.
  \item
    The right actions of $\Stab_{\ExtendedGaugeGroup{S}}(\Slice)$ on
    $H_a^\infty\paren{X\setminus Z_0,S\otimes \fl_0;\bp_0}$,
    $H_b^\infty\paren{X\setminus Z_0,S\otimes \fl_0;\bp_0}$, and
    $\Gamma\paren{Z_0,\ResidueBundle_{\bp_0}}$
    are tame smooth.
    \qed
  \end{enumerate}
\end{prop}

\begin{definition}
  \label{Def_UniversalBundles_Atlases}
  The following completes the construction begun in \autoref{Def_Universal_1}.
  Endow
  $\BundleOfAdmissibleSingularSpinors$,
  $\BundleOfSingularSpinors$, and
  $\BundleOfResidues$    
  with the unique tame smooth Fréchet bundle structure in which for every $(\bp_0,[\fl_0]) \in \SpaceOfDiracBundles \times \SpaceOfRamifiedLineBundles$ and  for every local uniformiser $u \co \sU \to \ExtendedGaugeGroup{S}$ as in \autoref{Prop_TransformToNormalForm}
  the bijections
  $\sigma \co \BundleOfAdmissibleSingularSpinors|_\sU \to \sU \times H_a^\infty\paren{X\setminus \Br(\fl_0),S\otimes \fl_0;\bp_0}$,
  $\tau \co \BundleOfSingularSpinors|_\sU \to \sU \times H_b^\infty\paren{X\setminus \Br(\fl_0),S\otimes \fl_0;\bp_0}$, and
  $\upsilon \co \BundleOfResidues|_\sU \to \sU \times \Gamma\paren{\Br(\fl_0),\ResidueBundle_{\bp_0}}$
  defined by
  \begin{gather*}
    \sigma(\bp,[\fl];\phi)
    \coloneq
    (\bp,[\fl];u(\bp,[\fl])\phi), \quad
    %%% 
    \tau(\bp,[\fl];\phi)
    \coloneq
    (\bp,[\fl];u(\bp,[\fl])\phi), \qandq \\
    %%% 
    \upsilon(\bp,[\fl];\rho)
    \coloneq
    (\bp,[\fl];u(\bp,[\fl])\rho)
  \end{gather*}
  are local trivialisations.
  By \autoref{Prop_UniversalBundles_Atlases_Chart},
  these are compatible with the tame Fréchet space structures on the fibres.
  By \autoref{Prop_UniversalBundles_Atlases_Transition},
  they form an atlas.
\end{definition}

Given the above construction,
the following is an immediate consequence of \autoref{Prop_Universal_Slice}.

\begin{cor}
  \label{Cor_UniversalDiracOperatorAndResidueMapsAreSmooth}   
  The universal Dirac operator $\UniversalDiracOperator \co \BundleOfAdmissibleSingularSpinors \to \BundleOfSingularSpinors$ and the universal residue map $\UniversalResidueMap \co \BundleOfAdmissibleSingularSpinors \to \BundleOfResidues$ are maps of tame Fréchet bundles.
\end{cor}

\begin{remark}
  \label{Rmk_BundleOfSingularSpinors}
  Here are some observations regarding the above discussion:
  \begin{enumerate}
  \item
    \label{Rmk_BundleOfSingularSpinors_Robust}
    The construction of the tame Fréchet space bundle structures is quite robust, insofar as it depends only on \autoref{Prop_Universal_Slice}, \autoref{Prop_UniversalBundles_Atlases_Chart}, and \autoref{Prop_UniversalBundles_Atlases_Transition}.
    In particular, the method shall be used later in this article to construct certain tame Fréchet subbundles of $\BundleOfAdmissibleSingularSpinors$.
  \item
    \label{Rmk_BundleOfSingularSpinors_B}
    The construction of $\BundleOfSingularSpinors$ could have been carried out with a less delicate definition of slice,
    because $H_b^\infty\Gamma\paren{X\setminus Z,S\otimes \fl;\bp}$ does not depend on the normal structure.
  \item
    \label{Rmk_BundleOfSingularSpinors_NotInverseLimitsOfHilbertSpaceBundles}
    For $k \in \N_0$ the right actions of $\Stab_{\ExtendedGaugeGroup{S}}(\Slice)$ on the Hilbert spaces $H_a^k\paren{X\setminus Z_0,S\otimes \fl_0;\bp_0}$, $H_b^k\paren{X\setminus Z_0,S\otimes \fl_0;\bp_0}$, and $H^k\paren{Z_0, \ResidueBundle_{\bp_0}}$ fail to be smooth (for the same reason that the translation action $\R \circlearrowright H^1(\R)$ fails to be smooth).
    Therefore, $\BundleOfAdmissibleSingularSpinors$, $\BundleOfSingularSpinors$, and $\BundleOfResidues$ are not (in an obvious way) inverse limits of Hilbert space bundles---%
    at least not in the usual sense, but see \cite[§4.4]{TaubesWu2021:Z2ZEigenfunctionsTopologicalAspects}.
  \item
    \label{Rmk_BundleOfSingularSpinors_Equivariant}
    By construction,
    $\BundleOfAdmissibleSingularSpinors$, $\BundleOfSingularSpinors$, and $\BundleOfResidues$ are $\ExtendedGaugeGroup{S}$--equivariant tame Fréchet space bundles over $\SpaceOfDiracBundles \times \SpaceOfRamifiedLineBundles$,
    and $\UniversalDiracOperator$ and $\UniversalResidueMap$ are $\ExtendedGaugeGroup{S}$--equivariant (up to the caveat pointed out in \autoref{Footnote_ActionLift}).
    \qedhere
  \end{enumerate}
\end{remark}

%%% Local Variables:
%%% mode: latex
%%% TeX-master: "UniversalModuliSpaceOfZ2ZHarmonicSpinors"
%%% ispell-local-dictionary: "british"
%%% End:

\subsection{Covariant derivatives of the universal Dirac operator}
\label{Sec_PartialCovariantDerivatives}

This section explains how to differentiate the universal Dirac operator $\bD$.
Of course, this requires the choice of (partial) connections.

\begin{definition}
  \label{Def_SliceDerivative}
  The \defined{slice distribution} $\sF^\Slice \subseteq T\paren{\SpaceOfDiracBundles\times\SpaceOfRamifiedLineBundles}$ is the distribution induced by the slices foliating $\SpaceOfDiracBundles \times \SpaceOfRamifiedLineBundles$;
  that is:
  for every $(\bp,[\fl]) \in \SpaceOfDiracBundles \times \SpaceOfRamifiedLineBundles$
  \begin{equation*}    
    \sF_{\bp,[\fl]}^\Slice \coloneq T_{\bp,[\fl]}\Slice_{\bp,[\fl]}
  \end{equation*}
  with $\Slice_{\bp,[\fl]}$ denoting the unique slice through $(\bp,[\fl])$.
  The canonical trivialisations of $\BundleOfAdmissibleSingularSpinors|_\Slice$, $\BundleOfSingularSpinors|_\Slice$, and $\BundleOfResidues|_\Slice$ defined by \autoref{Prop_Universal_Slice} define the \defined{slice partial covariant derivatives} $\nabla^\Slice$ with respect to $\sF^\Slice$ on $\BundleOfAdmissibleSingularSpinors$, $\BundleOfSingularSpinors$, and $\BundleOfResidues$.
\end{definition}

To compute $\nabla^\Slice\bD$,
it is helpful to decompose $\sF^\Slice$ according to \autoref{Prop_UniversalSpaceOfDiracBundles_TangentBundle}.

\begin{prop}
  \label{Prop_SlicesAreTameFrechetSubmanifolds}
  For every $(Z_0,[\fl_0]) \in \SpaceOfRamifiedLineBundles$, normal data $(N,h,I)$ for $Z_0$, and $(\bp,[\fl]) \in \Slice$ the tangent space $T_{\bp,[\fl]}\Slice$ decomposes into three tame Fréchet subspaces
  \begin{align*}
    T_{\bp,[\fl]}\Slice
    =
    H_{\BG;\bp,[\fl]}^\Slice \oplus V_{\sG;\bp,[\fl]}^\Slice \oplus V_{\nabla,\Uh;\bp,[\fl]}^\Slice
  \end{align*}
  defined by
  \begin{align*}
    H_{\BG;\bp,[\fl]}^\Slice
    &\coloneq
      \set[\big]{
      (\dot g;\dot\gamma^\BG,\dot\nabla^\BG;0) \in H_{\BG;\bp,[\fl]}
      :
      \dot g|_Z \in \Gamma(Z,\Hom(S^2TZ,\R))
      }, \\
    V_{\sG;\bp,[\fl]}^\Slice
    &\coloneq
      \set[\big]{
      (0;[\xi,\gamma],-\nabla\xi;0) \in V_{\sG;\bp,[\fl]}
      :
      [\xi|_Z,\gamma|_{\Wedge^2N}] = 0
      }, \qand \\
    V_{\nabla,\Uh;\bp,[\fl]}^\Slice
    &\coloneq
      V_{\nabla,\Uh;\bp,[\fl]};
  \end{align*}
  in particular,
  $V_{\nabla;\bp} \oplus 0 \subseteq T_{\bp,[\fl]}\Slice$.
\end{prop}

\begin{proof}%[Proof of \autoref{Prop_SlicesAreTameFrechetSubmanifolds}]
  Direct inspection reveals that
  $(\dot g;\dot\gamma,\dot\nabla;0) \in T_\bp\SpaceOfDiracBundles \oplus T_{[\fl]}\SpaceOfRamifiedLineBundles$ lies in $\sF^{\Slice}$ if and only if:
  \begin{enumerate}
  \item
    $\dot g|_Z \in \Gamma(Z,\Hom(S^2TZ,\R))$ and
  \item
    for every $x \in Z_0$ and every orthonormal basis $(e_1,e_2)$ of $N_x$
    \begin{equation*}
      \dot\gamma(e_1)\gamma(e_2) + \gamma(e_1)\dot\gamma(e_2) = 0.
    \end{equation*}    
  \end{enumerate}
  This together with \autoref{Prop_UniversalSpaceOfDiracBundles_TangentBundle} implies the assertion. 
\end{proof}

\begin{prop}
  \label{Prop_SliceDerivativeOfUniversalDiracOperator}
  For every $(\bp,[\fl]) \in \SpaceOfDiracBundles\times\SpaceOfRamifiedLineBundles$ the following hold:
  \begin{enumerate}
  \item
    \label{Prop_SliceDerivativeOfUniversalDiracOperator_1}
    For every $\dot\bp_\BG = (\dot g;\dot \gamma^\BG,\dot\nabla^\BG;0) \in H_{\BG;\bp,[\fl]}^\Slice$
    \begin{equation*}
      \nabla_{\dot\bp_\BG}^\Slice \bD(\bp,\fl;\phi)
      =
      - \frac12 \sum_{i,j=1}^d \dot g(e_i,e_j) \gamma(e_i)\nabla_{e_j} \phi
      + \frac14\gamma\paren[\big]{(\nabla^*\dot g)^\sharp + \nabla(\tr_g\dot g)}\phi.
    \end{equation*}
  \item
    \label{Prop_SliceDerivativeOfUniversalDiracOperator_2}
    For every $\dot\bp_\sG = (0;[\xi,\gamma],-\nabla\xi;0) \in V_{\sG;\bp,[\fl]}^\Slice$
    \begin{equation*}
      \nabla_{\dot\bp_\sG}^\Slice \bD(\bp,\fl;\phi)
      =
      \sum_{i=1}^d \paren[\big]{
        [\xi,\gamma(e_i)] \nabla_{e_i} \phi
        - \gamma(e_i) (\nabla_{e_i}\xi) \phi
      }.
    \end{equation*}
  \item
    \label{Prop_SliceDerivativeOfUniversalDiracOperator_3}
    For every $\dot\bp_\nabla = (0;0,\dot\nabla;0) \in V_{\nabla;\bp,[\fl]}^\Slice$
    \begin{equation*}
      \nabla_{\dot\bp_\nabla}^\Slice \bD(\bp,\fl;\phi)
      =
      \sum_{i=1}^d \gamma(e_i) \dot\nabla_{e_i}\phi.
    \end{equation*}    
  \end{enumerate}
  Moreover, $\nabla^\Slice \UniversalResidueMap = 0$.
\end{prop}

\begin{proof}%[Proof of \autoref{Prop_SliceDerivativeOfUniversalDiracOperator}]
  The variation $\dot\sharp$ of the isomorphism $\sharp \co T^*X \to TX$ induced by $\dot g \in T_g\SpaceOfMetrics$ satisfies
  \begin{equation*}
    \dot\sharp \alpha =  - \sum_{i=1}^d \dot g(\alpha^\sharp,e_i) e_i.
  \end{equation*}
  Therefore,
  for every $(\bp;[\fl]) \in \SpaceOfDiracBundles\times\SpaceOfRamifiedLineBundles$
  and $\dot \bp = (\dot g;\dot\gamma,\dot\nabla;0) \in \sF_{\bp,[\fl]}^\Slice,$
  \begin{equation}
    \label{Eq_AbstractVariationOfD}
    \nabla_{\dot\bp}^\Slice \bD(\bp,\fl;\phi)
    =
    \sum_{i=1}^d \paren[\big]{
      \dot\gamma(e_i) \nabla_{e_i}\phi
      + \gamma(e_i)\dot\nabla_{e_i}\phi
    }
    - \sum_{i,j=1}^d \dot g(e_i,e_j) \gamma(e_i) \nabla_{e_j}\phi.
  \end{equation}
  For every $\dot\bp = \dot\bp_\BG = (\dot g;\dot \gamma^\BG,\dot\nabla^\BG;0)$ as in \autoref{Prop_UniversalSpaceOfDiracBundles_TangentBundle}~\autoref{Prop_UniversalSpaceOfDiracBundles_TangentBundle_HBG} this concretises to
  \begin{align*}
    \nabla_{\dot\bp_\BG}^\Slice \bD(\bp,\fl;\phi)
    &=
      -\frac12 \sum_{i,j=1}^d \dot g(e_i,e_j) \gamma(e_i) \nabla_{e_j}\phi
      +
      \frac18 \sum_{i,j,k=1}^d \paren{\nabla_{e_i} \dot g}(e_j,e_k) \gamma(e_k)[\gamma(e_i),\gamma(e_j)]\phi \\
    &=
      - \frac12 \sum_{i,j=1}^d \dot g(e_i,e_j) \gamma(e_i)\nabla_{e_j} \phi
      + \frac14\gamma\paren[\big]{(\nabla^*\dot g)^\sharp + \nabla(\tr_g\dot g)}\phi.
  \end{align*}
  The final step uses the observation
  \begin{equation*}
    \gamma(e_k)[\gamma(e_i),\gamma(e_j)] + \gamma(e_j)[\gamma(e_i),\gamma(e_k)]
    =
    -2\delta_{ij} \gamma(e_k)    
    +4\delta_{jk} \gamma(e_i)
    -2\delta_{ki} \gamma(e_j).
  \end{equation*}
  This proves \autoref{Prop_SliceDerivativeOfUniversalDiracOperator_1}.
  The remaining formulae \autoref{Prop_SliceDerivativeOfUniversalDiracOperator_2} and \autoref{Prop_SliceDerivativeOfUniversalDiracOperator_3} are obvious.
\end{proof}

\begin{remark}
  \label{Rmk_SliceDerivativeOfUniversalDiracOperator_Conormal}
  Direct inspection of the above formulae and a glance at \autoref{Prop_SliceDerivativeOfUniversalDiracOperator} reveal that
  \begin{gather*}
    \nabla_{\dot\bp_\BG}^\Slice\bD(\bp,\fl;-),
    \nabla_{\dot\bp_\nabla}^\Slice\bD(\bp,\fl;-)
    \in \DiffOp_b^1(S\otimes\fl) \qand \\
    \nabla_{\dot\bp_\sG}^\Slice\bD(\bp,\fl;-)
    \in
    \DiffOp_b^0(S\otimes\fl) D_\bp^\fl + \DiffOp_b^1(S\otimes\fl).
    \qedhere
  \end{gather*}
\end{remark}

\begin{remark}
  \label{Rmk_VolumeRenormalisedSliceDerivative}
  The differential operator $\nabla_{\dot\bp_\BG}^\Slice \bD(\bp,\fl;-)$ is not formally self-adjoint.
  In fact, a brief computation shows that its formal adjoint satisfies
  \begin{equation*}
    \nabla_{\dot\bp_\BG}^\Slice \bD(\bp,\fl;-)^*
    =
    \nabla_{\dot\bp_\BG}^\Slice \bD(\bp,\fl;-)
    -
    \frac12 \gamma(\nabla(\tr_g\dot g)).
  \end{equation*}
  This can be rectified by using the \defined{volume-renormalised slice partial covariant derivative}
  \begin{equation*}
    \nabla_{\dot\bp}^{\Slice,\vol}
    \coloneq
    \nabla_{\dot\bp}^\Slice + \tfrac14 \tr_g\dot g
  \end{equation*}
  instead of $\nabla^\Slice$;
  cf.~\cites[§2.2]{Maier1997}[§2.5]{AmmannDahl2025}.
  For the purposes of this article,
  the above is not an issue and it suffices to work with $\nabla^\Slice$.
  However,
  it would be one in setting up a Brill--Noether theory of $\Z/2\Z$ harmonic spinors;
  cf.~\cite[§1.A]{Doan2018}.
\end{remark}

The right action of $\ExtendedGaugeGroup{S}$ on $\SpaceOfDiracBundles \times \SpaceOfRamifiedLineBundles$ induces a right action
\begin{equation*}
  T(\SpaceOfDiracBundles \times \SpaceOfRamifiedLineBundles)
  \circlearrowleft
  T\ExtendedGaugeGroup{S} = \ExtendedGaugeGroup{S} \ltimes \Lie{\ExtendedGaugeGroup{S}}.
\end{equation*}
This can be used to construct (local) partial covariant derivatives with respect to
\begin{equation*}
  \sF^\SpaceOfRamifiedLineBundles \coloneq \pr_2^*T\SpaceOfRamifiedLineBundles
  \subseteq
  T\paren{\SpaceOfDiracBundles \times \SpaceOfRamifiedLineBundles}
\end{equation*}
from $\nabla^\Slice$ and the following infinitesimal version of \autoref{Prop_TransformToNormalForm}.

\begin{prop}
  \label{Prop_InfinitesimalSliceUniformisers}
  For every $(\bp_0,[\fl_0]) \in \SpaceOfDiracBundles \times \SpaceOfRamifiedLineBundles$
  there is a \defined{local infinitesimal uniformiser};
  that is:
  an open neighbourhood $(\bp_0,[\fl_0]) \in \sU \subseteq \SpaceOfDiracBundles \times \SpaceOfRamifiedLineBundles$ and
  a tame smooth map
  \begin{align*}
    \sF^\SpaceOfRamifiedLineBundles|_\sU &\to \Vect(X) \\
    (\bp,[\fl];v) &\mapsto \bar v = \bar v(\bp,[\fl];v)
  \end{align*}
  such that:
  \begin{enumerate}
  \item
    \label{Prop_InfinitesimalSliceUniformisers_Condition}
    for every $(\bp,[\fl];v) \in \sF^\SpaceOfRamifiedLineBundles|_\sU$
    \begin{equation*}
      (\bp,[\fl];v) \cdot \lift_\bp^K(\bar v) \in \sF_{\bp,[\fl]}^{\Slice},
    \end{equation*}
  \item
    \label{Prop_InfinitesimalSliceUniformisers_Linear}
    $\bar v(\bp,[\fl];-)$ is $\R$--linear,
    and
  \item
    \label{Prop_InfinitesimalSliceUniformisers_Tame}
    for every $(\bp,[\fl];v) \in \sF^\SpaceOfRamifiedLineBundles|_\sU$
    \begin{equation*}
      \Abs{\bar v(\bp,[\fl];v)}_{H^{k+1}} \lesssim_k \Abs{v}_{H^k}.
    \end{equation*}
  \end{enumerate}  
  Moreover,
  given a compact subset $K \subseteq X \setminus \Br(\fl_0)$,
  the local infinitesimal uniformiser can be chosen such that
  $\bar v(\bp,[\fl];v)|_K = 0$ for every $(\bp,[\fl];v) \in \sF^\SpaceOfRamifiedLineBundles|_\sU$.
\end{prop}

The proof of \autoref{Prop_InfinitesimalSliceUniformisers}~\autoref{Prop_InfinitesimalSliceUniformisers_Tame} uses the following well-known observation.

\begin{lemma}
  \label{Lem_1JetExtensionMap}
  Let $d,k \in \N_0$.
  There is a \defined{$1$--jet extension map}
  \begin{equation*}
    \ext^{(1)} \co \sS(\R^d) \oplus \sS(\R^d)^k \to \sS(\R^d \oplus \R^k)
  \end{equation*}
  such that
  for every $(f,g_1,\ldots,g_k) \in \sS(\R^d)^{k+1}$ its extension $F \coloneq  \ext^{(1)}(f,g_1,\ldots,g_k) \in \sS(\R^d\oplus\R^k)$ satisfies
  \begin{equation}
    \label{Lem_1JetExtensionMap_Restriction}
    F(-,0) = f \qandq
    \frac{\del F}{\del y_i}(-,0) = g_i,
  \end{equation}
  and
  for every $s \in \R$
  \begin{equation}
    \label{Lem_1JetExtensionMap_Estimate}
    \Abs{F}_{H^{s}}
    \lesssim_{s}
    \Abs{f}_{H^{s-k/2}} + \sum_{i=1}^k \Abs{g_i}_{H^{s-k/2-1}}.
  \end{equation}
\end{lemma}

\begin{proof}%[Proof of \autoref{Lem_1JetExtensionMap}]  
  Choose a $\chi \in \sS(\R^k)$ with $\chi(0) = 1$ and $\nabla^\ell \chi(0) = 0$ for every $\ell \in \N$.
  For $(f,g_1,\ldots,g_k) \in \sS(\R^d)^{k+1}$ define $F \coloneq \ext^{(1)}(f,g_1,\ldots,g_k)$ by
  \begin{align*}
    F(x,y)
    &\coloneq
      \int_{\R^d} e^{2\pi i\Inner{\xi,x}} \chi\paren{\bracket{\xi} y} \hat f(\xi) \, \rd \xi \\
    &\quad
      + \sum_{i=1}^k y_i  \int_{\R^d} e^{2\pi i\Inner{\xi,x}} \chi\paren{\bracket{\xi} y} \hat g_i(\xi) \, \rd \xi.
  \end{align*}
  \autoref{Lem_1JetExtensionMap_Restriction} is an immediate consequence of the Fourier inversion theorem.
  
  By direct computation,
  with the substitution $\tilde\eta = \eta/\bracket{\xi}$,
  \begin{align*}
    \Abs*{\int_{\R^d} e^{2\pi i\Inner{\xi,x}} \chi\paren{\bracket{\xi} y} \hat f(\xi) \, \rd \xi}_{H^s}^2
    &=
      \int_{\R^d} \abs{\hat f(\xi)}^2 \int_{\R^k} \bracket{(\xi,\eta)}^{2s} \abs{\sF_y\paren{\chi\paren{\bracket{\xi}y}}(\eta)}^2 \,\rd \eta\,\rd \xi \\
    &=
      \int_{\R^d} \bracket{\xi}^{2s-k} \abs{\hat f(\xi)}^2
      \underbrace{\int_{\R^k}  \bracket{\tilde \eta}^{2s} \abs{\hat\chi\paren{\tilde \eta}}^2 \,\rd \tilde\eta}_{\eqcolon c(s)}
      \rd \xi \\
    &=
      c(s) \Abs{f}_{H^{s-k/2}}^2,
  \end{align*}  
  and
  \begin{align*}
    \Abs*{y_i\int_{\R^d} e^{2\pi i\Inner{\xi,x}} \chi\paren{\bracket{\xi} y} \widehat{g_i}(\xi) \, \rd \xi}_{H^s}^2
    &=
      \int_{\R^d} \abs{\widehat{g_i}(\xi)}^2 \int_{\R^k} \bracket{(\xi,\eta)}^{2s} \abs{\sF_y\paren{y_i\chi\paren{\bracket{\xi}y}}(\eta)}^2 \,\rd \eta\,\rd \xi \\
    &=
      \int_{\R^d} \bracket{\xi}^{2s-k-2 }\abs{\widehat{g_i}(\xi)}^2
      \underbrace{\int_{\R^k} \bracket{\eta}^{2s} \abs{\sF_y\paren{y_i\chi\paren{y}}(\eta)}^2 \,\rd \eta}_{\eqcolon c_i(s)}
      \,\rd \xi \\
    &=
      c_i(s) \Abs{g_i}_{H^{s-k/2-1}}^2.
  \end{align*}
  Here $\sF_y$ denotes the Fourier transform in the variable $y \in \R^k$.
  Since $c(s), c_i(s) < \infty$,
  the above prove \autoref{Lem_1JetExtensionMap_Estimate}.
\end{proof}

\begin{proof}[Proof of \autoref{Prop_InfinitesimalSliceUniformisers}]  
  Let $(\bp;Z,[\fl];v) \in \sF^\SpaceOfRamifiedLineBundles|_\sU$ and $\bar v \in \Vect(X)$.
  By \autoref{Prop_LieAlgebraActionOfKosmannLift}~\autoref{Prop_LieAlgebraActionOfKosmannLift_DiracBundle},
  \begin{equation*}
    (\bp,[\fl];v) \cdot \lift_\bp^K(\bar v)
    =
    \paren[\big]{\sL_{\bar v}g;
    \dot\gamma^\BG,
    \dot\nabla^\BG + F_\nabla^\tw(\bar v,-);
    v-\bar v|_Z^N}
  \end{equation*}
  with $(\sL_{\bar v}g;\dot\gamma^\BG,\dot\nabla^\BG) \in T_\bp\SpaceOfDiracBundles$ denoting the Bourguignon--Gauduchon lift of $\sL_{\bar v}g \in T_g\SpaceOfMetrics$.
  Here $-^N$ denotes the orthogonal projection onto $NZ$.
  By inspection of \autoref{Prop_SlicesAreTameFrechetSubmanifolds},
  condition \autoref{Prop_InfinitesimalSliceUniformisers_Condition} translates into
  \begin{equation*}
    \bar v|_Z^N = v \qandq
    \paren{\sL_{\bar v} g}|_Z \in \Gamma\paren{Z,\Hom(S^2TZ,\R)}.
  \end{equation*}
  By \autoref{Eq_LieDerivativeOfMetric},
  the second condition means that for every $x\in Z$, $t \in T_xZ$, $n,m \in N_xZ$
  \begin{equation*}
    g(\nabla_n \bar v,t) = -g(n,\nabla_t \bar v)
    \qandq
    g(\nabla_n \bar v,m) + g(n,\nabla_m \bar v) = 0.
  \end{equation*}
  These are algebraic conditions on the $1$--jet $J^1\bar v|_Z$ of $\bar v$ at $Z$.
  Evidently,
  it is possible to algebraically solve for $J^1 \bar v \in \Gamma\paren{Z,J^1TX}$ and to extend it to a vector field $\bar v \in \Vect(X)$.
  In fact, using \autoref{Lem_1JetExtensionMap}, this can be done locally in $\SpaceOfDiracBundles \times \SpaceOfRamifiedLineBundles$ in such a way that $\bar v$ depends tame smoothly on $(\bp,[\fl];v)$ and satisfies \autoref{Prop_InfinitesimalSliceUniformisers_Linear} and \autoref{Prop_InfinitesimalSliceUniformisers_Tame}.
\end{proof}

\begin{definition}
  \label{Def_PartialConnectionAlongSpaceOfRamifiedLineBundles}
  Given a local infinitesimal uniformiser $\sF^\SpaceOfRamifiedLineBundles|_\sU \to \Vect(X), (\bp,[\fl];v) \mapsto \bar v = \bar v(\bp,[\fl],v)$ as in \autoref{Prop_InfinitesimalSliceUniformisers},
  define the partial covariant derivatives $\nabla^\SpaceOfRamifiedLineBundles$ with respect to $\sF^\SpaceOfRamifiedLineBundles$ on $\BundleOfAdmissibleSingularSpinors|_\sU$, $\BundleOfSingularSpinors|_\sU$, and $\BundleOfResidues|_\sU$ by
  \begin{equation*}
    \nabla^\SpaceOfRamifiedLineBundles_v \coloneq \nabla^\Slice_{(\bp,[\fl],v) \cdot \lift_\bp^K(\bar v)}.
    \qedhere
  \end{equation*}
\end{definition}

Unsurprisingly,
$\nabla_v^\SpaceOfRamifiedLineBundles\bD$ can be expressed in terms of the Kosmann Lie derivative defined in \autoref{Prop_LieAlgebraActionOfKosmannLift}.
This is crucial for the discussion in \autoref{Sec_ConstructionOfTameFrechetmanifoldStructures}.

\begin{prop}[{cf.~\cite[§V.1.2]{Kosmann1972}}]
  \label{Prop_CommutatorLKvD}  
  In the situation of \autoref{Def_PartialConnectionAlongSpaceOfRamifiedLineBundles},
  for every $(\bp;[\fl];\phi) \in \BundleOfAdmissibleSingularSpinors|_\sU$ and $v \in T_{[\fl]}\SpaceOfRamifiedLineBundles$
  \begin{equation*}
    \nabla_v^\SpaceOfRamifiedLineBundles\bD(\bp,[\fl];\phi)
    =
    [\sL_{\bar v}^K,D_\bp^\fl]\phi.
  \end{equation*}
\end{prop}

\begin{proof}%[Proof of \autoref{Prop_CommutatorLKvD}]
  By \autoref{Prop_LieAlgebraActionOfKosmannLift},
  \begin{align*}
    [\sL_{\bar v}^K,\gamma(e_i)]
    &=
      \dot\gamma^\BG(e_i) + \gamma([\bar v,e_i]) \qand \\
    [\sL_{\bar v}^K,\nabla_{e_i}]
    &=
      \dot\nabla_{e_i}^\BG + F_\nabla^\tw(\bar v,e_i) + \nabla_{[\bar v,e_i]}.
  \end{align*}
  Therefore,
  \begin{equation*}
    [\sL_{\bar v}^K,D]
    =
    \sum_{i=1}^d \dot\gamma^\BG(e_i)\nabla_{e_i}
    + \gamma(e_i)\dot\nabla_{e_i}^\BG
    + \gamma(e_i)F_\nabla^\tw(\bar v,e_i) 
    - \gamma(\nabla_{e_i}\bar v)\nabla_{e_i}
    - \gamma(e_i)\nabla_{\nabla_{e_i} \bar v}.
  \end{equation*}  
  A glance at \autoref{Eq_AbstractVariationOfD} reveals that this is precisely $ \nabla_v^\SpaceOfRamifiedLineBundles\bD(\bp,[\fl];-)$.
\end{proof}

\begin{remark}
  \label{Rmk_CommutatorLKvD}
  The formula in \autoref{Prop_CommutatorLKvD} has to be treated with some care,
  because
  \begin{equation*}
    \sL_{\bar v}^KD_\bp^\fl \qandq D_\bp^\fl\sL_{\bar v}^K
  \end{equation*}
  do not map $\BundleOfAdmissibleSingularSpinors_{\bp,\fl}$ to $\BundleOfSingularSpinors_{\bp,\fl}$,
  but their difference does.
  In fact,
  as a consequence of \autoref{Rmk_SliceDerivativeOfUniversalDiracOperator_Conormal},
  \begin{equation*}
    \nabla_v^\SpaceOfRamifiedLineBundles\bD(\bp,\fl;-) \in \DiffOp_b^1(S\otimes\fl).
    \qedhere
  \end{equation*}
\end{remark}

%%% Local Variables:
%%% mode: latex
%%% TeX-master: "UniversalModuliSpaceOfZ2ZHarmonicSpinors"
%%% ispell-local-dictionary: "british"
%%% End:

\subsection{Families of residue conditions}
\label{Sec_FamiliesOfResidueConditions}

This subsection explains how to impose residue conditions in families.

\begin{definition}
  \label{Prop_FamilyOfResidueCondition}
  A \defined{family of residue conditions} $(B,f,\bR)$ consists of:
  \begin{enumerate}
  \item
    a tame Fréchet manifold $B$,
  \item
    a tame smooth map $f \co B \to \SpaceOfDiracBundles \times \SpaceOfRamifiedLineBundles$, and    
  \item
    a tame Fréchet space subbundle $\bR \subseteq f^*\BundleOfResidues$.
    \qedhere
  \end{enumerate}
\end{definition}

Here is an example of a family of (spectral) residue conditions.

\begin{example}
  \label{Ex_FamilyOfAPSResidueConditions}
  Let $\lambda \in \R$.
  Denote by $\sU_\lambda \subseteq \SpaceOfDiracBundles \times \SpaceOfRamifiedLineBundles$ the open subset of those $(\bp,[\fl]) \in \SpaceOfDiracBundles \times \SpaceOfRamifiedLineBundles$ for which the branching locus operator $A_{\bp,[\fl]}$ does not have $\lambda$ as an eigenvalue;
  that is: $\lambda \notin \spec(A_{\bp,[\fl]})$.
  There is a tame Fréchet space subbundle
  \begin{equation*}
    \mathbf{APS}_\lambda \subseteq \BundleOfResidues|_{\sU_\lambda}
  \end{equation*}
  such that,
  for every $(\bp,[\fl]) \in \sU_\lambda$,
  \begin{equation*}
    \mathbf{APS}_{\lambda;\bp,[\fl]}
    =
    \one_{(-\infty,\lambda)}(A_{\bp,[\fl]})\Gamma\paren{\Br(\fl),\ResidueBundle_\bp}.
  \end{equation*}
  This defines the \defined{family of APS residue conditions} with threshold $\lambda \in \R$.
\end{example}

\begin{prop}
  \label{Prop_FamilyOfResidueConditions=>Subbundle}  
  For every family of residue conditions $(B,f,\bR)$
  \begin{equation*}
    \BundleOfAdmissibleSingularSpinors_\bR
    \coloneq
    (f^*{\UniversalResidueMap})^{-1}(\bR)
    \subseteq f^*\BundleOfAdmissibleSingularSpinors
  \end{equation*}
  is a tame Fréchet space subbundle.
\end{prop}

\begin{figure}[h]
  \centering
  \begin{equation*}
    \begin{tikzcd}
      \BundleOfAdmissibleSingularSpinors_\bR \ar{r}{\bD_\bR} \ar{d}{\UniversalResidueMap} \ar[hook]{rd} & f^*\BundleOfSingularSpinors \ar[equals]{rd} \\
      \bR \ar[hook]{rd} & f^*\BundleOfAdmissibleSingularSpinors \ar{r}{f^*\bD} \ar{d}{f^*{\UniversalResidueMap}} & f^*\BundleOfSingularSpinors \\
      & f^*\BundleOfResidues
    \end{tikzcd}
  \end{equation*}
  \caption{The construction of $\bD_\bR$.}
  \label{Fig_DR}
\end{figure}

\begin{definition}
  \label{Def_DR}
  In the situation of \autoref{Prop_FamilyOfResidueConditions=>Subbundle},
  the restriction of $f^*\bD$ to $\BundleOfAdmissibleSingularSpinors_\bR$ is denoted by
  \begin{equation*}
    \bD_\bR \co \BundleOfAdmissibleSingularSpinors_\bR \to f^*\BundleOfSingularSpinors.
    \qedhere
  \end{equation*}
\end{definition}

The proof of \autoref{Prop_FamilyOfResidueConditions=>Subbundle} relies on the following.

\begin{prop}
  \label{Prop_UniversalBundleOfZ2ZSpinors}
  The \defined{universal bundle of $\Z/2\Z$ spinors}
  \begin{equation*}
    \BundleOfZModTwoZSpinors \coloneq \ker \UniversalResidueMap \subseteq \BundleOfAdmissibleSingularSpinors
  \end{equation*}
  is a tame Fréchet space subbundle.
  Moreover,
  the exact sequence
  \begin{equation*}
    \BundleOfZModTwoZSpinors \incl \BundleOfAdmissibleSingularSpinors \stackrel{\UniversalResidueMap}{\onto} \BundleOfResidues
  \end{equation*}
  of tame Fréchet space bundles locally splits;
  that is:
  for every $(\bp,[\fl]) \in \SpaceOfDiracBundles \times \SpaceOfRamifiedLineBundles$
  there are an open neighbourhood $(\bp,[\fl]) \in \sU \subseteq \SpaceOfDiracBundles \times \SpaceOfRamifiedLineBundles$ and
  a map of tame Fréchet space bundles $\UniversalExtensionMap \co \BundleOfResidues|_\sU \to \BundleOfAdmissibleSingularSpinors|_\sU$ satisfying $\UniversalResidueMap \circ \UniversalExtensionMap = \one$.
\end{prop}

\begin{proof}%[Proof of \autoref{Prop_UniversalBundleOfZ2ZSpinors}]
  That $\BundleOfZModTwoZSpinors \subseteq \BundleOfAdmissibleSingularSpinors$ is a tame Fréchet space subbundle is evident by inspection of the trivialisations constructed in \autoref{Def_UniversalBundles_Atlases}.
  The splitting arises from the extension map mentioned in \autoref{Sec_ResidueMap} and constructed in \cite[Definition 3.40 (2)]{BeraWalpuski2025}.
\end{proof}

\begin{remark}  
  \label{Rmk_UniversalBundleOfZ2ZSpinors}
  Here are some observations regarding $\BundleOfZModTwoZSpinors$:
  \begin{enumerate}
  \item
    \label{Rmk_UniversalBundleOfZ2ZSpinors_Charts}
    The construction of $\BundleOfZModTwoZSpinors$ could be carried out directly with a less delicate definition of slice, because $\ker \res = H_a^\infty\Gamma\paren{X\setminus Z,S\otimes \fl;0;\bp}$ does not depend on the normal structure;
    cf.~\autoref{Rmk_BundleOfSingularSpinors}~\autoref{Rmk_BundleOfSingularSpinors_B}.
  \item
    \label{Rmk_UniversalBundleOfZ2ZSpinors_Derivative}
    For every $v \in \Vect(X)$ and $(\bp,[\fl];\Phi) \in \BundleOfZModTwoZSpinors$,
    $(\bp,[\fl];\nabla_v\Phi) \in \BundleOfSingularSpinors$ and the corresponding map
    \begin{equation*}
      \nabla \co \Vect(X) \times \BundleOfZModTwoZSpinors \to \BundleOfSingularSpinors
    \end{equation*}
    is tame smooth \cite[Proposition 4.27]{BeraWalpuski2025}.
    \qedhere
  \end{enumerate}
\end{remark}

\begin{proof}[Proof of \autoref{Prop_FamilyOfResidueConditions=>Subbundle}]
  Using \autoref{Prop_UniversalBundleOfZ2ZSpinors},
  $\BundleOfAdmissibleSingularSpinors_\bR$ can locally be presented as
  \begin{equation*}
    \BundleOfAdmissibleSingularSpinors_\bR|_{f^{-1}(\sU)} = f^*\bE_0 \oplus (f^*{\UniversalExtensionMap})(\bR).
    \qedhere
  \end{equation*}
\end{proof}

The arguments in \autoref{Sec_ConstructionOfTameFrechetmanifoldStructures} use families of residue conditions of the following kind.

\begin{definition}
  \label{Def_FamilyOfLocalResidueConditions}
  Here is how to consider local residue conditions in families:
  \begin{enumerate}
  \item
    \begin{enumerate}
    \item
      The \defined{universal branching locus} is the tame Fréchet submanifold
      \begin{equation*}
        \UniversalBr
        \coloneq
        \set[\big]{
          (\bp,[\fl];x) \in \SpaceOfDiracBundles \times \SpaceOfRamifiedLineBundles \times X : x \in \Br(\fl)
        }
        \subseteq
        \SpaceOfDiracBundles \times \SpaceOfRamifiedLineBundles\times X.
      \end{equation*}
      The projection map $p \co \UniversalBr \to \SpaceOfDiracBundles \times \SpaceOfRamifiedLineBundles$ is a fibre bundle.
    \item
      The normal bundle $\bN \coloneq N\UniversalBr$ shall be identified with
      \begin{equation*}
        \bN = \coprod_{(\bp,[\fl]) \in \SpaceOfDiracBundles \times \SpaceOfRamifiedLineBundles} N\Br(\fl).
      \end{equation*}
      It inherits a Euclidean metric from the metrics on $X$ which are part of $\bp \in \SpaceOfDiracBundles$.
      This induces an isomorphism $(\Wedge^2 \bN)^{\otimes 2} \iso \R$ and endows
      \begin{equation*}
        \AlgebraBundle \coloneq \R \oplus i \Wedge^2 \bN
      \end{equation*}
      with the structure of a bundle of algebras over $\UniversalBr$,
      whose fibres are isomorphic to $\C$, but canonically only up to complex conjugation.
    \item
      The construction from  \cite[Definition 3.21 and Remark 3.22]{BeraWalpuski2025} explained in \autoref{Sec_ResidueMap}, parametrised by $(\bp,[\fl]) \in \SpaceOfDiracBundles \times \SpaceOfRamifiedLineBundles$, assembles into an $\AlgebraBundle$--module $\bN^{-1/2}$.      
    \item
      Clifford multiplication endows the restriction of $\pr_X^*S$ to $\UniversalBr$ with the structure of an $\AlgebraBundle$--module.    
      The \defined{universal residue bundle} is the rank $r$ Euclidean bundle over $\UniversalBr$ defined by
      \begin{equation*}
        \UniversalResidueBundle \coloneq
        \pr_X^*S \otimes_\AlgebraBundle \bN^{-1/2}
        =
        \coprod_{(\bp,[\fl]) \in \SpaceOfDiracBundles \times \SpaceOfRamifiedLineBundles} \ResidueBundle_\bp.
      \end{equation*}
    \end{enumerate}
  \item
    A \defined{family of local residue conditions} $(\bB,f,\bV)$ consists of:
    \begin{enumerate}
    \item
      a tame Fréchet manifold $\bB$,
    \item
      a tame smooth map $f \co \bB \to \SpaceOfDiracBundles \times \SpaceOfRamifiedLineBundles$, and
    \item
      a subbundle $\bV \subseteq (p^*f)^* \UniversalResidueBundle$.
    \end{enumerate}
    Here $f^*p \co f^*\UniversalBr \to \bB$ is the pullback of the universal branching locus and $p^*f \co f^*\UniversalBr \to \UniversalBr$ is the projection map;
    see \autoref{Fig_FamiliesOfLocalResidueConditions}.
  \item
    A family of local residue conditions $(\bB,f,\bV)$ defines a family of residue conditions $(\bB,f,\bR)$
    such that for every $b \in \bB$ 
    \begin{equation*}
      \bR_b \coloneq \Gamma(Z_b,\bV|_{Z_b}) \subseteq \Gamma(Z_b,\ResidueBundle_{\bp_b})
      \qwithq    
      (\bp_b;Z_b,[\fl_b]) \coloneq f(b).
      \qedhere
    \end{equation*}
  \end{enumerate}
\end{definition}

\begin{remark}
  \label{Rmk_FamilyOfLocalResidueConditions}
  Here are some comments regarding \autoref{Def_FamilyOfLocalResidueConditions}:
  \label{Rmk_ResdiueBundle=>BundleOfResidues}
  \begin{enumerate}
  \item
    $\UniversalResidueBundle$ is related to the universal bundle of residues $\BundleOfResidues$,
    constructed in \autoref{Sec_UniversalBundlesOfSingularSpinorsAndResidues},
    by
    \begin{equation*}
      \BundleOfResidues = p_*\UniversalResidueBundle;
    \end{equation*}
    in the sense that,
    for every $(\bp,[\fl]) \in \SpaceOfDiracBundles \times \SpaceOfRamifiedLineBundles$,
    $\BundleOfResidues_{\bp,[\fl]} = \Gamma\paren{p^{-1}(\bp,[\fl]),\UniversalResidueBundle}$.
  \item
    \label{Rmk_FamilyOfLocalResidueConditions=>FamilyOfResidueConditions}
    Identifying $f^*\BundleOfResidues = f^*p_*\UniversalResidueBundle = (f^*p)_*(p^*f)^*\UniversalResidueBundle$,
    the tame Fréchet space subbundle $\bR \subseteq f^*\BundleOfResidues$ is   
    \begin{equation*}
      \bR \coloneq (f^*p)_*\bV \subseteq f^*\BundleOfResidues.
      \qedhere
    \end{equation*}
  \end{enumerate}
\end{remark}

\begin{figure}[h]
  \centering
  \begin{subfigure}{.48\textwidth}
    \centering
    \begin{equation*}
      \begin{tikzcd}
        \bV \ar[hook]{r} & (p^*f)^*\UniversalResidueBundle \ar{r} \ar{d} & \UniversalResidueBundle \ar{d} \\
        & f^*\UniversalBr \ar{r}{p^*f} \ar[swap]{d}{f^*p} & \UniversalBr \ar{d}{p} \\
        & \bB \ar[swap]{r}{f} & \SpaceOfDiracBundles \times \SpaceOfRamifiedLineBundles
      \end{tikzcd}
    \end{equation*}
    \caption{}
    \label{Fig_A}
  \end{subfigure}
  \begin{subfigure}{.48\textwidth}
    \centering
    \begin{equation*}
      \begin{tikzcd}
        \bR \coloneq (f^*p)_*\bV \ar[hook]{r} & f^*\BundleOfResidues \ar{r} \ar{d} & \BundleOfResidues = p_*\UniversalResidueBundle \ar{d} \\
        & \bB \ar[swap]{r}{f} & \SpaceOfDiracBundles \times \SpaceOfRamifiedLineBundles
      \end{tikzcd}
    \end{equation*}
    \caption{}
    \label{Fig_B}
  \end{subfigure}
  \caption{Families of local residue conditions.}
  \label{Fig_FamiliesOfLocalResidueConditions}
\end{figure}

For the purposes of this article,
it is very desirable for $\bD_\bR$ to have the following property.

\begin{definition}
  \label{Def_LinearUniformlyFredholm}
  Let
  $E$ and $F$ be tame Fréchet space bundles over a tame Fréchet manifold $B$.
  A map of tame Fréchet space bundles $L \co E \to F$ is \defined{uniformly Fredholm}
  if
  for every $b \in B$
  there exist an open neighbourhood $b \in \sU \subseteq B$,
  finite-rank vector bundles $\DeformationBundle$ and $\ObstructionBundle$ over $\sU$,
  and maps of tame Fréchet space bundles $\pi \co E|_\sU \onto \DeformationBundle$ and $\iota \co \ObstructionBundle \into F|_\sU$ such that the \defined{thickening}
  \begin{equation*}
    \begin{pmatrix}
      L & \iota \\
      \pi & 0
    \end{pmatrix}
    \co E|_\sU \oplus \ObstructionBundle \to F|_\sU \oplus \DeformationBundle
  \end{equation*}
  is an isomorphism of tame Fréchet space bundles.
\end{definition}

\begin{prop}
  \label{Prop_TotallyReal=>UniformlyFredholm}
  Let $(\bB,f,\bV)$ be a family of local residue conditions such that
  for every $b \in \bB$, $x \in Z_b$, and $v \in T_xZ_b \setminus\set{0}$
  \begin{equation*}
    \bV_{(b,x)} \oplus J_{\bp_b}\gamma(v) \bV_{(b,x)} = \ResidueBundle_{\bp_b,x}.
  \end{equation*}
  Denote by $(\bB,f,\bR)$ the corresponding family of residue conditions.
  If $\fa \co f^*\BundleOfSingularSpinors \to f^*\BundleOfSingularSpinors$ is a map of tame Fréchet bundles of order zero,
  then
  \begin{equation*}
    \UniversalDiracOperator_\bR + \fa \co \BundleOfAdmissibleSingularSpinors_\bR \to f^*\BundleOfSingularSpinors
  \end{equation*}
  is uniformly Fredholm.
  Moreover,
  for every $b_0 \in \bB$
  the maps $\pi \co \BundleOfAdmissibleSingularSpinors_\bR|_\sU \to \DeformationBundle$ and $\iota \co \ObstructionBundle \into f^*\BundleOfSingularSpinors|_\sU$ in \autoref{Def_LinearUniformlyFredholm} can be chosen such that there is a compact subset $K \subseteq X$ such that for every
  $b \in \sU$
  the following hold:
  \begin{enumerate}
  \item
    $K \subseteq X\setminus Z_b$;
  \item
    if $\phi \in \BundleOfAdmissibleSingularSpinors_{\bR,b}$ has $\supp(\phi) \cap K = \emptyset$, then $\pi(\phi) = 0$; and
  \item  
    $\supp(\iota(o)) \subseteq K$ for every $o \in \ObstructionBundle_b$.
  \end{enumerate}  
\end{prop}

\begin{proof}%[Proof of \autoref{Prop_TotallyReal=>UniformlyFredholm}]
  By \cite[Proposition 4.41]{BeraWalpuski2025},
  for every $b \in \bB$ the residue condition $\bR_b$ is elliptic;
  in particular:
  for every $k \in \N_0$ the operator
  \begin{equation*}
    D_{\bR_b}^{(k)} + \fa_b^{(k)} \coloneq D_{\bp_b,\bR_b}^{\fl_b,(k)} + \fa_b^{(k)} \co H_a^{k+1}\Gamma\paren{X\setminus Z_b, S\otimes\fl_b;\bR_b;\bp_b} \to H_b^k\Gamma\paren{X\setminus Z_b, S\otimes\fl_b;\bp_b}
  \end{equation*}
  is Fredholm, and $\ker \paren{D_{\bR_b}^{(k)} + \fa_b^{(k)}}$ and $\coker \paren{D_{\bR_b}^{(k)} + \fa_b^{(k)}}$ are independent of $k \in \N_0$.
  
  Let $b_0 \in \bB$.
  Choose the following:
  \begin{enumerate}[label=\rm{(\Roman*)},ref=\Roman*]
  \item
    a compact subset $K \subseteq X \setminus Z_{b_0}$ with non-empty interior,
  \item
    an open neighbourhood $b_0 \in \sU \subseteq \bB$ such that $K \subseteq X \setminus Z_b$ for every $b \in \sU$,
  \item
    finite rank Euclidean vector bundles $\DeformationBundle$ and $\ObstructionBundle$ over $\sU$ and maps of tame Fréchet space bundles $\pi \co \BundleOfAdmissibleSingularSpinors_\bR|_\sU \onto \DeformationBundle$ and $\iota \co \ObstructionBundle \into f^*\BundleOfSingularSpinors|_\sU$ such that for every $b \in \sU$:
    \begin{enumerate}
    \item
      \label{It_PDefinableOnHa1}
      for every $\phi \in \BundleOfAdmissibleSingularSpinors_{\bR,b}$
      \begin{equation*}
        \Abs{\pi(\phi)} \lesssim \Abs{\phi}_{H_a^1},
      \end{equation*}
    \item
      \label{It_SuppInK}
      \begin{enumerate}
      \item
        if $\supp(\phi) \cap K = \emptyset$,
        then $\pi(\phi) = 0$,
      \item
        $\supp(\iota(o)) \subseteq K$ for every $b \in \sU$ and $o \in \ObstructionBundle_b$,
      \end{enumerate}
    \item
      \label{It_Invertible}
      for every $b \in \sU$ and $k \in \N_0$
      \begin{equation*}
        L_b^{(k)}
        \coloneq
        \begin{pmatrix}
          D_{\bR_b}^{(k)} + \fa_b^{(k)} & \iota_b \\
          \pi_b & 0
        \end{pmatrix}
      \end{equation*}
      is invertible.
    \end{enumerate}
  \end{enumerate}  
  \autoref{It_SuppInK} can be arranged with the help of  \autoref{Prop_TransformToNormalForm}.
  To see that \autoref{It_Invertible} can be achieved,
  observe the following.
  If $L_b^{(0)}$ is invertible,
  then $L_b^{(k)}$ is invertible for every $k \in \N_0$ by elliptic regularity \cite[Theorem 4.17]{BeraWalpuski2025}.  
  Since $b \mapsto L_b^{(0)}$ is a continuous family of bounded operators between Banach spaces (in a suitable trivialisation),
  it suffices to ensure that $L_{b_0}^{(0)}$ is invertible.
  The latter holds provided the induced maps $\ker \paren{D_{\bR_{b_0}}^{(0)} + \fa_{b_0}^{(0)}} \to \DeformationBundle_{b_0}$ and $\ObstructionBundle_{b_0} \to \coker \paren{D_{\bR_{b_0}}^{(0)} + \fa_{b_0}^{(0)}}$ are isomorphisms.
  Because the kernels of both $D_{\bR_{b_0}} + \fa_{b_0}$ and its adjoint have the unique continuation property,
  this can easily be arranged while satisfying the desired support properties for a suitable $K \subseteq X \setminus Z_{b_0}$.

  The argument in \cite[Part II Proofs of Theorem 3.3.3 and Theorem 3.3.5]{Hamilton1982:NashMoser} shows that inverses obtained from \autoref{It_Invertible} assemble into an inverse of the map of tame Fréchet bundles
  \begin{equation*}
    \begin{pmatrix}
      D_\bR + \fa & \iota \\
      \pi & 0
    \end{pmatrix}
    \co
    {\BundleOfAdmissibleSingularSpinors_\bR}|_\sU \oplus \ObstructionBundle
    \to
    f^*\BundleOfSingularSpinors \oplus \DeformationBundle.
    \qedhere
  \end{equation*}
\end{proof}

%%% Local Variables:
%%% mode: latex
%%% TeX-master: "UniversalModuliSpaceOfZ2ZHarmonicSpinors"
%%% ispell-local-dictionary: "british"
%%% End:

\subsection{The universal subbundle of admissible $\Z/2\Z$ spinors}
\label{Sec_UniversalBundleOfAdmissibleZModTwoZSpinors}

Naively, one might hope to use $\BundleOfZModTwoZSpinors$ as the ambient space for the construction of tame Fréchet manifold structures on the universal moduli spaces of $\Z/2\Z$ harmonic spinors introduced in \autoref{Sec_Introduction}.
However, for the arguments in \autoref{Sec_Z2ZHarmonicSpinorsInDimensionThree} it is crucial to resolve the issue raised in \autoref{Rmk_CommutatorLKvD} by working with a smaller ambient tame Fréchet space bundle $\BundleOfAdmissibleZModTwoZSpinors$,
the \defined{universal bundle of admissible $\Z/2\Z$ spinors}.
The following describes the fibres of $\BundleOfAdmissibleZModTwoZSpinors$.

\begin{prop}
  \label{Prop_K_Tame}
  For every $\bp \in \SpaceOfDiracBundles$ and $(Z,[\fl]) \in \SpaceOfRamifiedLineBundles$ the graded Fréchet space $\paren{K^\infty\paren{X\setminus Z,S\otimes\fl;\bp},(\Abs{-}_{K^k})_{k \in \N_0}}$ defined by
  \begin{align*}
    K^\infty\paren{X\setminus Z,S\otimes\fl;\bp}
    &\coloneq
      H_a^\infty\paren{X\setminus Z,S\otimes\fl;0;\bp}
      \cap
      (D_\bp^\fl)^{-1} \WeightLog^{-1} H_a^\infty\paren{X\setminus Z,S\otimes\fl;0;\bp}, \qand \\
    \Abs{\phi}_{K^k}
    &\coloneq
      \Abs{\phi}_{H_b^{k+2}} + \Abs{\nabla \phi}_{H_b^{k+1}} + \Abs{\WeightLog D_\bp^\fl\phi}_{H_b^{k+1}} + \Abs{\nabla \WeightLog D_\bp^\fl\phi}_{H_b^k}
  \end{align*}
  is tame;
  that is, it has the smoothing operator property.
  Here $\WeightLog$ denotes multiplication by $\bracket{\log(r)} = \sqrt{1 + \log(r)^2}$ and $r$ denotes the distance to $Z$.
\end{prop}

The proof uses the following observation.

\begin{prop}
  \label{Prop_K_Estimate}
  For every $\bp \in \SpaceOfDiracBundles$, $(Z,[\fl]) \in \SpaceOfRamifiedLineBundles$,
  $\phi \in K^\infty\paren{X\setminus Z,S\otimes\fl;\bp}$, and $k \in \N_0$
  \begin{equation*}
    \Abs{\WeightLog \phi}_{H_b^{k+2}} + \Abs{r^{-1}\WeightLog \phi}_{H_b^{k+1}} + \Abs{\WeightLog \nabla \phi}_{H_b^{k+1}}
    \lesssim_k
    \Abs{\phi}_{K^k}.    
  \end{equation*}
\end{prop}

\begin{proof}%[Proof of \autoref{Prop_K_Estimate}
  By \cite[Lemma 3.3 and Proposition 4.27]{BeraWalpuski2025},  
  \begin{equation}
    \label{Ex_1/REstimate}
    \Abs{r^{-1}\phi}_{H_b^{k+1}}
    \lesssim_k
    \Abs{\phi}_{H_b^{k+1}} + \Abs{\nabla\phi}_{H_b^{k+1}}
    \lesssim_k
    \Abs{\phi}_{K^k}.
  \end{equation}
  Therefore, since $[D_\bp^\fl,\log(r)] = r^{-1}\gamma(\del_r)$,
  \begin{align*}
    \Abs{\WeightLog\phi}_{H_b^{k+1}} + \Abs{D_\bp^\fl \WeightLog\phi}_{H_b^{k+1}}
    \lesssim_k
    \Abs{\phi}_{H_b^{k+1}} + \Abs{\nabla\phi}_{H_b^{k+1}} + \Abs{\WeightLog D_\bp^\fl\phi}_{H_b^{k+1}} 
    \leq
    \Abs{\phi}_{K^k}.
  \end{align*}
  As a consequence,
  $\WeightLog\phi \in H_a^{k+1}(X\setminus Z,S\otimes\fl)$.
  Since $\phi \in L^\infty(X\setminus Z,S\otimes\fl)$,
  $\res(\WeightLog\phi) = 0$.
  By \cite[Theorem 4.11]{BeraWalpuski2025},
  \begin{equation*}
    \Abs{\WeightLog\phi}_{H_b^{k+2}}
    \lesssim_k
    \Abs{\WeightLog\phi}_{H_b^{k+1}} + \Abs{D_\bp^\fl \WeightLog\phi}_{H_b^{k+1}}
  \end{equation*}
  and $\WeightLog\phi \in H_a^{k+2}(X\setminus Z,S\otimes\fl;0;\bp)$.
  Therefore,
  by \cite[Lemma 3.3 and Proposition 4.27]{BeraWalpuski2025},
  \begin{equation*}
    \Abs{r^{-1}\WeightLog \phi}_{H_b^{k+1}}
    +
    \Abs{\WeightLog\nabla \phi}_{H_b^{k+1}}
    \lesssim_k
    \Abs{r^{-1}\phi}_{H_b^{k+1}} + \Abs{\nabla \WeightLog\phi}_{H_b^{k+1}}
    \lesssim
    \Abs{\phi}_{K^k}.
    \qedhere
  \end{equation*}
\end{proof}

\begin{proof}[Proof of \autoref{Prop_K_Tame}]
  For every $\phi \in \Gamma\paren{X\setminus Z, S\otimes\fl}$ with $\supp \phi \subseteq r^{-1}([0,1/2])$ set
  \begin{equation*}
    \Abs{\phi}_{\mathring K^k}
    \coloneq
    \Abs{\phi}_{H_b^{k+2}} + \Abs{\mathring\nabla \phi}_{H_b^{k+1}} + \Abs{\log(r)\mathring D_\bp^\fl\phi}_{H_b^{k+1}} + \Abs{\mathring\nabla \log(r) \mathring D_\bp^\fl\phi}_{H_b^{k}}.
  \end{equation*}
  Here $\mathring \nabla$ and $\mathring D_\bp^\fl$ are the models for $\nabla$ and $D_\bp^\fl$ constructed in \cite[§3.2]{BeraWalpuski2025}.
  Since $B \coloneq D_\bp^\fl - \mathring D_\bp^\fl \in \DiffOp_b^1(S\otimes \fl)$,
  $[B,\log(r)] = r^{-1}\sigma_B(\rd r)$ and its conormal derivatives are bounded.  
  Therefore,
  for $\phi \in K^\infty\paren{X\setminus Z,S\otimes\fl;\bp}$ with $\supp \phi \subseteq r^{-1}([0,1/2])$,
  by \autoref{Prop_K_Estimate}
  \begin{align*}
    \Abs{\log(r) B\phi}_{H_b^{k+1}}
    \lesssim_{B,k}  
    \Abs{\log(r) \phi}_{H_b^{k+2}}
    \lesssim_{B,k}
    \Abs{\phi}_{K^k}
  \end{align*}
  and
  \begin{align*}
    \Abs{\nabla \log(r) B\phi}_{H_b^k}
    &\lesssim_k
      \Abs{r^{-1} B \phi}_{H_b^k} + \Abs{\log(r) \nabla B\phi}_{H_b^k} \\
    &\lesssim_{B,k}
      \Abs{r^{-1} \phi}_{H_b^{k+1}} + \Abs{\log(r) \nabla \phi}_{H_b^{k+1}}
      \lesssim_k
      \Abs{\phi}_{K^k}.
  \end{align*}
  The above (and similar estimates) imply that  
  \begin{equation*}
    \Abs{\phi}_{\mathring K^k} \lesssim_k \Abs{\phi}_{K^{k}}
    \qandq
    \Abs{\phi}_{K^{k}} \lesssim_k \Abs{\phi}_{\mathring K^k}.
  \end{equation*}

  \autoref{Sec_SmoothingOperators} constructs a family of smoothing operators for $\paren{\Abs{-}_{\mathring K^k} : k \in \N_0}$.
  These combined with the usual smoothing operators away from $Z$ construct the required smoothing operators as, for example, in \cites[Proof of Lemma B.1]{Parker2023:Deformation}[Proof of Proposition 4.32]{BeraWalpuski2025}.
\end{proof}

\begin{prop}
  \label{Prop_K_PerturbationEstimate}
  For every $\bp \in \SpaceOfDiracBundles$, $(Z,[\fl]) \in \SpaceOfRamifiedLineBundles$,
  $\Phi \in K^\infty(X\setminus Z, S\otimes\fl;\bp)$, and $\bar v \in \Vect(X)$
  \begin{equation*}
    \sL_{\bar v}^KD_\bp^\fl \Phi
    \in
    H_b^\infty\paren{X\setminus Z, S\otimes\fl;\bp};
  \end{equation*}
  in fact, for every $k,\ell \in \N_0$ with $\ell > d/2$, and $\epsilon > 0$,
  there is a constant $c_k(\epsilon) < \infty$ such that
  \begin{equation*}
    \Abs{\sL_{\bar v}^KD_\bp^\fl \Phi}_{H_b^k}
    \leq
    \epsilon \Abs{\bar v}_{H^{k+1}}\Abs{\Phi}_{K^{\ell}}
    + c_k(\epsilon) \Abs{\bar v}_{H^1}\Abs{\Phi}_{K^{k+\ell}}.
  \end{equation*}
\end{prop}

\begin{proof}%[Proof of \autoref{Prop_K_PerturbationEstimate}]
  Evidently,
  \begin{align*}
    \Abs{\sL_{\bar v}^K D_\bp^\fl \Phi}_{H_b^k}
    &\lesssim_k
      \Abs{\nabla \bar v \otimes D_\bp^\fl \Phi}_{H_b^k}
      +
      \Abs{\bar v \otimes \nabla D_\bp^\fl \Phi}_{H_b^k}
    \\
    &
      =
      \Abs{\nabla \bar v \otimes D_\bp^\fl \Phi}_{H_b^k}
      +
      \Abs{r^{-1} \WeightLog^{-1} \bar v \otimes r \WeightLog \nabla D_\bp^\fl \Phi}_{H_b^k}.
  \end{align*}
  The summands on the right-hand side are estimated individually.
  
  By \cite[Proposition 4.26]{BeraWalpuski2025},  
  $H_a^\ell\Gamma\paren{X\setminus Z, S\otimes\fl;0,\bp} \subseteq L^\infty\Gamma\paren{X\setminus Z, S\otimes\fl;\bp}$.
  Moreover,
  by definition of $\Abs{-}_{K^k}$ and \autoref{Ex_1/REstimate},
  \begin{equation*}
    \Abs{D_\bp^\fl \Phi}_{H_a^{k+1}}
    \lesssim_k
    \Abs{\Phi}_{K^k}.
  \end{equation*}  
  Therefore,
  \begin{equation*}
    \Abs{\nabla \bar v \otimes D_\bp^\fl \Phi}_{H_b^k}
    \lesssim_k
    \sum_{m=0}^k
    \Abs{\nabla \bar v}_{H^{k-m}}\Abs{D\Phi}_{H_a^{m+\ell}}
    \lesssim_k
    \sum_{m=0}^k
    \Abs{\bar v}_{H^{k-m+1}}\Abs{\Phi}_{K^{m+\ell-1}}.
  \end{equation*}
  
  %%% 
  By (the proof of) the borderline $L^2$ Hardy inequality in dimension two,
  \begin{equation*}
    \Abs{r^{-1}\WeightLog^{-1}\bar v}_{H_b^k} \lesssim_k \Abs{\bar v}_{H^{k+1}};
  \end{equation*}
  see, e.g., \cite[Proposition 3.1]{DegeratuStern2013}.
  Moreover, by \cite[Proposition 4.28]{BeraWalpuski2025},
  $H_b^\ell\Gamma\paren{X\setminus Z, S\otimes\fl;\bp} \subseteq r^{-1}L^\infty\Gamma\paren{X\setminus Z, S\otimes\fl;\bp}$.
  Finally, by definition of $\Abs{-}_{K^k}$ and \autoref{Ex_1/REstimate},
  \begin{equation*}
    \Abs{\WeightLog \nabla D_\bp^\fl \Phi}_{H_b^k}
    \lesssim_k
    \Abs{\Phi}_{K^k}.
  \end{equation*}   
  Therefore,
  \begin{equation*}
    \Abs{r^{-1} \WeightLog^{-1} \bar v \otimes r \WeightLog \nabla D_\bp^\fl \Phi}_{H_b^k}
    \lesssim_k
    \sum_{m=0}^k
    \Abs{\bar v}_{H^{k-m+1}}\Abs{\WeightLog \nabla D_\bp^\fl \Phi}_{H_b^{m+\ell}}
    \lesssim_k
    \sum_{m=0}^k
    \Abs{\bar v}_{H^{k-m+1}}\Abs{\Phi}_{K^{m+\ell}}.
  \end{equation*}

  Combining the above estimates,
  \begin{equation*}
    \Abs{\sL_{\bar v}^K D_\bp^\fl \Phi}_{H_b^k}
    \lesssim_k
    \sum_{m=0}^k
    \Abs{\bar v}_{H^{k-m+1}}\Abs{\Phi}_{K^{m+\ell}}.
  \end{equation*}     
  This implies the asserted estimate using interpolation as in \cites[§4.2]{Hamilton1979:DeformationOfComplexStructuresII}[Lemma A.4]{DonaldsonLehmann2025:CY3Boundary}.   
\end{proof}

\begin{prop}[{cf.~\cite[Remark 4.45]{BeraWalpuski2025}}]
  \label{Prop_AdmissibleZ2ZSpinor_LeadingOrderTerm}
  Let $(\bp,[\fl]) \in \SpaceOfDiracBundles \times \SpaceOfRamifiedLineBundles$ and set $Z \coloneq \Br(\fl)$.
  Choose a local infinitesimal uniformiser as in \autoref{Prop_InfinitesimalSliceUniformisers}.
  For every $\Phi \in K^\infty(X\setminus Z, S\otimes\fl;\bp)$ and $v \in T_{[\fl]}\SpaceOfRamifiedLineBundles = \Gamma\paren{Z,NZ}$  
  \begin{equation*}
    \sL_{\bar v}^K \Phi \in H_a^\infty\Gamma\paren{X\setminus Z, S\otimes\fl;\bp};
  \end{equation*}
  moreover,
  for every $\Phi \in K^\infty(X\setminus Z, S\otimes\fl;\bp)$
  there is a unique $\bL_\Phi \in \Gamma\paren{Z,\Hom_\AlgebraBundle(\overline{NZ},\ResidueBundle_\bp)}$
  such that for every $v \in T_{[\fl]}\SpaceOfRamifiedLineBundles$
  \begin{equation*}
    \res(\sL_{\bar v}^K\Phi) = \bL_\Phi (v).
  \end{equation*}
\end{prop}

\begin{proof}%[Proof of \autoref{Prop_AdmissibleZ2ZSpinor_LeadingOrderTerm}]
  By \autoref{Rmk_CommutatorLKvD} and \autoref{Prop_K_Estimate},
  $\nabla_v^\SpaceOfRamifiedLineBundles\bD(\bp,\fl;\Phi) \in H_b^\infty\Gamma\paren{X\setminus Z, S\otimes \fl; \bp}$.
  Therefore, by \autoref{Prop_CommutatorLKvD} and \autoref{Prop_K_PerturbationEstimate},
  $\sL_{\bar v}^K \Phi \in H_a^\infty\Gamma\paren{X\setminus Z, S\otimes\fl;\bp}$.
  
  A moment's thought reveals that $v \mapsto \res(\sL_{\bar v}^K\Phi)$ is independent of the choice of infinitesimal uniformiser.
  Moreover, it is $C^\infty(Z)$--linear and, therefore, defines $\bL_\Phi \in \Gamma\paren{Z,\Hom\paren{NZ,\ResidueBundle_\bp}}$.
  Direct inspection of $D_\bp^\fl$, or rather the model operator $\mathring D_\bp^\fl$, shows that $\bL_\Phi$ is $\AlgebraBundle$--anti-linear.  
\end{proof}

By the method explained in \autoref{Sec_UniversalBundlesOfSingularSpinorsAndResidues},
the above fibres assemble into the following tame Fréchet space bundle.

\begin{definition}
  \label{Prop_BundleOfAdmissibleZModTwoZSpinors}
  Consider the tame Fréchet space bundle
  \begin{equation*}
    \BundleOfCoAdmissibleZModTwoZSpinors \coloneq \WeightLog^{-1}\BundleOfZModTwoZSpinors.
  \end{equation*}
  The \defined{universal bundle of admissible $\Z/2\Z$ spinors} is the tame Fréchet subbundle
  \begin{equation*}
    \BundleOfAdmissibleZModTwoZSpinors
    \coloneq
    \BundleOfZModTwoZSpinors \cap \bD^{-1}\BundleOfCoAdmissibleZModTwoZSpinors
    \subseteq
    \BundleOfZModTwoZSpinors;
  \end{equation*}
  indeed, the preceding discussion and direct inspection show that the atlas constructed in \autoref{Def_UniversalBundles_Atlases} induces an atlas for $\BundleOfAdmissibleZModTwoZSpinors$;
  cf.~\autoref{Rmk_BundleOfSingularSpinors}~\autoref{Rmk_BundleOfSingularSpinors_Robust}.
  Of course, the tame Fréchet space structure on each fibre of $\BundleOfAdmissibleZModTwoZSpinors$ is the one defined in \autoref{Prop_K_Tame}.
  Denote by
  \begin{equation*}
    \BundleOfNonDegenerateAdmissibleZTwoZSpinors \subseteq \BundleOfAdmissibleZModTwoZSpinors
  \end{equation*}
  the open subset of \defined{non-degenerate} admissible $\Z/2\Z$ spinors;
  that is: the subset consisting of those admissible $\Z/2\Z$ spinors $(\bp,[\fl];\Phi) \in \BundleOfAdmissibleZModTwoZSpinors$ for which $\bL_\Phi$ is injective.
\end{definition}

The crucial observation is that every non-degenerate admissible $\Z/2\Z$ harmonic spinor defines a residue condition.

\begin{prop}
  \label{Prop_AdmissibleZ2ZSpinor_NonDegenerate}
  Assume that $\rk_\AlgebraBundle\ResidueBundle = 2$.
  Denote by $f \co \BundleOfNonDegenerateAdmissibleZTwoZSpinors \subseteq \BundleOfAdmissibleZModTwoZSpinors \to \SpaceOfDiracBundles \times \SpaceOfRamifiedLineBundles$ the projection map and by
  \begin{equation*}
    \bV \subseteq (p^*f)^* \UniversalResidueBundle
  \end{equation*}
  the subbundle such that for every $(\bp,[\fl];\Phi) \in \BundleOfNonDegenerateAdmissibleZTwoZSpinors$
  \begin{equation*}
    \bV_{\bp,[\fl];\Phi} = \im \bL_\Phi.
  \end{equation*}
  $(\BundleOfNonDegenerateAdmissibleZTwoZSpinors,f,\bV)$ is a family of local residue conditions that satisfies the hypothesis of \autoref{Prop_TotallyReal=>UniformlyFredholm}.
\end{prop}

\begin{proof}%[Proof of \autoref{Prop_AdmissibleZ2ZSpinor_NonDegenerate}]
  Since $\bL_\Phi$ is $\AlgebraBundle$--anti-linear and $\dim \im \bL_\Phi = \frac12 \rk\ResidueBundle$,
  this follows from \cite[Example 4.44]{BeraWalpuski2025}.
\end{proof}

%%% Local Variables:
%%% mode: latex
%%% TeX-master: "UniversalModuliSpaceOfZ2ZHarmonicSpinors"
%%% ispell-local-dictionary: "british"
%%% End:

\subsection{Chiral adaptations}
\label{Sec_ChiralAdaptations}

Throughout this subsection,
assume that $X$ is oriented and $\dim X = 0 \pmod 4$.
For every $\bp \in \SpaceOfDiracBundles$ and $[\fl] \in \SpaceOfRamifiedLineBundles$ the volume form induces a chirality operator $\epsilon \coloneq \gamma(\vol_g)$.
This decomposes $S = S^+ \oplus S^-$ and splits the entire discussion so far into positive and negative chiral parts.
Most of this is entirely formal.

\begin{definition}
  Here is the decomposition of the objects constructed in \autoref{Sec_UniversalBundlesOfSingularSpinorsAndResidues}:
  \begin{enumerate}
  \item
    Decompose the tame Fréchet space bundles $\BundleOfAdmissibleSingularSpinors$, $\BundleOfSingularSpinors$, and $\BundleOfResidues$ into tame Fréchet subbundles
    \begin{equation}
      \label{Eq_ChiralDecomposition}
      \BundleOfAdmissibleSingularSpinors = \BundleOfAdmissibleSingularSpinors^+ \oplus \BundleOfAdmissibleSingularSpinors^-, \quad
      \BundleOfSingularSpinors = \BundleOfSingularSpinors^+ \oplus \BundleOfSingularSpinors^-, \qandq
      \BundleOfResidues = \BundleOfResidues^+ \oplus \BundleOfResidues^-
    \end{equation}
    with fibres over $(\bp,[\fl]) \in \SpaceOfDiracBundles \times \SpaceOfRamifiedLineBundles$ given by
    \begin{align*}
      \BundleOfAdmissibleSingularSpinors_{\bp,\fl}^\pm &=  H_a^\infty\Gamma\paren{X\setminus \Br(\fl),S^\pm\otimes\fl;\bp}, \\
      \BundleOfSingularSpinors_{\bp,\fl}^\pm &= H_b^\infty\Gamma\paren{X\setminus \Br(\fl),S^\pm\otimes\fl;\bp}, \qand \\
      \BundleOfResidues_{\bp,\fl}^\pm &= \Gamma\paren{\Br(\fl),\ResidueBundle_\bp^\pm}.
    \end{align*}
  \item
    Decompose the universal Dirac operator $\UniversalDiracOperator \co \BundleOfAdmissibleSingularSpinors \to \BundleOfSingularSpinors$ into the \defined{universal chiral Dirac operators} $\UniversalDiracOperator^\pm \co \BundleOfAdmissibleSingularSpinors^\pm \to \BundleOfSingularSpinors^\mp$ such that
    \begin{equation*}
      \UniversalDiracOperator
      =
      \begin{pmatrix}
        0 & \UniversalDiracOperator^- \\
        \UniversalDiracOperator^+ & 0
      \end{pmatrix}.
    \end{equation*}
  \item
    Decompose the universal residue map $\UniversalResidueMap \co \BundleOfAdmissibleSingularSpinors \to \BundleOfResidues$ into the \defined{universal chiral residue maps} $\UniversalResidueMap^\pm \co \BundleOfAdmissibleSingularSpinors^\pm \to \BundleOfResidues^\pm$.
    \qedhere
  \end{enumerate}
\end{definition}

\begin{figure}[h]
  \centering
  \begin{equation*}
    \begin{tikzcd}
      \BundleOfAdmissibleSingularSpinors^\pm \ar{r}{\UniversalDiracOperator^\pm} \ar{d}[swap]{\UniversalResidueMap^\pm} & \BundleOfSingularSpinors^\mp \\
      \BundleOfResidues^\pm
    \end{tikzcd}
  \end{equation*}
  \caption{The universal chiral Dirac operator and residue map.}
  \label{Fig_UniversalChiralDiracOperatorAndResidueMap}
\end{figure}

The partial covariant derivatives constructed in \autoref{Sec_PartialCovariantDerivatives} respect the above decomposition mostly, but not entirely.

\begin{prop}
  \label{Prop_PartialCovariantDerivatives_Chiral}
  ~
  \begin{enumerate}
  \item
    \label{Prop_PartialCovariantDerivatives_Chiral_Slice}
    The restriction of the slice partial covariant derivatives $\nabla^\Slice$ to $H_\BG^\Slice$ and $V_\nabla$ preserves the splittings \autoref{Eq_ChiralDecomposition}.
  \item
    \label{Prop_PartialCovariantDerivatives_Chiral_Ram}
    In the situation of \autoref{Def_PartialConnectionAlongSpaceOfRamifiedLineBundles},
    the partial covariant derivatives $\nabla^\SpaceOfRamifiedLineBundles$ preserve the splittings
    \autoref{Eq_ChiralDecomposition};
    moreover,
    for every $(\bp,[\fl];\phi) \in \BundleOfAdmissibleSingularSpinors^\pm|_\sU$ and $v \in T_{[\fl]}\SpaceOfRamifiedLineBundles$
    \begin{equation*}
      \nabla_v^\SpaceOfRamifiedLineBundles\UniversalDiracOperator^\pm(\bp,[\fl];\phi)
      =
      [\sL_{\bar v}^K,D_\bp^\fl]\phi.
    \end{equation*}
  \end{enumerate}
\end{prop}

\begin{proof}%[Proof of \autoref{Prop_PartialCovariantDerivatives_Chiral}]
  \autoref{Prop_PartialCovariantDerivatives_Chiral_Slice} follows by direct computation.
  %%%
  By direct inspection,
  $\kappa_v$ defined in \autoref{Def_KosmannLift} commutes with $\epsilon$.
  This implies \autoref{Prop_PartialCovariantDerivatives_Chiral_Ram}.
\end{proof}

\begin{definition}
  (Local) chiral residue conditions can be imposed in families as follows:
  \begin{enumerate}
  \item
    A \defined{family of chiral residue conditions} $(B,f,\bR^\pm)$ consists of:
    \begin{enumerate}
    \item
      a tame Fréchet manifold $B$,
    \item
      a tame smooth map $f \co B \to \SpaceOfDiracBundles \times \SpaceOfRamifiedLineBundles$, and    
    \item
      a tame Fréchet space subbundle $\bR^\pm \subseteq f^*\BundleOfResidues^\pm$.
    \end{enumerate}
    Set
    \begin{equation*}
      \BundleOfAdmissibleSingularSpinors_{\bR^\pm}^\pm \coloneq \paren{f^*\UniversalResidueMap^\pm}^{-1}(\bR^\pm)
    \end{equation*}
    and denote the restriction of $f^*\bD$ to $\BundleOfAdmissibleSingularSpinors_{\bR^\pm}^\pm$ by
    \begin{equation*}
      \bD_{\bR^\pm}^{\pm} \co \BundleOfAdmissibleSingularSpinors_{\bR^\pm}^\pm \to f^*\BundleOfSingularSpinors^\mp.
    \end{equation*}
  \item
    \begin{enumerate}
    \item
      Decompose the universal residue bundle $\UniversalResidueBundle$ as
      \begin{equation*}
        \UniversalResidueBundle = \UniversalResidueBundle^+ \oplus \UniversalResidueBundle^-
      \end{equation*}
      into the \defined{universal chiral residue bundles} $\UniversalResidueBundle^\pm$ over $\UniversalBr$.
    \item
      A \defined{family of local chiral residue conditions} $(\bB,f,\bV^\pm)$ consists of:
      \begin{enumerate}
      \item
        a tame Fréchet manifold $\bB$,
      \item
        a tame smooth map $f \co \bB \to \SpaceOfDiracBundles \times \SpaceOfRamifiedLineBundles$, and
      \item
        a subbundle $\bV^\pm \subseteq (p^*f)^* \UniversalResidueBundle^\pm$.
      \end{enumerate}
      Here $f^*p \co f^*\UniversalBr \to \bB$ is the pullback of the universal branching locus and $p^*f \co f^*\UniversalBr \to \UniversalBr$ is the projection map.
    \item
      A family of local chiral residue conditions $(\bB,f,\bV^\pm)$ defines a family of chiral residue conditions $(\bB,f,\bR^\pm)$
      such that for every $b \in \bB$ 
      \begin{equation*}
        \bR_b^\pm \coloneq \Gamma(Z_b,\bV^\pm|_{Z_b}) \subseteq \Gamma(Z_b,\ResidueBundle_{\bp_b}^\pm)
        \qwithq    
        (\bp_b;Z_b,[\fl_b]) \coloneq f(b).
        \qedhere
      \end{equation*}
    \end{enumerate}
  \end{enumerate}
\end{definition}

It is immediate from the discussion in \autoref{Sec_OnceMoreWithChirality} that the following variant of \autoref{Prop_TotallyReal=>UniformlyFredholm} holds.

\begin{prop}
  \label{Prop_TotallyReal=>UniformlyFredholm_Chiral}
  Let $(\bB,f,\bV^\pm)$ be a family of local chiral residue conditions such that
  for every $b \in \bB$, $x \in Z_b$, and $v \in T_xZ_b \setminus\set{0}$
  \begin{equation*}
    \bV_{(b,x)}^\pm \oplus J_{\bp_b}\gamma(v) \bV_{(b,x)}^\pm = \ResidueBundle_{\bp_b,x}^\pm.
  \end{equation*}
  Denote by $(\bB,f,\bR^\pm)$ the corresponding family of chiral residue conditions.
  If $\fa \co f^*\BundleOfSingularSpinors^\pm \to f^*\BundleOfSingularSpinors^\mp$ is a map of tame Fréchet bundles of order zero,
  then
  \begin{equation*}
    \UniversalDiracOperator_{\bR^\pm}^\pm + \fa \co \BundleOfAdmissibleSingularSpinors_{\bR^\pm}^\pm \to f^*\BundleOfSingularSpinors^\mp
  \end{equation*}
  is uniformly Fredholm.
  Moreover,
  for every $b_0 \in \bB$
  the maps $\pi \co \BundleOfAdmissibleSingularSpinors_{\bR^\pm}^\pm|_\sU \to \DeformationBundle$ and $\iota \co \ObstructionBundle \into f^*\BundleOfSingularSpinors^\mp|_\sU$ in \autoref{Def_LinearUniformlyFredholm} can be chosen such that there is a compact subset $K \subseteq X$ such that for every
  $b \in \sU$
  the following hold:
  \begin{enumerate}
  \item
    $K \subseteq X\setminus Z_b$;
  \item
    if $\phi \in \BundleOfAdmissibleSingularSpinors_{\bR^\pm,b}^\pm$ has $\supp(\phi) \cap K = \emptyset$, then $\pi(\phi) = 0$; and
  \item  
    $\supp(\iota(o)) \subseteq K$ for every $o \in \ObstructionBundle_b$.
    \fakeqed
  \end{enumerate}  
\end{prop}

Finally, it remains to decompose the universal bundle of admissible $\Z/2\Z$ spinors and the resulting family of local residue conditions.

\begin{definition} 
  The tame Fréchet space bundles $\BundleOfAdmissibleZModTwoZSpinors$ and $\BundleOfCoAdmissibleZModTwoZSpinors$ defined in \autoref{Prop_BundleOfAdmissibleZModTwoZSpinors} decompose into
  \begin{equation*}
    \BundleOfAdmissibleZModTwoZSpinors = \BundleOfAdmissibleZModTwoZSpinors^+ \oplus \BundleOfAdmissibleZModTwoZSpinors^-
    \qandq
    \BundleOfCoAdmissibleZModTwoZSpinors =   \BundleOfCoAdmissibleZModTwoZSpinors^+ \oplus \BundleOfCoAdmissibleZModTwoZSpinors^-
  \end{equation*}
  with $\BundleOfAdmissibleZModTwoZSpinors^\pm$ denoting the \defined{universal bundle of admissible chiral $\Z/2\Z$ spinors}.
  The latter contains the open subset
  \begin{equation*}
    \BundleOfNonDegenerateAdmissibleChiralZTwoZSpinors \subseteq \BundleOfAdmissibleZModTwoZSpinors^\pm
  \end{equation*}
  of \defined{non-degenerate} admissible chiral $\Z/2\Z$ spinors.
  For every $(\bp,[\fl];\Phi) \in  \BundleOfNonDegenerateAdmissibleChiralZTwoZSpinors$,
  the map $\bL_\Phi$ constructed in \autoref{Prop_AdmissibleZ2ZSpinor_LeadingOrderTerm} corestricts to
  \begin{equation*}
    \bL_\Phi^\pm \in \Gamma\paren{Z,\Hom_\AlgebraBundle(\overline{NZ},\ResidueBundle_\bp^\pm)}. 
  \end{equation*}
  The images define a family of local chiral residue conditions
  \begin{equation*}
    (\BundleOfNonDegenerateAdmissibleChiralZTwoZSpinors,f,\bV^\pm).
  \end{equation*}
  If $\rk_\AlgebraBundle\ResidueBundle^\pm = 2$, this family satisfies the hypothesis of \autoref{Prop_TotallyReal=>UniformlyFredholm_Chiral}.
\end{definition}

%%% Local Variables:
%%% mode: latex
%%% TeX-master: "UniversalModuliSpaceOfZ2ZHarmonicSpinors"
%%% ispell-local-dictionary: "british"
%%% End:

%%% Local Variables:
%%% mode: latex
%%% TeX-master: "UniversalModuliSpaceOfZ2ZHarmonicSpinors"
%%% ispell-local-dictionary: "british"
%%% End:

\section{The constructions of the tame Fréchet manifold structures}
\label{Sec_ConstructionOfTameFrechetmanifoldStructures}

Using the framework laid out in \autoref{Sec_UniversalDiracOperatorAndResidueMap},
this section constructs tame Fréchet manifold structures on $\ProjectiveSpaceOfNonDegenerateHarmonicSpinors$, $\ProjectiveSpaceOfNonDegenerateEigenSpinors$, and $\ProjectiveSpaceOfNonDegenerateChiralHarmonicSpinors$, thereby establishing \autoref{Thm_ProjectiveSpaceOfNonDegenerateHarmonicSpinors}, \autoref{Thm_ProjectiveSpaceOfNonDegenerateEigenSpinors}, and \autoref{Thm_ProjectiveSpaceOfNonDegenerateChiralHarmonicSpinors}.
The final two subsections derive \autoref{Thm_ProjectiveSpaceOfNonDegenerateHarmonicOneForms} and \autoref{Thm_ProjectiveSpaceOfNonDegenerateHarmonicSelfDualTwoAndOneForms}.

\subsection{Uniformly Fredholm maps}
\label{Sec_UniformlyFredholmMaps}

The results to be proved in this section refer to the following notion.

\begin{definition}
  \label{Def_UniformlyFredholmMap}
  Let $\bX,\bY$ be tame Fréchet manifolds.
  A tame smooth map $f \co \bX \to \bY$ is \defined{uniformly Fredholm} if for every $x \in \bX$ there are
  a tame Fréchet space $\bF$,
  a chart $\phi \co U \to \tilde U \subseteq \bF \times \R^k$ of $\bX$ with $x \in U$,
  a chart $\psi \co V \to \tilde V \subseteq \bF \times \R^\ell$ of $\bY$ with $f(x) \in V$, and
  a smooth map $F \co \tilde U \to \R^\ell$ such that
  \begin{equation*}
    \psi \circ f \circ \phi^{-1}(y,z) = (y,F(y,z)).
    \qedhere
  \end{equation*}
  In this case, $k - \ell$ is independent of the choice of charts and defines the \defined{index} of $f$ as a locally constant map
  \begin{equation*}
    \ind f \co \bX \to \Z.
  \end{equation*}
\end{definition}

\begin{remark}
  \label{Rmk_UniformlyFredholmMap}
  A uniformly Fredholm map for which one can always choose $\ell = 0$ is nothing but a \defined{uniform submersion} in the sense of \cite[p.~2]{Gloeckner2016:Submersions}.%
  \footnote{%
    \citeauthor{Gloeckner2016:Submersions} does not use the adjective \emph{uniform}.
  }
  Variations of the above notion are discussed in \cites[§2.2.3]{Diez2019:PhDThesis}[§3.3]{DiezRudolph2022:NormalForm}.
\end{remark}

The significance of uniformly Fredholm maps arises from the fact that they are compatible with transverse base change.

\begin{prop}%[uniform Fredholmness is stable under transverse base change] 
  \label{Prop_UniformFredholmMapsAreStableUnderTransverseBaseChange}
  Let $\bX$, $\bY$, and $\bZ$ be tame Fréchet manifolds,
  $f \co \bX \to \bY$ uniformly Fredholm and
  $g \co \bZ \to \bY$ tame smooth.
  If $f$ and $g$ are \defined{transverse};
  that is: for every $(x,z) \in \bX \times \bZ$ with $f(x) = g(z) \eqcolon y$
  the map $T_xf + T_zg \co T_x\bX \oplus T_z\bZ \to T_y\bY$ is surjective;
  then:
  \begin{enumerate}
  \item
    the fibre product
    \begin{equation*}
      g^*\bX \coloneq \set[\big]{ (x,z) \in \bX \times \bZ : f(x) = g(z) }
    \end{equation*}
    is a tame Fréchet submanifold, and
  \item
    the map $g^*f \co g^*\bX \to \bZ$ is uniformly Fredholm and $\ind g^*f = \paren{\ind f} \circ f^*g$.
  \end{enumerate}
\end{prop}

\begin{proof}
  It suffices to consider the situation in which
  $\bF$ is a tame Fréchet space,
  $\bX$ is an open subset of $\bF \times \R^k$,
  $\bY$ is an open subset of $\bF \times \R^\ell$, and
  $f(x_0,x_1) = (x_0,F(x_0,x_1))$.
  By direct inspection,
  \begin{align*}
    g^*\bX
    =
    \set[\big]{
    (z,x_0,x_1) \in \bZ \times \bF \times \R^k
    :
    x_0 = \pr_1 \circ g(z),
    F(\pr_1 \circ g(z),x_1) = \pr_2 \circ g(z)
    }.
  \end{align*}
  By the transversality assumption,
  at every $(z,x_0,x_1) \in g^*\bX$,
  the derivative of the map
  $(z,x_1) \mapsto F(\pr_1 \circ g(z),x_1) - \pr_2 \circ g(z)$
  is surjective.
  Therefore,
  the assertion follows from the version of the regular value theorem stated in \cite[Part III Theorem 2.3.1]{Hamilton1982:NashMoser}.
\end{proof}

The hypothesis in \autoref{Prop_UniformFredholmMapsAreStableUnderTransverseBaseChange} is generic according to the following parametric transversality result.

\begin{prop}
  \label{Prop_UniformlyFredholm_GenericTransversality}
  Let $\bX$, $\bY$, and $\bP$ be second-countable tame Fréchet manifolds,
  $Z$ a second-countable (finite-dimensional) manifold, and
  $f \co \bX \to \bY$ uniformly Fredholm.
  If $G \co \bP \times Z \to \bY$ is transverse to $f$,
  in particular, if it is a \defined{naive submersion};
  that is:
  $T_{(p,z)}G \co T_p\bP \oplus T_zZ \to T_{G(p,z)}\bY$ is surjective for every $(p,z) \in \bP \times Z$,
  then the subset  
  \begin{equation*}
    \bP^{\pitchfork}
    \coloneq
    \set[\big]{
      p \in \bP
      :
      g_p \coloneq G(p,-) \co Z \to \bY \text{ is transverse to } f
    }
  \end{equation*}
  is comeager.
\end{prop}

The proof relies on the following observation.

\begin{theorem}[{\citet{Smale1965}}]
  \label{Thm_UniformlyFredholm_SardSmale}
  Let $\bX$ and $\bY$ be second-countable tame Fréchet manifolds.
  If $f \co \bX \to \bY$ is uniformly Fredholm,
  then its set of regular values
  \begin{equation*}
    \RegVal(f)
    \coloneq
    \set{
      y \in \bY
      :
      T_xf ~\text{is surjective for every}~ x \in f^{-1}(y)
    }
  \end{equation*}
  is comeager.
\end{theorem}

\begin{proof}%[Proof of \autoref{Thm_UniformlyFredholm_SardSmale}]
  As pointed out by \cite[Remark 3.4]{DiezRudolph2022:NormalForm},
  the proof in \cite{Smale1965} carries over to the present situation.
\end{proof}

\begin{proof}[Proof of \autoref{Prop_UniformlyFredholm_GenericTransversality}]
  By \autoref{Prop_UniformFredholmMapsAreStableUnderTransverseBaseChange},
  $G^*\bX$ is a second-countable tame Fréchet submanifold,
  and $G^*f \co G^*\bX \to \bP \times Z$ is uniformly Fredholm.
  Let $(p,z,x) \in G^*\bX$ and set $y \coloneq f(x)$.
  The Snake Lemma applied to
  \begin{equation*}
    \begin{tikzcd}
      T_zZ \oplus T_x\bX \ar{d}[swap]{T_zg_p - T_xf} \ar[hook]{r} & T_p\bP \oplus T_zZ \oplus T_x\bX \ar[two heads]{d}{T_{(p,z)}G - T_xf} \ar[two heads]{r} & T_p\bP \\
      T_y\bY \ar[equals]{r} & T_y\bY
    \end{tikzcd}
  \end{equation*}
  yields an exact sequence
  \begin{equation*}
    T_{(p,z,x)}G^*\bX \xrightarrow{T_{(p,z,x)}\paren{\pr_\bP \circ G^*f}} T_p\bP \onto \coker \paren{T_zg_p - T_xf}.
  \end{equation*}
  Therefore,  
  $p \in \bP^{\pitchfork}$ if and only if it is a regular value of $\pr_\bP \circ G^*f \co G^*\bX \to \bP$.
  Since $Z$ is finite-dimensional,
  $\pr_\bP \circ G^*f$ is uniformly Fredholm.
  Therefore,
  the assertion follows from \autoref{Thm_UniformlyFredholm_SardSmale}.  
\end{proof}

\begin{remark}
  \label{Rmk_UniformlyFredholm_GenericTransversality}
  There is a version of \autoref{Thm_UniformlyFredholm_SardSmale} for uniformly left semi-Fredholm maps \cite{Quinn1970:LeftSemiFredholmSmale}.
  Using this the hypothesis that $Z$ is finite-dimensional can be dropped.  
\end{remark}

%%% Local Variables:
%%% mode: latex
%%% TeX-master: "UniversalModuliSpaceOfZ2ZHarmonicSpinors"
%%% ispell-local-dictionary: "british"
%%% End:

\subsection{The Nash--Moser submersion theorem}
\label{Sec_NashMoserSubmersionTheorem}

The proofs of \autoref{Thm_ProjectiveSpaceOfNonDegenerateHarmonicSpinors}, \autoref{Thm_ProjectiveSpaceOfNonDegenerateEigenSpinors}, and \autoref{Thm_ProjectiveSpaceOfNonDegenerateChiralHarmonicSpinors} use the following Nash--Moser version of the submersion theorem.

\begin{theorem}
  \label{Thm_NashMoserSubmersionTheorem}
  Let $\bX$ and $\bB$ be tame Fréchet manifolds,
  $\bV$ a tame Fréchet space bundle over $\bX$,
  $\pr_\bB \co \bX \to \bB$ a uniform submersion,
  $s \in \Gamma(\bX,\bV)$ a tame smooth section, and
  $\nabla$ a partial connection on $\bV$ with respect to the vertical tangent bundle $\VerticalTangentBundle{\bX}{\bB} \coloneq \ker T\pr_\bB \subseteq T\bX$.
  If
  there is a finite rank vector bundle $\DeformationBundle$ over $\bX$ and
  a map of tame Fréchet space bundles $\pi \co T\bX/\bB \onto \DeformationBundle$ such that the thickening
  \begin{equation}
    \label{Thm_NashMoserSubmersionTheorem_Isomorphism}
    \begin{pmatrix}
      \nabla s \\
      \pi
    \end{pmatrix}
    \co
    T\bX/\bB \to \bV \oplus \DeformationBundle
  \end{equation}
  is an isomorphism of tame Fréchet space bundles,
  then:
  \begin{enumerate}
  \item
    \label{Thm_NashMoserSubmersionTheorem_ZeroLocus}
    the zero locus
    \begin{equation*}
      \bZ \coloneq s^{-1}(0) \subseteq \bX
    \end{equation*}
    is a tame Fréchet submanifold,
  \item
    \label{Thm_NashMoserSubmersionTheorem_Projection}
    the projection $\pr_\bB \co \bZ \to \bB$ is a uniform submersion, and
  \item
    \label{Thm_NashMoserSubmersionTheorem_VerticalTangentBundle}
    $\VerticalTangentBundle{\bZ}{\bB} = \ker T(\pr_\bB|_{\bZ}) \iso \DeformationBundle|_\bZ$.
  \end{enumerate}
\end{theorem}

\begin{proof}%[Proof of \autoref{Thm_NashMoserSubmersionTheorem}]
  Since the statement is local,
  it suffices to consider the situation in which
  $\bX = \bU \times \bB$ is the product of open neighbourhoods of the origin in two tame Fréchet spaces, and
  $\bV$ is a product bundle, whose fibre we continue to denote by $\bV$.
  In this case,
  \begin{equation*}
    s \co \bU \times \bB \to \bV
  \end{equation*}
  is a tame smooth map,
  and we may further assume that $(0,0) \in s^{-1}(0)$.

  The assumptions imply that $\DeformationBundle \iso \ker(\nabla s)$;
  let $\DeformationBundle_0$ denote the fibre over $(0,0)$, and
  let $\pi_0 \co T\bU \to \DeformationBundle_0$ be a projection to this fibre.
  For an open neighbourhood of the origin $\bK \subseteq \DeformationBundle_0$,
  consider the family of augmented maps
  \begin{equation*}
    \overline s \co \bX \times \bK \to \bV \oplus \DeformationBundle_0
  \end{equation*}
  defined by $\overline s(x,k) \coloneq (s(x), \pi_0(x) - k)$,
  where $\pi_0(x)$ is understood by viewing $\bU \subseteq T\bU$ as a subset.
  The covariant derivative
  \begin{equation*}
    \nabla \overline s \coloneq (\nabla s, \pi_0) \co T\bX/\bB \to \bV \oplus \DeformationBundle_0
  \end{equation*}
  is an isomorphism;
  indeed, viewing $\DeformationBundle \subseteq T\bU$ as $\ker(\nabla s)$ as above,
  $\pi_0 \co \DeformationBundle \to \DeformationBundle_0$ is a map of finite-rank vector bundles that is the identity over the origin,
  hence we may assume it is an isomorphism after shrinking $\bU$, $\bB$, and $\bK$.

  Let $G \co \bV \to T\bU$ denote the restriction of the inverse of \autoref{Thm_NashMoserSubmersionTheorem_Isomorphism} to the first component in the codomain,
  which is a tame smooth family of right inverses of $\nabla s$.
  Then
  \begin{equation*}
    \overline G \co \bV \oplus \DeformationBundle_0 \to T\bX/\bB
    \qwithq
    \overline G_{(u,b)}(v,q) \coloneq G_{(u,b)}v + \pi_0^{-1}\paren[\big]{q - \pi_0 G_{(u,b)}v}
  \end{equation*}
  is a tame smooth family of inverses for $\nabla \overline s$.
  Moreover,
  \begin{equation*}
    (T\overline s)_{(u,b)} \overline G_{(u,b)}(v,q)
    =
    (v,q) - \paren[\big]{\Gamma_{(u,b)}\paren[\big]{\overline G_{(u,b)}(v,q)}s, 0}.
  \end{equation*}
  Thus $\overline G$ is an inverse of $T\overline s$ up to an $\overline s$--quadratic error term.
  \cite[Part III Theorem 3.3.4]{Hamilton1982:NashMoser} then applies to show that there is a tame smooth map $\Phi \co \bB \times \bK \to \bX$ whose value at $(b,k)$ is the unique solution to
  \begin{equation*}
    s(\Phi(b,k)) = 0
    \qandq
    \pi_0\paren[\big]{\pr_{\bU}\Phi(b,k)} = k
  \end{equation*}
  in $\bX$.
  By construction,
  $(\pr_\bB, \pi_0 \circ \pr_{\bU}) \circ \Phi = \id$,
  hence $\Phi$ provides a chart on $\bZ = s^{-1}(0)$ which is already of the form in \autoref{Def_UniformlyFredholmMap}.
  This shows \autoref{Thm_NashMoserSubmersionTheorem_ZeroLocus} and     \autoref{Thm_NashMoserSubmersionTheorem_Projection}.
  The same composition in the second factor alone shows that $T\Phi(b,-) \co \DeformationBundle_0 \to T\bZ/\bB$ is an isomorphism,
  which shows \autoref{Thm_NashMoserSubmersionTheorem_VerticalTangentBundle}.
\end{proof}

%%% Local Variables:
%%% mode: latex
%%% TeX-master: "UniversalModuliSpaceOfZ2ZHarmonicSpinors"
%%% ispell-local-dictionary: "british"
%%% End:

\subsection{The universal space of non-degenerate \texorpdfstring{$\Z/2\Z$}{Z/2Z} harmonic spinors}
\label{Sec_Z2ZHarmonicSpinorsInDimensionThree}

This subsection proves \autoref{Thm_ProjectiveSpaceOfNonDegenerateHarmonicSpinors};
that is:
assuming that $\dim X = 3$ and $\rk S = 4$,
the space of non-degenerate $\Z/2\Z$ harmonic spinors
\begin{equation*}
  \SpaceOfNonDegenerateHarmonicSpinors
  \coloneq
  \set*{
    (\bp,[\fl];\Phi) \in \BundleOfZModTwoZSpinors
    :
    D_\bp^\fl \Phi = 0
    ~\text{and}~    
    \Phi ~\text{is non-degenerate}
  }
\end{equation*}
shall be equipped with the structure of a tame Fréchet manifold such that the projection map $\pr_\SpaceOfDiracBundles \co \SpaceOfNonDegenerateHarmonicSpinors \to \SpaceOfDiracBundles$ is uniformly Fredholm of index zero.
Since
\begin{equation*}
  \ProjectiveSpaceOfNonDegenerateHarmonicSpinors = \SpaceOfNonDegenerateHarmonicSpinors/\R^\times,
\end{equation*}
this immediately implies \autoref{Thm_ProjectiveSpaceOfNonDegenerateHarmonicSpinors}.
Of course, throughout the following four subsubsections it is assumed that $\dim X = 3$ and $\rk S = 4$.

\subsubsection{Local thickenings of the universal space of $\Z/2\Z$ harmonic spinors}
\label{Sec_LocalThickeningsOfTheSpaceOfZModTwoZHarmonicSpinors}

The following is the essential step in the proof of \autoref{Thm_ProjectiveSpaceOfNonDegenerateHarmonicSpinors}.

\begin{prop}
  \label{Prop_ThickenedSpaceOfZ2ZHarmonicSpinors_Submanifold+Submersion}
  For every $(\bp_0,[\fl_0];\Phi_0) \in \SpaceOfNonDegenerateHarmonicSpinors$
  there are
  an open subset $\sU \subseteq \BundleOfNonDegenerateAdmissibleZTwoZSpinors$ with $(\bp_0,[\fl_0];\Phi_0) \in \sU$ and
  a finite rank subbundle $\ObstructionBundle \subseteq f^*\BundleOfCoAdmissibleZModTwoZSpinors|_\sU$
  such that:
  \begin{enumerate}
  \item
    \label{Prop_ThickenedSpaceOfZ2ZHarmonicSpinors_Submanifold+Submersion_Open}
    $\sV \coloneq \pr_{\SpaceOfDiracBundles}(\sU) \subseteq \SpaceOfDiracBundles$ is open,
  \item
    \label{Prop_ThickenedSpaceOfZ2ZHarmonicSpinors_Submanifold+Submersion_Submanifold}
    the local thickening
    \begin{equation*}
      \ThickenedSpaceOfNonDegenerateHarmonicSpinors
      \coloneq
      \set[\big]{
        (\bp,[\fl];\Phi) \in \sU
        :
        D_\bp^\fl \Phi \in \ObstructionBundle
      }
      \subseteq \sU
    \end{equation*}
    is a tame Fréchet manifold, and
  \item
    \label{Prop_ThickenedSpaceOfZ2ZHarmonicSpinors_Submanifold+Submersion_Submersion}
    the projection map $\pr_\sV \co \ThickenedSpaceOfNonDegenerateHarmonicSpinors \to \sV$ is a uniform submersion.
  \end{enumerate}
\end{prop}

\begin{proof}%[Proof of \autoref{Prop_ThickenedSpaceOfZ2ZHarmonicSpinors_Submanifold+Submersion}]
  The proof consists of five steps.
  Denote by $(\BundleOfNonDegenerateAdmissibleZTwoZSpinors,f,\bR)$ the residue condition from  \autoref{Prop_AdmissibleZ2ZSpinor_NonDegenerate}.  
  Let $(\bp_0,[\fl_0];\Phi_0) \in \SpaceOfNonDegenerateHarmonicSpinors$.
  
  \begin{step}
    Choice of $\sU$ and $\ObstructionBundle \subseteq \BundleOfCoAdmissibleZModTwoZSpinors|_\sU$.
  \end{step}  

  By \autoref{Prop_TotallyReal=>UniformlyFredholm} and \autoref{Prop_AdmissibleZ2ZSpinor_NonDegenerate},
  there are
  an open neighbourhood $(\bp_0,[\fl_0];\Phi_0) \in \sU \subseteq \BundleOfNonDegenerateAdmissibleZTwoZSpinors$,
  finite rank vector bundles $\DeformationBundle$ and $\ObstructionBundle$ over $\sU$, and
  maps of tame Fréchet space bundles $\pi \co \BundleOfAdmissibleSingularSpinors_\bR|_\sU \onto \DeformationBundle$ and $\iota \co \ObstructionBundle \into {f}^*\BundleOfSingularSpinors|_\sU$ such that
  \begin{equation*}
    \begin{pmatrix}
      \bD_\bR & \iota \\
      \pi & 0
    \end{pmatrix}
  \end{equation*}
  is invertible.
  Moreover,
  \autoref{Prop_TotallyReal=>UniformlyFredholm} guarantees that $\iota$ can be chosen so that it factors through $\jmath \co \ObstructionBundle \into f^*\BundleOfCoAdmissibleZModTwoZSpinors|_\sU$;
  in fact, the $\iota(o)$ are always supported away from $Z$.
  Henceforth, $\ObstructionBundle$ shall be regarded as a subbundle of $f^*\BundleOfCoAdmissibleZModTwoZSpinors|_\sU$.
  
  Since $\pr_{\SpaceOfDiracBundles} \co \BundleOfNonDegenerateAdmissibleZTwoZSpinors \to \SpaceOfDiracBundles$ is a submersion,
  $\sV \coloneq \pr_\SpaceOfDiracBundles(\sU) \subseteq \SpaceOfDiracBundles$ is open;
  therefore, \autoref{Prop_ThickenedSpaceOfZ2ZHarmonicSpinors_Submanifold+Submersion_Open} holds. 
  The task at hand is to prove that,
  after possibly shrinking $\sU$,  \autoref{Prop_ThickenedSpaceOfZ2ZHarmonicSpinors_Submanifold+Submersion_Submanifold} and \autoref{Prop_ThickenedSpaceOfZ2ZHarmonicSpinors_Submanifold+Submersion_Submersion} hold with the help of \autoref{Thm_NashMoserSubmersionTheorem}.
  
  \begin{step}
    Setup for \autoref{Thm_NashMoserSubmersionTheorem}.
  \end{step}

  Consider the tame Fréchet space bundle
  \begin{equation*}
    \bQ
    \coloneq
    {f}^*\BundleOfCoAdmissibleZModTwoZSpinors/\ObstructionBundle
  \end{equation*}
  over $\sU$.
  Define the section $s \in \Gamma\paren{\sU,\bQ}$ by
  \begin{equation*}
    s(\bp,[\fl];\Phi) \coloneq D_\bp^\fl \Phi \pmod{\ObstructionBundle_{\bp,[\fl];\Phi}}.
  \end{equation*}
  Evidently,
  $\SpaceOfHarmonicSpinors^\bt = s^{-1}(0)$.

  The vertical tangent bundle of $\sU$ fits into the short exact sequence
  \begin{equation*}
    {f}^*\BundleOfAdmissibleZModTwoZSpinors
    \into
    \VerticalTangentBundle{\sU}{\SpaceOfDiracBundles}
    \onto
    {f}^*\sF^\SpaceOfRamifiedLineBundles.
  \end{equation*}
  Choose a local infinitesimal uniformiser and define $\nabla^\SpaceOfRamifiedLineBundles$ as in \autoref{Def_PartialConnectionAlongSpaceOfRamifiedLineBundles}.
  This splits the above short exact sequence and defines a partial covariant derivative on $\bQ$ such that,
  for every $(\bp,[\fl];\Phi) \in \sU$
  \begin{equation}
    \label{Eq_NablaS}
    \begin{split}      
      (\nabla s)_{(\bp,[\fl];\Phi)}
      \co
      {\BundleOfAdmissibleZModTwoZSpinors}_{;\bp,[\fl]} \oplus T_{[\fl]}\SpaceOfRamifiedLineBundles
      &\to
        \bQ_{\bp,[\fl];\Phi} \\
      (\phi,v)
      &\mapsto
        D_\bp^\fl \phi + (\nabla_v^\SpaceOfRamifiedLineBundles\bD)(\bp,[\fl];\Phi)
        \pmod{\ObstructionBundle_{\bp,[\fl];\Phi}}.
    \end{split}
  \end{equation}
  It can be arranged that the subsets $\supp {(\nabla_v^\SpaceOfRamifiedLineBundles\bD)(\bp,[\fl];\Phi)}$ are always disjoint from the compact subset in which the elements of $\ObstructionBundle$ are supported.

  The task at hand is to prove that,
  after possibly shrinking $\sU \ni (\bp_0,[\fl_0];\Phi_0)$,
  the thickening
  \begin{equation*}
    \begin{pmatrix}
      \nabla s \\ \pi
    \end{pmatrix}
  \end{equation*}
  is an isomorphism of tame Fréchet space bundles and thus the hypotheses of \autoref{Thm_NashMoserSubmersionTheorem} are satisfied.
  The latter directly implies \autoref{Prop_ThickenedSpaceOfZ2ZHarmonicSpinors_Submanifold+Submersion_Submanifold} and \autoref{Prop_ThickenedSpaceOfZ2ZHarmonicSpinors_Submanifold+Submersion_Submersion}.
  
  \begin{step}
    Reduction to the invertibility of a lift.
  \end{step}

  It is evident from \autoref{Eq_NablaS} that $\nabla s$ lifts to
  \begin{equation*}
    \widetilde{\nabla s} \co f^*\paren{\BundleOfZModTwoZSpinors \oplus \sF^\SpaceOfRamifiedLineBundles} \to \tilde\bQ \coloneq f^*\BundleOfSingularSpinors|_\sU/\ObstructionBundle
  \end{equation*}
  such that the following diagram commutes
  \begin{equation*}
    \begin{tikzcd}
      {f}^*\paren{\BundleOfAdmissibleZModTwoZSpinors \oplus \sF^\SpaceOfRamifiedLineBundles} \ar{r}{\paren{\nabla s, \pi}} \ar[hook]{d} & \bQ \oplus \DeformationBundle \ar[hook]{d} \\
      {f}^*\paren{\BundleOfZModTwoZSpinors \oplus \sF^\SpaceOfRamifiedLineBundles} \ar[swap]{r}{\paren{\widetilde{\nabla s}, \pi}}  & \tilde \bQ \oplus \DeformationBundle.
    \end{tikzcd}
  \end{equation*}
  By \autoref{Prop_K_Estimate} and \autoref{Rmk_CommutatorLKvD},
  for every $(\bp,[\fl];\Phi) \in \BundleOfAdmissibleZModTwoZSpinors$,
  $(\nabla_v\bD)(\bp,[\fl];\Phi) \in {\BundleOfCoAdmissibleZModTwoZSpinors}_{;\bp,[\fl]}$.
  Therefore,
  if $(\phi,v) \in {\BundleOfZModTwoZSpinors}_{;\bp,[\fl]} \oplus \sF^\SpaceOfRamifiedLineBundles_{\bp,[\fl]}$
  and $\psi \in {\BundleOfCoAdmissibleZModTwoZSpinors}_{;\bp,[\fl]}$ satisfy
  \begin{equation*}
    D_\bp^\fl \phi + \nabla_v^\SpaceOfRamifiedLineBundles\bD(\bp,[\fl];\Phi)
    = \psi
    \pmod{\ObstructionBundle_{\bp,[\fl];\Phi}},
  \end{equation*}
  then
  $D_\bp^\fl\phi \in {\BundleOfCoAdmissibleZModTwoZSpinors}_{;\bp,[\fl]}$.
  As a consequence,
  if $\paren{\widetilde{\nabla s}, \pi}$ is an isomorphism of tame Fréchet space bundles,
  then $\paren{\nabla s, \pi}$ is an isomorphism of tame Fréchet space bundles.

  \medskip
  
  The next two steps verify that $\paren{\widetilde{\nabla s}, \pi}$ is an isomorphism of tame Fréchet space bundles,
  after possibly shrinking $\sU$.

  \begin{step}
    A slight change of perspective.
  \end{step}

  Define
  $\Theta_{\ext},\Theta_K
  \co
  f^*\paren{\BundleOfZModTwoZSpinors
  \oplus
  \sF^\SpaceOfRamifiedLineBundles}
  \to
  \BundleOfAdmissibleSingularSpinors_\bR$
  by
  \begin{align*}
    \Theta_{\ext}(\bp,[\fl],\Phi;\phi,v)
    &\coloneq
      \paren{\bp,[\fl];\Phi;\phi - \ext(\bL_\Phi v)}
    \qand \\
    \Theta_K(\bp,[\fl],\Phi;\phi,v)
    &\coloneq
      \paren{\bp,[\fl];\Phi;\phi - \sL_{\bar v}^{K}\Phi}.
  \end{align*}
  By the discussion in \autoref{Sec_FamiliesOfResidueConditions},
  $\Theta_{\ext}$ is an isomorphism of tame Fréchet space bundles with inverse
  \begin{equation*}
    \Theta_{\ext}^{-1}
    \paren{\bp,[\fl],\Phi;\psi}
    \coloneq
    \paren{
      \bp,[\fl],\Phi;
      \psi - \ext(\res(\psi)),
      - \bL_\Phi^{-1}\res(\psi)
    }.
  \end{equation*}
  By direct computation,
  \begin{equation*}
    \Theta_{\ext}^{-1}\circ \Theta_K
    \paren{\bp,[\fl],\Phi;\phi,v}
    =
    \paren{\bp,[\fl],\Phi;\phi - \Psi_\Phi(v),v}
  \end{equation*}
  with
  \begin{equation*}
    \Psi_\Phi(v)
    \coloneq
    \sL_{\bar v}^{K}\Phi-\ext(\bL_\Phi v).
  \end{equation*}
  By construction,
  $\res(\Psi_\Phi(v)) = \res(\sL_{\bar v}^{K}\Phi)-\res(\ext(\bL_\Phi v)) = 0$.
  Evidently, $\Theta_{\ext}^{-1}\circ \Theta_K$ is an isomorphism of tame Fréchet space bundles;
  hence, so is $\Theta_K$.

  From \autoref{Prop_CommutatorLKvD} and \autoref{Eq_NablaS} it follows that
  \begin{equation*}
    \widetilde{\nabla s} \circ \Theta_K^{-1}
    =
    \bD_\bR+\bP
    \pmod{\ObstructionBundle}
  \end{equation*}
  with the perturbation $\bP$ defined by
  \begin{equation*}
    \bP\paren{\bp,[\fl],\Phi;\psi}
    \coloneq
    \paren{
      \bp,[\fl],\Phi;
      \sL_{\bar v_\Phi(\psi)}^K D_\bp^\fl\Phi
    }
    \qandq
    \bar v_\Phi(\psi)
    \coloneq
    \bar v(\bp,[\fl],-\bL_{\Phi}^{-1}\res(\psi)).
  \end{equation*}
    
  \begin{step}
    \label{St_X}
    Proof of invertibility.
  \end{step}

  By construction,
  the thickening
  \begin{equation*}
    \overline\bD_\bR \coloneq
    \begin{pmatrix}
      \bD_\bR \\
      \pi
    \end{pmatrix}
    \pmod{\ObstructionBundle}
  \end{equation*}
  is an isomorphism of tame Fréchet space bundles and its inverse $\overline\bD_\bR^{-1}$ satisfies tame estimates of the form
  \begin{equation*}
    \Abs{\overline{\bD}_\bR^{-1}(\psi,\delta)}_{H_a^{k+1}} \lesssim \Abs{\psi}_{H_b^k} + \Abs{\delta}.
  \end{equation*}
  By \autoref{Prop_K_PerturbationEstimate},
  $\bP$ is a small perturbation in the sense of \cite[Appendix A]{DonaldsonLehmann2025:CY3Boundary}.
  Therefore, as a consequence of \cite[Proposition A.5]{DonaldsonLehmann2025:CY3Boundary},
  $\overline\bD_\bR + \bP$ is invertible.
\end{proof}

%%% Local Variables:
%%% mode: latex
%%% TeX-master: "UniversalModuliSpaceOfZ2ZHarmonicSpinors"
%%% ispell-local-dictionary: "british"
%%% End:

\subsubsection{Transversality of the obstruction map}
\label{Sec_TransversalityOfTheObstructionMap}

\begin{prop}
  \label{Prop_TransversalityOfTheObstructionMap}
  In the situation of \autoref{Prop_ThickenedSpaceOfZ2ZHarmonicSpinors_Submanifold+Submersion}
  the \defined{obstruction map} $\ObstructionMap \co \ThickenedSpaceOfNonDegenerateHarmonicSpinors \to \ObstructionBundle$ defined by
  \begin{equation*}
    \ObstructionMap(\bp,[\fl];\Phi) \coloneq D_\bp^\fl \Phi
  \end{equation*}
  is transverse to the zero section.
\end{prop}

\begin{proof}[Proof of \autoref{Prop_TransversalityOfTheObstructionMap}]
  Let $(\bp,[\fl];\Phi) \in \ob^{-1}(0) \subseteq \SpaceOfNonDegenerateHarmonicSpinors$.
  By \autoref{Prop_SliceDerivativeOfUniversalDiracOperator},
  for every $\dot\nabla \in \Omega^1\paren{X,\fo_\Cl(S)}$,
  $\dot\bp_\nabla = (0;0;\dot\nabla;0) \in T_{(\bp,[\fl])}\Slice$ and
  \begin{equation*}
    \nabla_{\dot\bp_\nabla}^\Slice\bD (\bp,[\fl];\Phi)
    =
    \sum_{i=1}^3 \gamma(e_i)\dot\nabla_{e_i} \Phi.
  \end{equation*}  
  There is a compact subset $K \subseteq X\setminus Z$ with non-empty interior on which $\Phi$ does not vanish.
  By \autoref{Rmk_OClS} and \autoref{Lem_AlgebraInTransversality}, stated after this proof,
  for every $\psi \in \Gamma(X\setminus Z,S\otimes\fl)$ with $\supp(\psi) \subseteq K$ there is a $\dot\nabla$ such that
  \begin{equation*}
    \sum_{i=1}^3 \gamma(e_i)\dot\nabla_{e_i} \Phi = \psi.
  \end{equation*}
  Since $\pr_\SpaceOfDiracBundles \co \SpaceOfHarmonicSpinors^\bt \to \SpaceOfDiracBundles$ is a submersion,
  there is a lift $\Omega^1\paren{X,\fo_{\Cl}(S)} \to T_{(\bp,[\fl];\Phi)}\SpaceOfHarmonicSpinors^\bt$.
  From the above it is clear that the composition
  \begin{equation*}
    \Omega^1\paren{X,\fo_{\Cl}(S)}
    \incl T_{(\bp,[\fl];\Phi)}\SpaceOfHarmonicSpinors^\bt
    \xrightarrow{T_{(\bp,[\fl];\Phi)}\ob}
    T_{(\bp,[\fl];\Phi;0)}\ObstructionBundle
    \onto
    \ObstructionBundle_{(\bp,[\fl];\Phi)}
  \end{equation*}
  is surjective.
  Therefore, $\ob$ is transverse to the zero section.
\end{proof}

\begin{lemma}
  \label{Lem_AlgebraInTransversality}
  For every $\phi \in \H\setminus\set{0}$ and $\psi \in \H$ there are $\xi,\eta,\zeta \in \Im \H$ such that
  \begin{equation*}
    i\phi\xi + j\phi\eta + k\phi\zeta = \psi.
  \end{equation*}
\end{lemma}

\begin{proof}%[Proof of \autoref{Lem_AlgebraInTransversality}]
  By direct inspection,
  it is possible to choose $\tilde\xi,\tilde\eta,\tilde\zeta \in \Im \H$ such that $i\tilde\xi + j\tilde\eta + k\tilde\zeta = \psi\phi^{-1}$.
  Set $\xi \coloneq \phi^{-1}\tilde\xi\phi$, $\eta \coloneq \phi^{-1}\tilde\eta\phi$, $\zeta \coloneq \phi^{-1}\tilde\zeta\phi$.
\end{proof}

\begin{proof}[Proof of \autoref{Thm_ProjectiveSpaceOfNonDegenerateHarmonicSpinors}]
  It is an immediate consequence of
  \autoref{Prop_ThickenedSpaceOfZ2ZHarmonicSpinors_Submanifold+Submersion},
  \autoref{Prop_TransversalityOfTheObstructionMap},
  and \cite[Part~III, Theorem~2.3.1]{Hamilton1982:NashMoser}
  that $\SpaceOfNonDegenerateHarmonicSpinors \subseteq \BundleOfNonDegenerateAdmissibleZTwoZSpinors$ is a tame Fréchet submanifold and that $\pr_\SpaceOfDiracBundles \co \SpaceOfNonDegenerateHarmonicSpinors \to \SpaceOfDiracBundles$ is uniformly Fredholm.

  By \cite[Corollary~2.30]{BeraWalpuski2025},
  the family of residue conditions introduced in \autoref{Prop_AdmissibleZ2ZSpinor_NonDegenerate} is Lagrangian.
  As a consequence $\rk \DeformationBundle = \rk \ObstructionBundle$ and, therefore, $\pr_\SpaceOfDiracBundles$ has index zero.
  This directly implies \autoref{Thm_ProjectiveSpaceOfNonDegenerateHarmonicSpinors}.
\end{proof}

\begin{remark}
  \label{Rmk_TangentSpaces}
  The proof of \autoref{Thm_ProjectiveSpaceOfNonDegenerateHarmonicSpinors} identifies the vertical tangent spaces of $\SpaceOfNonDegenerateHarmonicSpinors$:
  indeed, if $(\bp,[\fl];\Phi) \in \SpaceOfNonDegenerateHarmonicSpinors$,
  then identifying ${\BundleOfZModTwoZSpinors}_{;\bp,[\fl]} \oplus T_{[\fl]}\SpaceOfRamifiedLineBundles = {\BundleOfAdmissibleSingularSpinors}_{\bR;\bp,[\fl];\Phi}$
  \begin{equation*}
    \ker T_{(\bp,[\fl];\Phi)}\pr_\SpaceOfDiracBundles
    =
    \ker D_{R_\Phi}.
    \qedhere
  \end{equation*}
  Of course,
  the corresponding vertical tangent space of $\ProjectiveSpaceOfNonDegenerateHarmonicSpinors$ is $\ker D_{R_\Phi}/\R\Span{\Phi}$.
\end{remark}

%%% Local Variables:
%%% mode: latex
%%% TeX-master: "UniversalModuliSpaceOfZ2ZHarmonicSpinors"
%%% ispell-local-dictionary: "british"
%%% End:

\subsubsection{A transversality criterion}
\label{Sec_TransversalityCriterion}

Here is how to verify whether or not a smooth map $\Gamma \co B \to \SpaceOfDiracBundles$ is transverse to $\pr_\SpaceOfDiracBundles \co \ProjectiveSpaceOfNonDegenerateHarmonicSpinors \to \SpaceOfDiracBundles$.

\begin{definition}
  Let $(\bp,[\fl]) \in \SpaceOfDiracBundles \times \SpaceOfRamifiedLineBundles$ and $\phi,\psi \in \Gamma\paren{X\setminus \Br(\fl), S\otimes \fl}$.
  \begin{enumerate}
  \item
    The \defined{stress-energy tensor}
    $T_{\phi,\psi} \in \Gamma\paren{X\setminus Z,\Hom(S^2TX,\R)}$ is defined by
    \begin{equation*}
      T_{\phi,\psi} (v,w)
      \coloneq
      -\frac18\paren{\Inner{\gamma(v)\nabla_w\phi,\psi} + \Inner{\gamma(w)\nabla_v\phi,\psi}
      + \Inner{\gamma(v)\nabla_w\psi,\phi} + \Inner{\gamma(w)\nabla_v\psi,\phi}}.
    \end{equation*}
  \item  
    Define
    $\mu(\phi,\psi) \in \Omega^1\paren{X\setminus Z,\fo_{\Cl}(S)}$
    by
    \begin{equation*}
      \mu(\phi,\psi)
      \coloneq
      \frac12\bar\gamma^*\paren*{\phi\Inner{\psi,-} + \psi\Inner{\phi,-}}
    \end{equation*}
    with $\bar\gamma^*$ denoting the adjoint of the linear map $\bar \gamma \co T^*X \otimes \fo_{\Cl}(S) \to \Sym(S)$ defined by $\bar\gamma(\alpha\otimes \xi) \coloneq \gamma(\alpha^\sharp)\xi$.
    \qedhere
  \end{enumerate}
\end{definition}

\begin{prop}
  \label{Prop_TransversalityCriterion_HarmonicSpinors}
  Let $B$ be a manifold and \  $\Gamma \co B \to \SpaceOfDiracBundles$ a smooth map.
  The following are equivalent:
  \begin{enumerate}
  \item
    $\Gamma$ is transverse to $\pr_\SpaceOfDiracBundles \co \ProjectiveSpaceOfNonDegenerateHarmonicSpinors \to \SpaceOfDiracBundles$.
  \item
    \label{Prop_TransversalityCriterion_HarmonicSpinors_CoPetri}
    For every $b \in B$ and $(\bp,[\fl];\Phi) \in \SpaceOfNonDegenerateHarmonicSpinors$ with $\Gamma(b) = \bp$
    the map
    \begin{equation*}
      \delta_{\Phi} \co T_bB \to \paren{\ker D_{R_\Phi}}^*
    \end{equation*}
    defined by
    \begin{equation*}
      \delta_\Phi(v)\kappa
      \coloneq
      \Inner*{T_{\Phi,\kappa},T_{\bp}\pr_\SpaceOfMetrics \circ T_b\Gamma(v)}_{L^2}
      + \Inner*{\mu(\Phi,\kappa),\pr_{V_\nabla} \circ T_b\Gamma(v)}_{L^2}
    \end{equation*}    
    is surjective.
    Here $R_\Phi$ denotes the residue condition induced by $\Phi$.
  \end{enumerate}
\end{prop}

\begin{proof}%[Proof of \autoref{Prop_TransversalityCriterion_HarmonicSpinors}]
  Let $b \in B$ and $(\bp,[\fl];\Phi) \in \SpaceOfNonDegenerateHarmonicSpinors$ with $\Gamma(b) = \bp$.
  In the construction of $\SpaceOfHarmonicSpinors^\bt$,
  $\ObstructionBundle$ can be chosen such that the $L^2$ orthogonal projection $\Pi$ onto $\ker D_{R_\Phi}$ induces an isomorphism $\ObstructionBundle_{(\bp,[\fl];\Phi)} \iso \ker D_{R_\Phi}$.
  In particular,
  as a consequence of the construction of $\SpaceOfNonDegenerateHarmonicSpinors$,
  $\Gamma$ and $\pr_\SpaceOfDiracBundles$ are transverse at $b$ and $(\bp,[\fl];\Phi)$ precisely if
  the map 
  \begin{align*}
    T_bB &\to \ker D_{R_\Phi} \\
    v &\mapsto \Pi \nabla_{T_b\Gamma(v)}\bD(\bp,[\fl];\Phi)
  \end{align*}
  is surjective.
  By direct inspection this is equivalent to $\delta_\Phi$ being surjective;
  cf.~\cite[Proof of Proposition 4.1]{Maier1997} or \cite[Proof of Lemma 3.2]{AmmannDahl2025}.
  Evidently, the $V_\sG$ component of $T_b\Gamma(v)$ does not contribute.
\end{proof}

\begin{remark}
  \label{Rmk_TransversalityCriterion_HarmonicSpinors}
  The (dual of the) condition \autoref{Prop_TransversalityCriterion_HarmonicSpinors_CoPetri} is a \defined{Petri condition}; cf.~\cites[§5]{Wendl2016}[Definition 1.1.11]{Doan2018}.
  Here are some comments on the usefulness of the contributions to $\delta_\Phi$ to verify this condition:
  \begin{enumerate}
  \item    
    The map $\mu\paren{\Phi,-} \co \Gamma\paren{X\setminus Z, S\otimes \fl} \to \Omega^1\paren{X\setminus Z,\fo_{\Cl}(S)}$ factors through a linear map $S\otimes \fl \to \Hom\paren{TX, \fo_{\Cl}(S)}|_{X\setminus Z}$ which is injective on the dense open subset $X \setminus \Phi^{-1}(0)$ by \autoref{Lem_AlgebraInTransversality}.
    Of course,
    this is what enabled the proof of \autoref{Prop_TransversalityOfTheObstructionMap} and makes the second contribution to $\delta_\Phi$ quite useful.
  \item
    The map $T_{\Phi,-} \co \Gamma\paren{X\setminus Z, S\otimes\fl} \to \Gamma\paren{X\setminus Z,\Hom(S^2TX,\R)}$ factors through $J^1(S\otimes \fl) \to \Hom(S^2TX,\R)|_{X\setminus Z}$.
    If $T_{\Phi,\Phi} = 0$,
    then $g$ is conformally flat by \cite[Proof of Theorem 1.2]{Maier1997}.
    Therefore,
    the first contribution to $\delta_\Phi$ might not always be useful.
    However, it is quite possible that for generic metrics it is.
    Maybe, the methods developed by \cite{GreilhuberKepplinger2026:SAH} can be brought to bear on this.
    
    The above issue also is what stands in the way of establishing a version of \autoref{Thm_ProjectiveSpaceOfNonDegenerateHarmonicSpinors} for $\Z/2\Z$ harmonic spinors on Dirac bundles arising from a (topological) spin structure on $X$.
    \qedhere
  \end{enumerate}
\end{remark}

%%% Local Variables:
%%% mode: latex
%%% TeX-master: "UniversalModuliSpaceOfZ2ZHarmonicSpinors"
%%% ispell-local-dictionary: "british"
%%% End:

\subsubsection{The universal zero locus}
\label{Sec_UniversalZeroLocus}

Here is an observation to the effect that generic non-degenerate $\Z/2\Z$ harmonic spinors do not vanish outside of the branching locus.
This observation essentially already appears in \cite[Proof of Corollary~1.12]{HeParker2024} and is repeated here for the readers' convenience.

\begin{prop}
  \label{Prop_UniversalZeroLocus_Harmonic}
  The \defined{universal zero locus} of non-degenerate $\Z/2\Z$ harmonic spinors
  \begin{equation*}
    \bZ_{\ProjectiveSpaceOfNonDegenerateHarmonicSpinors}
    \coloneq
    \set[\big]{
      (\bp,[\fl];[\Phi],x) \in \ProjectiveSpaceOfNonDegenerateHarmonicSpinors \times X
      :
      x \notin \Br(\fl),
      \Phi(x) = 0      
    }
  \end{equation*}
  is a tame Fréchet submanifold of $\ProjectiveSpaceOfNonDegenerateHarmonicSpinors \times X$ and $\pr_{\ProjectiveSpaceOfNonDegenerateHarmonicSpinors} \co \bZ_{\ProjectiveSpaceOfNonDegenerateHarmonicSpinors} \to \ProjectiveSpaceOfNonDegenerateHarmonicSpinors$ has index $-1$.
\end{prop}

\begin{proof}%[Proof of \autoref{Prop_UniversalZeroLocus_Harmonic}]
  There is a rank $4$ vector bundle $\sS$ over $\mathring\bX_{\ProjectiveSpaceOfNonDegenerateHarmonicSpinors} \coloneq \set{ (\bp,[\fl];[\Phi],x) \in \ProjectiveSpaceOfNonDegenerateHarmonicSpinors \times X : x \notin \Br(\fl) }$ whose fibre over $(\bp,[\fl];[\Phi],x)$ is $\Hom\paren{\R\Span{\Phi},S_x\otimes\fl_x}$.
  Evaluation defines a section $\ev \in \Gamma\paren{\mathring\bX_{\ProjectiveSpaceOfNonDegenerateHarmonicSpinors},\sS}$.
  If $\ev$ is transverse to the zero section,
  then,
  by \cite[Part~III, Theorem~2.3.1]{Hamilton1982:NashMoser},
  $\bZ_{\ProjectiveSpaceOfNonDegenerateHarmonicSpinors}$ is a tame Fréchet submanifold and $\ind \pr_{\ProjectiveSpaceOfNonDegenerateHarmonicSpinors} = \dim X - \rk \sS = - 1$.

  To prove that $\ev$ is transverse to the zero section,
  it suffices to prove that for every $(\bp,[\fl];[\Phi],x) \in \bZ_{\ProjectiveSpaceOfNonDegenerateHarmonicSpinors}$ and $\psi \in (S\otimes\fl)_x$ there are $\phi \in H_a^\infty\Gamma\paren{X\setminus Z,S\otimes\fl;R_\Phi;\bp}$ and $\dot\bp \in T_\bp\SpaceOfDiracBundles$ such that
  \begin{equation*}
    \phi(x) = \psi \qandq
    D_\bp^\fl\phi + \nabla_{\dot\bp}\bD\paren{\bp,[\fl];\Phi} = 0.
  \end{equation*}
  To prove this,
  it suffices to prove that
  \begin{align*}
    \ker \ev_x \oplus T_\bp\SpaceOfDiracBundles &\to H_b^\infty\paren{X\setminus Z, S\otimes \fl;\bp} \\
    (\phi,\dot\bp) &\mapsto D_\bp^\fl\phi + \nabla_{\dot\bp}\bD\paren{\bp,[\fl];\Phi}
  \end{align*}
  is surjective.
  Here $\ev_x \co H_a^\infty\Gamma\paren{X\setminus Z,S\otimes\fl;R_\Phi;\bp} \to S_x \otimes \fl_x$ denotes the evaluation map.
  The restriction of $D_\bp^\fl$ to $\ker \ev_x$ is Fredholm and $\coker D_\bp^\fl|_{\ker \ev_x} \iso \paren{\ker D_{R_\Phi} \oplus D_\bp^\fl \widetilde{S_x\otimes\fl_x}}^*$ for every choice of lift $\widetilde{S_x\otimes\fl_x}$ of $S_x \otimes \fl_x$ along $\ev_x$.
  With this in mind,
  the surjectivity of the above map follows by the argument employed in the proof of \autoref{Prop_TransversalityOfTheObstructionMap}. 
\end{proof}

The above might be somewhat interesting in view of \cite{Doan2017c,Yan2025:Model,Yan2025:Gluing}.

%%% Local Variables:
%%% mode: latex
%%% TeX-master: "UniversalModuliSpaceOfZ2ZHarmonicSpinors"
%%% ispell-local-dictionary: "british"
%%% End:

%%% Local Variables:
%%% mode: latex
%%% TeX-master: "UniversalModuliSpaceOfZ2ZHarmonicSpinors"
%%% ispell-local-dictionary: "british"
%%% End:

\subsection{The universal space of non-degenerate $\Z/2\Z$ eigenspinors}
\label{Sec_Z2ZEigenSpinorsInDimensionThree}

Assuming that $\dim X = 3$ and $\rk S = 4$,
consider the universal space of non-degenerate $\Z/2\Z$ eigenspinors
\begin{equation*}
  \SpaceOfNonDegenerateEigenSpinors
  \coloneq
  \set*{
    (\bp,[\fl];\Phi,\lambda) \in \BundleOfNonDegenerateAdmissibleZTwoZSpinors \times \R
    :
    D_\bp^\fl \Phi = \lambda\Phi
  }.
\end{equation*}
To prove that $\SpaceOfNonDegenerateEigenSpinors$ is a submanifold and $\pr_\SpaceOfDiracBundles \co     \SpaceOfNonDegenerateEigenSpinors \to \SpaceOfDiracBundles$ is uniformly Fredholm of index $1$,
one proceeds as in \autoref{Sec_Z2ZHarmonicSpinorsInDimensionThree}.

\begin{prop}
  \label{Prop_ThickenedSpaceOfZ2ZEigenSpinors_Submanifold+Submersion}
  For every $(\bp_0,[\fl_0];\Phi_0,\lambda_0) \in \SpaceOfNonDegenerateEigenSpinors$
  there are
  an open subset $\sU \subseteq \BundleOfNonDegenerateAdmissibleZTwoZSpinors \times \R$ with $(\bp_0,[\fl_0];\Phi_0,\lambda_0) \in \sU$ and
  a finite rank subbundle $\ObstructionBundle \subseteq \pr_1^*f^*\BundleOfCoAdmissibleZModTwoZSpinors|_{\sU}$
  such that:
  \begin{enumerate}
  \item
    \label{Prop_ThickenedSpaceOfZ2ZEigenSpinors_Submanifold+Submersion_Open}
    $\sV \coloneq \pr_{\SpaceOfDiracBundles}(\sU) \subseteq \SpaceOfDiracBundles$ is open,
  \item
    \label{Prop_ThickenedSpaceOfZ2ZEigenSpinors_Submanifold+Submersion_Submanifold}
    the local thickening
    \begin{equation*}
      \ThickenedSpaceOfNonDegenerateEigenSpinors
      \coloneq
      \set[\big]{
        (\bp,[\fl];\Phi,\lambda) \in \sU
        :
        (D_\bp^\fl-\lambda) \Phi \in \ObstructionBundle
      }
      \subseteq \sU
    \end{equation*}
    is a tame Fréchet manifold, and
  \item
    \label{Prop_ThickenedSpaceOfZ2ZEigenSpinors_Submanifold+Submersion_Submersion}
    the projection map $\pr_\sV \co \ThickenedSpaceOfNonDegenerateEigenSpinors \to \sV$ is a uniform submersion.
  \end{enumerate}
\end{prop}

\begin{proof}
  The proof of \autoref{Prop_ThickenedSpaceOfZ2ZHarmonicSpinors_Submanifold+Submersion} carries over with very minor modifications.
  In particular,
  $\ObstructionBundle$ is constructed using \autoref{Prop_TotallyReal=>UniformlyFredholm} applied to $\UniversalDiracOperator - \lambda$ and $s$ is defined by
  \begin{equation*}
    s(\bp,[\fl];\Phi,\lambda) \coloneq \paren{D_\bp^\fl-\lambda}\Phi \pmod{\ObstructionBundle_{\bp,[\fl];\Phi,\lambda}}.
    \qedhere
  \end{equation*}
\end{proof}

\begin{prop}
  \label{Prop_TransversalityOfTheObstructionMap_Eigen}
  In the situation of \autoref{Prop_ThickenedSpaceOfZ2ZEigenSpinors_Submanifold+Submersion}
  the \defined{obstruction map} $\ObstructionMap \co \ThickenedSpaceOfNonDegenerateEigenSpinors \to \ObstructionBundle$ defined by
  \begin{equation*}
    \ObstructionMap(\bp,[\fl];\Phi,\lambda) \coloneq (D_\bp^\fl-\lambda) \Phi
  \end{equation*}
  is transverse to the zero section.
  \qed
\end{prop}

\begin{proof}[Proof of \autoref{Thm_ProjectiveSpaceOfNonDegenerateEigenSpinors}]
  As a consequence of
  \autoref{Prop_ThickenedSpaceOfZ2ZEigenSpinors_Submanifold+Submersion},
  \autoref{Prop_TransversalityOfTheObstructionMap_Eigen},
  and \cite[Part~III, Theorem~2.3.1]{Hamilton1982:NashMoser},
  it follows that $\SpaceOfNonDegenerateEigenSpinors \subseteq \BundleOfNonDegenerateAdmissibleZTwoZSpinors \times \R$ is a tame Fréchet manifold and that $\pr_\SpaceOfDiracBundles \co \SpaceOfNonDegenerateEigenSpinors \to \SpaceOfDiracBundles$ is uniformly Fredholm of index $1$.
  This implies the assertion.
\end{proof}

Here is the analogue of the transversality criterion \autoref{Prop_TransversalityCriterion_HarmonicSpinors}.

\begin{prop}
  \label{Prop_TransversalityCriterion_EigenSpinors}
  Let $B$ be a manifold and $\Gamma \co B \to \SpaceOfDiracBundles$ a smooth map.
  The following are equivalent:
  \begin{enumerate}
  \item
    $\Gamma$ is transverse to $\pr_\SpaceOfDiracBundles \co \ProjectiveSpaceOfNonDegenerateEigenSpinors \to \SpaceOfDiracBundles$.
  \item
    For every $b \in B$ and $(\bp,[\fl];[\Phi],\lambda) \in \ProjectiveSpaceOfNonDegenerateEigenSpinors$ with $\Gamma(b) = \bp$
    the map
    \begin{equation*}
      \delta_{\Phi,\lambda} \co T_bB \times \R \to \paren[\big]{\ker \paren{D_{R_\Phi}-\lambda}}^*
    \end{equation*}
    defined by
    \begin{equation*}
      \delta_{\Phi,\lambda}(v,t)\kappa
      \coloneq
      \Inner*{T_{\Phi,\kappa},T_{\bp}\pr_\SpaceOfMetrics \circ T_b\Gamma(v)}_{L^2}
      + \Inner*{\mu(\Phi,\kappa),\pr_{V_\nabla} \circ T_b\Gamma(v)}_{L^2}
      - t\Inner*{\Phi,\kappa}_{L^2}
    \end{equation*}    
    is surjective.
    Here $R_\Phi$ denotes the residue condition induced by $\Phi$.
    \qed
  \end{enumerate}
\end{prop}

\begin{remark}
  It should be observed that $\tr_g T_{\Phi,\kappa} = -\frac12\lambda \Inner{\Phi,\kappa}$ and, therefore, $T_{\Phi,-}$ cannot vanish if $\lambda \neq 0$.
\end{remark}

Finally, as in \autoref{Prop_UniversalZeroLocus_Harmonic},
generic non-degenerate $\Z/2\Z$ eigenspinors do not vanish outside of the branching locus.

\begin{prop}
  \label{Prop_UniversalZeroLocus_Eigen}
  The \defined{universal zero locus} of non-degenerate $\Z/2\Z$ eigenspinors
  \begin{equation*}
    \bZ_{\ProjectiveSpaceOfNonDegenerateEigenSpinors}
    \coloneq
    \set[\big]{
      (\bp,[\fl];[\Phi],\lambda,x) \in \ProjectiveSpaceOfNonDegenerateEigenSpinors \times X
      :
      x \notin \Br(\fl),
      \Phi(x) = 0      
    }
  \end{equation*}
  is a tame Fréchet submanifold of $\ProjectiveSpaceOfNonDegenerateEigenSpinors \times X$ and $\pr_{\ProjectiveSpaceOfNonDegenerateEigenSpinors} \co \bZ_{\ProjectiveSpaceOfNonDegenerateEigenSpinors} \to \ProjectiveSpaceOfNonDegenerateEigenSpinors$ has index $-1$.
  \qed
\end{prop}

Finally,
it should be pointed out that it is straightforward to verify that zero is a regular value of the projection map $\lambda \co \ProjectiveSpaceOfNonDegenerateEigenSpinors \to \R$.

%%% Local Variables:
%%% mode: latex
%%% TeX-master: "UniversalModuliSpaceOfZ2ZHarmonicSpinors"
%%% ispell-local-dictionary: "british"
%%% End:

\subsection{The universal space of non-degenerate chiral \texorpdfstring{$\Z/2\Z$}{Z/2Z} harmonic spinors}
\label{Sec_Z2ZHarmonicSpinorsInDimensionFour}

\begin{proof}[Proof of \autoref{Thm_ProjectiveSpaceOfNonDegenerateChiralHarmonicSpinors}]
  The proof of \autoref{Thm_ProjectiveSpaceOfNonDegenerateHarmonicSpinors} carries over with cosmetic modifications, using \autoref{Prop_TotallyReal=>UniformlyFredholm_Chiral} instead of \autoref{Prop_TotallyReal=>UniformlyFredholm}.
  The index formula is a consequence of \autoref{Thm_IndexFormula}.
\end{proof}

Here is the analogue of the transversality criterion \autoref{Prop_TransversalityCriterion_HarmonicSpinors}.

\begin{prop}
  \label{Prop_TransversalityCriterion_HarmonicSpinors_4D}
  Let $B$ be a manifold and $\Gamma \co B \to \SpaceOfDiracBundles$ a smooth map.
  The following are equivalent:
  \begin{enumerate}
  \item
    $\Gamma$ is transverse to $\pr_\SpaceOfDiracBundles \co \ProjectiveSpaceOfNonDegenerateChiralHarmonicSpinors \to \SpaceOfDiracBundles$.
  \item
    \label{Prop_TransversalityCriterion_HarmonicSpinors_4D_CoPetri}
    For every $b \in B$ and $(\bp,[\fl];\Phi) \in \SpaceOfNonDegenerateChiralHarmonicSpinors$ with $\Gamma(b) = \bp$
    the map
    \begin{equation*}
      \delta_{\Phi} \co T_bB \to \paren[\big]{\ker D_{R_\Phi^\mp}^\mp}^*
    \end{equation*}
    defined by
    \begin{equation*}
      \delta_\Phi(v)\kappa
      \coloneq
      \Inner*{T_{\Phi,\kappa},T_{\bp}\pr_\SpaceOfMetrics \circ T_b\Gamma(v)}_{L^2}
      + \Inner*{\mu(\Phi,\kappa),\pr_{V_\nabla} \circ T_b\Gamma(v)}_{L^2}
    \end{equation*}    
    is surjective.
    Here $R_{\Phi}^\pm$ denotes the chiral residue condition induced by $\Phi$.
    \qed
  \end{enumerate}
\end{prop}

\begin{remark}
  The reader should beware that the codomain of $\delta_\Phi$ is the dual of the kernel of the \emph{adjoint} of $D_{R_\Phi^\pm}^\pm$.
  In the situation considered in \autoref{Prop_TransversalityCriterion_HarmonicSpinors},
  the relevant operator is self-adjoint and this distinction does not arise.
\end{remark}

%%% Local Variables:
%%% mode: latex
%%% TeX-master: "UniversalModuliSpaceOfZ2ZHarmonicSpinors"
%%% ispell-local-dictionary: "british"
%%% End:

\subsection{The universal space of non-degenerate \texorpdfstring{$\Z/2\Z$}{Z/2Z} harmonic $1$--forms}
\label{Sec_Z2ZHarmonicOneForms}

This subsection proves \autoref{Thm_ProjectiveSpaceOfNonDegenerateHarmonicOneForms}.
Throughout this subsection, $\dim X = 3$, $X$ is oriented, and $S = \underline \R \oplus T^*X$.
The map $\HdR \co \SpaceOfMetrics \to \SpaceOfDiracBundles$,
alluded to before the statement of \autoref{Thm_ProjectiveSpaceOfNonDegenerateHarmonicOneForms},
assigns to every Riemannian metric $g \in \SpaceOfMetrics$ the Dirac bundle structure on $S$ with respect to which the Clifford multiplication is given by
\begin{equation*}
  \gamma(v)(f,\alpha) \coloneq \paren{ -i_v\alpha, v^\flat f + *\paren{v^\flat \wedge \alpha} }
\end{equation*}
and the spin connection is induced by the Levi-Civita connection.
The pullback
\begin{equation*}
  \SpaceOfNonDegenerateHarmonicOneForms
  \coloneq
  \HdR^*\SpaceOfNonDegenerateHarmonicSpinors
\end{equation*}
is the \defined{universal space of non-degenerate $\Z/2\Z$ harmonic $1$--forms} on $X$.
Indeed,
if $(g,[\fl];f,\alpha) \in \SpaceOfNonDegenerateHarmonicOneForms$,
then
\begin{equation}
  \label{Eq_ZeroForm_IntegrationByParts}
  \rd^*\rd f = \rd^*(\rd f + *\rd\alpha) = 0
\end{equation}
and, therefore, $\rd f = 0$ by integration by parts and using that $(f,\alpha)$ has vanishing residue.
Since $\Br(\fl) \neq \emptyset$, an assumption that is tacitly made, $f = 0$.
Therefore $([\fl],\alpha)$ is a $\Z/2\Z$ harmonic $1$--form with respect to $g$.
Conversely and evidently, if $([\fl],\alpha)$ is a non-degenerate $\Z/2\Z$ harmonic $1$--form with respect to $g$,
then $(g,[\fl];0,\alpha) \in \SpaceOfNonDegenerateHarmonicOneForms$.

Here is some preparation regarding the finite rank vector bundle $\sH^1 \to \SpaceOfMetrics \times \SpaceOfRamifiedLineBundles$ mentioned in \autoref{Thm_ProjectiveSpaceOfNonDegenerateHarmonicOneForms}.

\begin{prop}
  \label{Prop_ExistenceOfBranchedDoubleCover}
  For every $[\fl] \in \SpaceOfRamifiedLineBundles$ the double cover
  $\mathring{\pi} \co \tilde X^\circ \coloneq \set[\big]{ v \in \fl : \abs{v} = 1 } \to X \setminus Z$ extends to a smooth branched double cover $\pi \co \tilde X \to X$.
\end{prop}

\begin{proof}%[Proof of \autoref{Prop_ExistenceOfBranchedDoubleCover}]
  The Euclidean line bundle $\fl$ determines a double cover $\widetilde{SNZ}$ of the unit sphere bundle $SNZ$.
  The latter is the boundary of the unit disc bundle $DNZ$.
  The question is whether $\widetilde{SNZ}$ also is the boundary of a disc bundle $\widetilde{DNZ}$ that is a branched double cover of $DNZ$.
 The obstruction to the existence of $\widetilde{DNZ}$ is $w_2(NZ) \in \rH^2\paren{Z,\set{\pm 1}}$.
  The latter, however, is the restriction $\PD([Z])|_Z$ and vanishes by \autoref{Rmk_ModuliSpaceOfRamifiedEuclideanLineBundles_Topology}.
\end{proof}

\begin{definition}[{cf.~\cite[Example 3.48]{BeraWalpuski2025}}]
  \label{Def_OneFormDirichletResidueCondition}
  For every $(g,[\fl]) \in \SpaceOfMetrics \times \SpaceOfRamifiedLineBundles$ the residue bundle $\ResidueBundle$ over $Z \coloneq \Br(\fl)$ contains the $\AlgebraBundle$--linear subbundle
  \begin{equation*}
    \ResidueBundle_D \coloneq N^*Z \otimes_\AlgebraBundle NZ^{-1/2}.
  \end{equation*}
  These form a family of local residue conditions $\paren{\SpaceOfMetrics \times \SpaceOfRamifiedLineBundles,\HdR \times \id_\SpaceOfRamifiedLineBundles,\UniversalResidueBundle_D}$.  
  Denote the associated family of residue conditions by
  $\paren{\SpaceOfMetrics \times \SpaceOfRamifiedLineBundles,\HdR \times \id_\SpaceOfRamifiedLineBundles,\bR_D}$.
\end{definition}

This is an elliptic Lagrangian family of local residue conditions.
As a consequence of \citeauthor{Teleman1983:SignatureLipschitz}'s Lipschitz Hodge theory \cite{Teleman1983:SignatureLipschitz},
the associated kernels form a vector bundle.

\begin{prop}
  \label{Prop_KerDRD}
  For every $(g,[\fl]) \in \SpaceOfMetrics \times \SpaceOfRamifiedLineBundles$ the following hold:
  \begin{enumerate}
  \item
    \label{Prop_KerDRD_L2Harmonic}
    The kernel of $D_{g,R_D}^\fl$ consists precisely of the $L^2$ harmonic $\fl$--twisted $1$--forms;
    that is:
    \begin{equation*}
      \ker D_{g,R_D}^\fl
      =
      \sH^1(g,[\fl])
      \coloneq
      \set[\big]{
        \alpha \in L^2\Omega^1(X\setminus Z,\fl)
        :
        \rd \alpha = \rd^*\alpha = 0
      }
    \end{equation*}
  \item
    \label{Prop_KerDRD_DeRham}
    Pulling back along $\pi \co \tilde X \to X$ induces an isomorphism
    \begin{equation*}
      \sH^1(g,[\fl]) \iso \rH_\dR^1(\tilde X)^-.
    \end{equation*}
    Here $\rH_\dR^1(\tilde X)^-$ denotes the subspace of $\rH_\dR^1(\tilde X)$ which is anti-invariant under the sheet-swapping involution $\tau$.
  \end{enumerate}
\end{prop}

\begin{proof}%[Proof of \autoref{Prop_KerDRD}] 
  If $(f,\alpha)\in\ker D_{g,R_D}^\fl$,
  then $f = 0$ by the argument following \autoref{Eq_ZeroForm_IntegrationByParts}.
  This proves that $\ker D_{g,R_D}^\fl \subseteq \sH^1(g,[\fl])$.
  Conversely,
  if $\alpha\in\sH^1(g,[\fl])$,
  then by \cite[(4.1)]{Donaldson2021} it has an expansion of the form
  \begin{equation*}
    \alpha
    =
    \Re(A z^{-1/2}\rd z) + \text{higher order terms}.
  \end{equation*}
  Here $z$ is a local complex normal coordinate and, globally, $A \in \Gamma\paren{Z,NZ^{-1/2}}$.
  Therefore, $(0,\alpha) \in H_a^\infty\paren{X\setminus \Br(\fl),S\otimes\fl}$ and $\res(0,\alpha)\in R_D$.
  This proves \autoref{Prop_KerDRD_L2Harmonic}.
  
  The proof of \autoref{Prop_KerDRD_DeRham} is a variation of the proof of \cite[Proposition 5]{Wang1993:Involutions}.  
  The metric $\pi^*g$ degenerates along the ramification locus $\pi^{-1}\paren{\Br(\fl)}$.
  This is a consequence of the fact that near the ramification locus $\pi$ is locally modelled on $(x,z) \mapsto (x,z^2)$.
  There is a Lipschitz map $\sigma \co \tilde X \to X$ which is:
  \begin{enumerate*}
  \item
    modelled on $(x,z) \mapsto (x,z^2/\abs{z})$ near the ramification locus,
  \item
    invariant under $\tau$,
  \item
    homotopic to $\pi$,
    and
  \item
    such that $\tilde g \coloneq \sigma^*g$ is a Lipschitz Riemannian metric in the sense of \cite[§3]{Teleman1983:SignatureLipschitz}.
  \end{enumerate*}
  A moment's thought shows that $\sH^1(g,[\fl])$ pulls back under $\sigma$ precisely to the space of $\tau$--anti-invariant $L^2$ harmonic forms with respect to $\tilde g$ in the sense of \cite[(4.4)]{Teleman1983:SignatureLipschitz}.
  Therefore, the assertion follows from the version of the Hodge theorem stated in  \cite[Theorem~4.1]{Teleman1983:SignatureLipschitz}.
\end{proof}

\begin{cor}
  \label{Cor_H1Bundle}
  The vector spaces $\sH^1(g,[\fl])$ form a finite rank subbundle of $\paren{\HdR\times\id_\SpaceOfRamifiedLineBundles}^*\BundleOfAdmissibleSingularSpinors \to \SpaceOfMetrics \times \SpaceOfRamifiedLineBundles$.
\end{cor}

It turns out that the local residue conditions defined in \autoref{Prop_AdmissibleZ2ZSpinor_NonDegenerate} and
\autoref{Def_OneFormDirichletResidueCondition} are related as follows.

\begin{prop}
  \label{Prop_SpinorResidueConditon=DirichletResidueCondition}
  If $(g,[\fl];\alpha) \in \SpaceOfNonDegenerateHarmonicOneForms$,
  then
  \begin{equation*}
    \im \bL_\alpha = \ResidueBundle_D.
  \end{equation*}
\end{prop}

\begin{proof}%[Proof of \autoref{Prop_SpinorResidueConditon=DirichletResidueCondition}]
  As explained in \cite[(4.1)]{Donaldson2021},
  $\alpha$ has an expansion of the form
  \begin{equation*}
    \alpha = \Re(B z^{1/2} \rd z) + \text{higher order terms}.
  \end{equation*}
  Here $z$ is a local complex normal coordinate and, globally, $B \in \Gamma\paren{Z,\Hom_\AlgebraBundle\paren{\overline{NZ},NZ^* \otimes_\AlgebraBundle NZ^{-1/2}}}$.
  A moment's thought shows that up to a constant factor $B$ is nothing but $\bL_\alpha$.
  As a consequence,
  $\im \bL_\alpha \subseteq \ResidueBundle_D$.
  Since $\rk \bL_\alpha = \rk \ResidueBundle_D$,
  this inclusion is an identity.
\end{proof}

\begin{proof}[Proof of \autoref{Thm_ProjectiveSpaceOfNonDegenerateHarmonicOneForms}]
  Since $\HdR$ is not transverse to $\pr_\SpaceOfDiracBundles \co \SpaceOfNonDegenerateHarmonicSpinors \to \SpaceOfDiracBundles$,
  \emph{a priori} $\SpaceOfNonDegenerateHarmonicOneForms$ might not be a tame Fréchet submanifold of $\HdR^*\BundleOfAdmissibleZModTwoZSpinors$.
  
  Let $(g_0,[\fl_0];\alpha) \in \SpaceOfNonDegenerateHarmonicOneForms$.
  Consider a local thickening $\ThickenedSpaceOfNonDegenerateHarmonicSpinors$ containing $(\HdR(g_0),[\fl_0];\alpha)$ as in \autoref{Prop_ThickenedSpaceOfZ2ZHarmonicSpinors_Submanifold+Submersion}.
  Since $\pr_\sV \co \ThickenedSpaceOfNonDegenerateHarmonicSpinors \to \sV$ is a uniform submersion,
  $\HdR^*\ThickenedSpaceOfNonDegenerateHarmonicSpinors \subseteq \HdR^*\BundleOfNonDegenerateAdmissibleZTwoZSpinors$ is a tame Fréchet submanifold and the projection map $\HdR^*\ThickenedSpaceOfNonDegenerateHarmonicSpinors \to \HdR^{-1}(\sV)$ is a uniform submersion.
  Evidently,  
  $\HdR^*\ThickenedSpaceOfNonDegenerateHarmonicSpinors \supseteq
  \SpaceOfNonDegenerateHarmonicOneForms \cap \HdR^*\sU$.

  By the preceding propositions the thickening $\ThickenedSpaceOfNonDegenerateHarmonicSpinors$ can be chosen such that for every $(g,[\fl];\alpha) \in \HdR^*\ThickenedSpaceOfNonDegenerateHarmonicSpinors$ the $L^2$ orthogonal projection $\ObstructionBundle_{g,[\fl];\alpha} \to \sH^1(g,[\fl])$ is injective.
  By construction,
  if $(g,[\fl];\alpha) \in \HdR^*\ThickenedSpaceOfNonDegenerateHarmonicSpinors$,
  then
  \begin{equation*}
    \rd^*\alpha = 0 \qandq *\rd\alpha \in \ObstructionBundle_{g,[\fl];\alpha}.
  \end{equation*}
  However, the $L^2$ orthogonal projection of $*\rd\alpha$ onto $\sH^1(g,[\fl])$ vanishes by integration by parts, since $\alpha$ has vanishing residue.
  As a consequence the obstruction map $\HdR^*\ob$ vanishes.
  Therefore, $\HdR^*\ThickenedSpaceOfNonDegenerateHarmonicSpinors$ is an open neighbourhood of $(g_0,[\fl_0];\alpha)$ in $\SpaceOfNonDegenerateHarmonicOneForms$.
  This proves \autoref{Thm_ProjectiveSpaceOfNonDegenerateHarmonicOneForms_Submersion}.

  In light of \autoref{Rmk_TangentSpaces},
  \autoref{Thm_ProjectiveSpaceOfNonDegenerateHarmonicOneForms_VerticalTangenBundle} holds as a consequence of \autoref{Prop_KerDRD} and \autoref{Prop_SpinorResidueConditon=DirichletResidueCondition}.
\end{proof}

%%% Local Variables:
%%% mode: latex
%%% TeX-master: "UniversalModuliSpaceOfZ2ZHarmonicSpinors"
%%% ispell-local-dictionary: "british"
%%% End:

\subsection{The universal space of non-degenerate \texorpdfstring{$\Z/2\Z$}{Z/2Z} harmonic self-dual $2$--forms}
\label{Sec_Z2ZHarmonicSelfDualFormsInDimensionFour}

This subsection proves \autoref{Thm_ProjectiveSpaceOfNonDegenerateHarmonicSelfDualTwoAndOneForms}.
Throughout this subsection, $\dim X = 4$, $X$ is oriented, and $S = T^*X \oplus \underline \R \oplus \Lambda^+ T^*X$.
The map $\HdR \co \SpaceOfMetrics \to \SpaceOfDiracBundles$,
alluded to before the statement of \autoref{Thm_ProjectiveSpaceOfNonDegenerateHarmonicSelfDualTwoAndOneForms},
assigns to every Riemannian metric $g \in \SpaceOfMetrics$ the Dirac bundle structure on $S$ with respect to which the Clifford multiplication is given by
\begin{equation*}
  \gamma(v)\paren{\alpha;f,\beta}
  \coloneq
  \paren{ v^\flat f - \sqrt{2}i_v\beta; -i_v\alpha, \sqrt{2}\paren{v^\flat \wedge \alpha}^+ }
\end{equation*}
and the spin connection is induced by the Levi-Civita connection.
The pullbacks
\begin{equation*}
   \ProjectiveSpaceOfNonDegenerateHarmonicOneForms \coloneq \HdR^*\ProjectiveSpaceOfNonDegeneratePositiveHarmonicSpinors
   \qandq
   \ProjectiveSpaceOfNonDegenerateHarmonicSelfDualTwoForms \coloneq 
   \HdR^*\ProjectiveSpaceOfNonDegenerateNegativeHarmonicSpinors
\end{equation*}
are the \defined{universal moduli spaces of non-degenerate $\Z/2\Z$ harmonic $1$--forms} and \defined{self-dual $2$--forms} respectively.

\begin{definition}[{cf.~\cite[Example 3.48]{BeraWalpuski2025}}]
  \label{Def_OneFormAndSelfDualTwoFormDirichletResidueCondition}
  For every $(g,[\fl]) \in \SpaceOfMetrics \times \SpaceOfRamifiedLineBundles$ the chiral residue bundle $\ResidueBundle^\pm$ over $Z \coloneq \Br(\fl)$ contains the $\AlgebraBundle$--linear subbundle
  \begin{equation*}
    \ResidueBundle_D^+    
    \coloneq
    N^*Z \otimes_\AlgebraBundle NZ^{-1/2}
    \quad\text{or}\quad
    \ResidueBundle_D^-
    \coloneq
    (T^*Z \wedge N^*Z)^+ \otimes_\AlgebraBundle NZ^{-1/2}
  \end{equation*}
  respectively.
  These form a family of local chiral residue conditions $\paren{\SpaceOfMetrics \times \SpaceOfRamifiedLineBundles,\HdR \times \id_\SpaceOfRamifiedLineBundles,\UniversalResidueBundle_D^\pm}$.  
  Denote the associated family of chiral residue conditions by
  $\paren{\SpaceOfMetrics \times \SpaceOfRamifiedLineBundles,\HdR \times \id_\SpaceOfRamifiedLineBundles,\bR_D^\pm}$.
\end{definition}

$\ResidueBundle_D^+ \oplus \ResidueBundle_D^-$ is an elliptic Lagrangian family of local residue conditions and the associated kernels form vector bundles.

\begin{prop}
  \label{Prop_KerDRD+-}
  For every $(g,[\fl]) \in \SpaceOfMetrics \times \SpaceOfRamifiedLineBundles$ the following hold:
  \begin{enumerate}
  \item
    \label{Prop_KerDRD+-_L2Harmonic}
    The kernel of $D_{g,R_D^\pm}^\fl$ consists precisely of the $L^2$ harmonic $\fl$--twisted $1$--forms and self-dual $2$--forms respectively;
    that is:
    \begin{align*}
      \ker D_{g,R_D^+}^\fl
      &=
        \sH^1(g,\fl)
        \coloneq
        \set[\big]{
        \alpha \in L^2\Omega^1(X\setminus Z,\fl)
        :
        \rd \alpha = \rd^*\alpha = 0
        } \qand \\
      \ker D_{g,R_D^-}^\fl
      &=
        \sH^+(g,\fl)
        \coloneq
        \set[\big]{
        \beta \in L^2\Omega^+(X\setminus Z,\fl)
        :
        \rd \beta = \rd^*\beta = 0
        }.
    \end{align*}
  \item
    \label{Prop_KerDRD^+-_DeRham}
    Pulling back along $\pi \co \tilde X \to X$ induces isomorphisms
    \begin{equation*}
      \sH^1(g,\fl) \iso \rH_\dR^1(\tilde X)^-
      \qandq
      \sH^+(g,\fl) \iso \rH_\dR^+(\tilde X)^-.
    \end{equation*}
    Here $\rH_\dR^\bullet(\tilde X)^-$ denotes the subspace of $\rH_\dR^\bullet(\tilde X)$ which is anti-invariant under the sheet-swapping involution $\tau$ and $\rH_\dR^+(\tilde X)^- \subseteq \rH_{\dR}^2(\tilde X)^-$ denotes \emph{a} maximal subspace on which the intersection form is positive-definite.
  \end{enumerate}
\end{prop}

\begin{proof}%[Proof of \autoref{Prop_KerDRD+-}]
  The proof is analogous to that of \autoref{Prop_KerDRD} using the fact that a $\beta \in \sH^+(g,\fl)$ has an expansion of the form
  \begin{equation*}
    \beta = \Re\paren{A z^{-1/2} \rd w \wedge \rd z} + \text{higher order terms}.
  \end{equation*}
  Here $w$ is a local complex coordinate on $Z$ and $z$ is a local complex normal coordinate.  
\end{proof}

\begin{cor}
  \label{Cor_H1+Bundle}
  The vector spaces $\sH^1(g,\fl)$ and $\sH^+(g,\fl)$ form finite rank subbundles of $\paren{\HdR\times\id_\SpaceOfRamifiedLineBundles}^*\BundleOfAdmissibleSingularSpinors^\pm \to \SpaceOfMetrics \times \SpaceOfRamifiedLineBundles$ respectively.
\end{cor}

\begin{prop}
  \label{Prop_SpinorResidueConditon=DirichletResidueConditionTwoForms}
  If $(g,[\fl];\alpha) \in \SpaceOfNonDegenerateHarmonicOneForms$ (or $(g,[\fl];\beta) \in \SpaceOfNonDegenerateHarmonicSelfDualTwoForms$),
  then
  \begin{equation*}
    \im \bL_\alpha = \ResidueBundle_D^+
    \quad
    (\text{or}\quad\im \bL_\beta = \ResidueBundle_D^-).
  \end{equation*}
\end{prop}

\begin{proof}%[Proof of \autoref{Prop_SpinorResidueConditon=DirichletResidueConditionTwoForms}]
  The proof is analogous to that of \autoref{Prop_SpinorResidueConditon=DirichletResidueCondition} using that
  \begin{equation*}
    \beta = \Re(B z^{1/2} \rd w \wedge \rd z) + \text{higher order terms}.
  \end{equation*}
  Here $w$ is a local complex coordinate on $Z$ and $z$ is a local complex normal coordinate and, globally, $B \in \Gamma\paren{Z,\Hom_\AlgebraBundle\paren{\overline{NZ},\paren{T^*Z \wedge NZ^*}^+ \otimes_\AlgebraBundle NZ^{-1/2}}}$.
\end{proof}

\begin{proof}[Proof of \autoref{Thm_ProjectiveSpaceOfNonDegenerateHarmonicSelfDualTwoAndOneForms}]
  The proof is analogous to that of \autoref{Thm_ProjectiveSpaceOfNonDegenerateHarmonicOneForms} using the above ingredients and cosmetic modifications to track the chirality.
\end{proof}

\begin{remark}
  \label{Rmk_IndexInSelfDualCase}
  If $(g,[\fl];\alpha) \in \SpaceOfNonDegenerateHarmonicOneForms$,
  then the index of $D_{g,R_\alpha^+}^\fl$ can be determined as follows.
  If $X$ is a closed oriented $4$--manifold,
  then
  \begin{equation*}
    b^1(X) - b^0(X) - b^+(X) = -\frac12\paren{\chi(X) + \sigma(X)}.
  \end{equation*}
  By \cite[Lemma 7.1.7]{GompfStipsicz1999:KirbyCalculus},
  \begin{equation*}
    \chi(\tilde X) = 2\chi(X) - \chi(Z) \qandq
    \sigma(\tilde X) = 2\sigma(X) - \frac12 [Z]\cdot[Z].
  \end{equation*}
  If $Z$ is not orientable, then $[Z]\cdot[Z]$ needs to be interpreted accordingly.
  The orientation local systems $\fo$ of $TZ$ and $NZ$ agree.
  Therefore, there is a well-defined Euler class $\fe(NZ) \in \rH^2(Z,\fo)$ and $[Z] \in \rH_2(Z,\fo)$ and both can be paired.

  As a consequence of \autoref{Prop_KerDRD+-}, \autoref{Prop_SpinorResidueConditon=DirichletResidueConditionTwoForms}, and the above,
  \begin{align*}
    \ind D_{g,R_\alpha^+}^{\fl,+}
    &=
      \dim \sH^1(g,\fl) - \dim \sH^+(g,\fl) \\
    &=
      -\frac12\paren{\chi(\tilde X) + \sigma(\tilde X)}
      +
      \frac12\paren{\chi(X) + \sigma(X)} \\
    &=    
      -\frac12\paren{\chi(X) + \sigma(X)}
      +
      \frac12\chi(Z) + \frac14 [Z]\cdot[Z].
  \end{align*}
  In fact, since $\alpha$ is non-degenerate, $[Z]\cdot[Z] = 0$.
\end{remark}

Finally,
here is a straightforward construction of non-degenerate $\Z/2\Z$ harmonic self-dual $2$--forms on Kähler surfaces.

\begin{example}
  Let $X$ be a Kähler surface and $q \in \rH^0\paren{X,K_X^{\otimes 2}}$ transverse to zero.
  The latter defines a holomorphic branched double cover $\tilde X \coloneq \set{ (x,v) \in K_X : v^2 = q(x) } \to X$.
  This induces a ramified Euclidean line bundle $\fl$ with $Z \coloneq \Br(\fl) = q^{-1}(0)$ and $\Re \sqrt{q} \in \Omega^+(X\setminus Z,\fl)$ is a $\Z/2\Z$ harmonic self-dual $2$--form.

  It is not difficult to find $X$ and $q$ as above.
  By Bertini's theorem,
  if $\abs{2K_X}$ is base-point free, then a general $q \in \rH^0\paren{X,K_X^{\otimes 2}}$ is transverse to zero.
  A concrete example is a smooth hypersurface $X_d \subseteq \CP^3$ of degree $d \geq 5$ and a general section of $K_{X_d}^{\otimes 2} = \sO_{X_d}(2d-8)$.  
\end{example}

By \autoref{Thm_ProjectiveSpaceOfNonDegenerateHarmonicSelfDualTwoAndOneForms},
the $\Z/2\Z$ harmonic self-dual $2$--forms obtained above persist under small deformations of the metric on $X$.

%%% Local Variables:
%%% mode: latex
%%% TeX-master: "UniversalModuliSpaceOfZ2ZHarmonicSpinors"
%%% ispell-local-dictionary: "british"
%%% End:

%%% Local Variables:
%%% mode: latex
%%% TeX-master: "UniversalModuliSpaceOfZ2ZHarmonicSpinors"
%%% ispell-local-dictionary: "british"
%%% End:

\appendix
\section{Smoothing operators}
\label{Sec_SmoothingOperators}

The following appendix supplies the missing ingredient in the proof of \autoref{Prop_K_Tame};
that is: it constructs a family of smoothing operators for the scale of norms $\paren{\Abs{-}_{\mathring K^k} : k \in \N_0}$ defined by
\begin{align*}
  \Abs{\phi}_{\mathring K^k}
  \coloneq
  \Abs{\phi}_{H_b^{k+2}} + \Abs{\mathring\nabla\phi}_{H_b^{k+1}}
  + \Abs{\log(r) \mathring D\phi}_{H_b^{k+1}} + \Abs{\mathring\nabla \log(r) \mathring D\phi}_{H_b^{k}}.
\end{align*}

\begin{prop}
  \label{Prop_SmoothingOperators}
  There is a family of smoothing operators $\paren{S_\epsilon : \epsilon \in (0,1) }$ on $L^2\Gamma\paren{NZ\setminus Z, \mathring S \otimes \mathring\fl}$ such that
  for every
  $\epsilon \in (0,1)$, \ 
  $\phi \in L^2\Gamma\paren{NZ\setminus Z, \mathring S \otimes \mathring\fl}$, and
  $k,\ell \in \N_0$ with $k \geq \ell$,
  \begin{align}
    \label{Eq_SmoothingOperators_Smoothing}
    \Abs{S_\epsilon\phi}_{\mathring K^k}
    &\lesssim_{k,\ell}
      \epsilon^{-k+\ell}
      \Abs{\phi}_{\mathring K^\ell}, \\
      %%% 
    \label{Eq_SmoothingOperators_Approximation}
    \Abs{(S_\epsilon-\one)\phi}_{\mathring K^\ell}
    &\lesssim_{k,\ell}
      \epsilon^{k-\ell}
      \Abs{\phi}_{\mathring K^k}, \\
      %%% 
    \label{Eq_SmoothingOperators_Derivative-}
    \epsilon\Abs{\del_\epsilon S_\epsilon\phi}_{\mathring K^k}
    &\lesssim_{k,\ell}
      \epsilon^{-k+\ell} \Abs{\phi}_{\mathring K^\ell}, \qand \\
      %%% 
    \label{Eq_SmoothingOperators_Derivative+}
    \epsilon\Abs{\del_\epsilon S_\epsilon\phi}_{\mathring K^\ell}
    &\lesssim_{k,\ell}
      \epsilon^{k-\ell} \Abs{\phi}_{\mathring K^k}.
  \end{align}  
\end{prop}

As explained in \cite[§3.4]{BeraWalpuski2025},
Fubini's theorem and spectral theory imply that the Hilbert space $L^2\Gamma\paren{NZ\setminus Z, \mathring S \otimes \mathring\fl}$ decomposes as
\begin{equation*}
  L^2\Gamma\paren{NZ\setminus Z, \mathring S \otimes \mathring\fl}
  =
  L^2\paren{(0,\infty),r \rd r, L^2(F, \underline S \otimes\underline\fl)}
  =
  \bigoplus_{(\lambda,\mu) \in \Spectrum}
  L^2\paren{(0,\infty),r\rd r;E_{\lambda,\mu}}.
\end{equation*}
Here $\sigma \subseteq \R^2$ is a discrete subset and the $E_{\lambda,\mu}$ are finite-dimensional as in \cite[Proposition 3.24]{BeraWalpuski2025}.
Moreover, by \cite[Proposition 4.8]{BeraWalpuski2025},
with respect to this decomposition
\begin{equation*}
  \Abs{\phi}_{H_b^k}^2
  \asymp_k
  \sum_{(\lambda,\mu)\in\Spectrum}
  \sum_{\ell=0}^k
  \bracket{(\lambda,\mu)}^{2\ell}
  \Abs{
    (r\del_r)^{k-\ell}\phi_{\lambda,\mu}
  }_{r^{-1/2}L^2}^2
\end{equation*}
and
\begin{align*}
  \Abs{\mathring \nabla \phi}_{H_b^k}^2
  \asymp_k
  \sum_{(\lambda,\mu) \in \Spectrum}
  \left(
  \bracket{\mu}^2 \Abs{\phi_{\lambda,\mu}}_{H_b^k}^2
  + \Abs*{\tfrac{\lambda}{r}\phi_{\lambda,\mu}}_{H_b^k}^2
  + \Abs{\del_r\phi_{\lambda,\mu}}_{H_b^k}^2
  \right).
\end{align*}
The model Dirac operator does not quite preserve $E_{\lambda,\mu}$ but rather the direct sum $E_{\lambda,\mu} \oplus E_{-(\lambda+1),-\mu}$, on which it is given by
\begin{equation*}
  \mathring{D}^{\lambda,\mu}
  \coloneq      
  \begin{pmatrix}
    \mu & J\paren{\del_r  + \frac{\lambda+1}{r}} \\
    J\paren{\del_r - \frac{\lambda}{r}} & -\mu
  \end{pmatrix}.
\end{equation*}
(The case $(\lambda,\mu) = (-1/2,0)$ requires a bit more care as explained in \cite[Definition 3.25]{BeraWalpuski2025}, but this is inconsequential for the remaining discussion.)

The proof of \autoref{Prop_K_Estimate}, applied to the model operator,
gives
\begin{equation*}
  \Abs{\log(r)\mathring\nabla\phi}_{H_b^{k+1}}
  \lesssim_k
  \Abs{\phi}_{\mathring K^k}.
\end{equation*}
Together with \cite[Proposition~4.8]{BeraWalpuski2025} and the above formula for $\mathring D^{\lambda,\mu}$, this yields
\begin{align*}
  \Abs{\phi}_{\mathring K^k}^2
  \asymp_k  
  \sum_{(\lambda,\mu) \in \Spectrum}
  \Bigl(
  &
    \Abs{\phi_{\lambda,\mu}}_{H_b^{k+2}}^2
    + \Abs{\del_r \phi_{\lambda,\mu}}_{H_b^{k+1}}^2
    + \Abs{\tfrac{\lambda}{r} \phi_{\lambda,\mu}}_{H_b^{k+1}}^2
    + \Abs{\log(r)\paren{\del_r - \tfrac{\lambda}{r}}\phi_{\lambda,\mu}}_{H_b^{k+1}}^2 \\
  &
    + \Abs{\del_r \log(r)\paren{\del_r - \tfrac{\lambda}{r}}\phi_{\lambda,\mu}}_{H_b^k}^2
    + \Abs{\tfrac{\lambda}{r} \log(r)\paren{\del_r - \tfrac{\lambda}{r}} \phi_{\lambda,\mu}}_{H_b^k}^2
  \Bigr).
\end{align*}
Defining $(S_\epsilon)$ as the composition of a spectral cutoff at $\epsilon^{-1}$ in $\lambda$ and $\mu$,
as in \cites[Appendix B]{Parker2023:Deformation}[§4.5]{BeraWalpuski2025}, 
and the following smoothing operators in the radial direction establishes \autoref{Prop_SmoothingOperators}.

\begin{prop}
  \label{Prop_SmoothingOperators_1D}
  On $L^2((0,\infty),r \rd r;\C)$ define
  \begin{equation*}
    \Abs{\phi}_{H_b^k} \coloneq \sum_{\ell=0}^k \Abs{(r\del_r)^\ell \phi}_{r^{-1/2}L^2}
  \end{equation*}
  and for every $k \in \N_0$ and $\lambda \in \R$ set
  \begin{align*}
    \Abs{\phi}_{\mathring K_\lambda^k}
    &\coloneq
      \Abs{\phi}_{H_b^{k+2}}
      + \Abs{\del_r \phi}_{H_b^{k+1}}
      + \Abs{\tfrac{\lambda}{r} \phi}_{H_b^{k+1}} 
      + \Abs{\log(r)\paren{\del_r - \tfrac{\lambda}{r}}\phi}_{H_b^{k+1}} \\
    &\quad
      + \Abs{\del_r \log(r)\paren{\del_r - \tfrac{\lambda}{r}}\phi}_{H_b^k}
      + \Abs{\tfrac{\lambda}{r} \log(r)\paren{\del_r - \tfrac{\lambda}{r}} \phi}_{H_b^k}.
  \end{align*}
  There is a family of smoothing operators $(S_\epsilon : \epsilon \in (0,1))$ satisfying the analogues of \autoref{Eq_SmoothingOperators_Smoothing}, \autoref{Eq_SmoothingOperators_Approximation}, \autoref{Eq_SmoothingOperators_Derivative-}, and \autoref{Eq_SmoothingOperators_Derivative+}.
\end{prop}

The family $(S_\epsilon)$ in \autoref{Prop_SmoothingOperators_1D} is constructed by \defined{Mellin convolution} with a family of smoothing kernels $\sigma_\epsilon$;
that is:
\begin{equation*}
  (S_\epsilon \phi)(r) \coloneq (\sigma_\epsilon *_\sM \phi)(r) \coloneq \int_0^\infty \sigma_\epsilon(r/s) \phi(s) \,\frac{\rd s}{s}.
\end{equation*}
The Mellin convolution $*_\sM$ on $(0,\infty)$ is related to the (usual) convolution $*$ on $\R$ via the isometry $T \co L^2\paren{(0,\infty), r \rd r; \C} \to L^2\paren{\R;\C}$ defined by
\begin{equation*}
  T(f)(x) \coloneq e^xf(e^x);
\end{equation*}
indeed
\begin{equation*}
  (f *_\sM g)(r)
  =
  T^{-1}(Tf * Tg)(r).
\end{equation*}
Denote by $\sF \co L^2\paren{\R;\C} \to L^2\paren{\R;\C}$ the Fourier transform.
The \defined{Mellin transform}
\begin{equation*}
  \sM \co L^2\paren{(0,\infty), r\rd r; \C}\to L^2\paren{\R;\C}
\end{equation*}
is defined by
\begin{equation*}
  \sM \coloneq \sF \circ T.
\end{equation*}
In particular,
\begin{equation*}
  \sM(f*_\sM g) = \sM(f)\sM(g).
\end{equation*}
Combining the identities $T(r\del_r)T^{-1} = \del_x - 1$ and $\sF \del_x \sF^{-1} = 2\pi i \xi$ shows that for every $\phi \in C^\infty((0,\infty),\C)$ the norms are expressed in terms of the Mellin transform by 
\begin{equation}
  \label{Eq_HbkMellin}
  \Abs{\phi}_{H_b^k}
  =
  \sum_{\ell=0}^k \Abs{(r\del_r)^\ell\phi}_{r^{-1/2}L^2}
  \asymp_k
  \Abs{\bracket{\xi}^k\sM(\phi)}_{L^2}.
\end{equation}

The following is the key to \autoref{Prop_SmoothingOperators_1D}.

\begin{lemma}
  \label{Lem_FamilyOfSmoothingKernels}
  Let
  $\chi \in C_c^\infty(\R,[0,1])$ and $\rho \in C_c^\infty(\R,[0,1])$ such that
  \begin{equation*}
    \chi|_{[-1,1]} = 1, \quad
    \supp(\chi) \subseteq [-2,2], \quad
    \rho(0) = 1, \qandq
    \supp(\rho) \subseteq [-1,1].
  \end{equation*}
  Define the \defined{family of smoothing kernels} $\paren{\sigma_\epsilon : \epsilon \in (0,1) }$ by 
  \begin{equation*}
    \sigma_\epsilon(r) \coloneq \sM^{-1}(\chi_\epsilon)(r) \rho(\log(r))
  \end{equation*}
  with $\chi_\epsilon(\xi) \coloneq \chi(\epsilon\xi)$.
  The following hold:
  \begin{enumerate}
  \item
    \label{Lem_FamilyOfSmoothingKernels_Estimates}
    For every $\epsilon \in (0,1)$, $k \in \N_0$, and $\xi \in \R$
    \begin{align}
      \label{Eq_FamilyOfSmoothingKernel_Smoothing}
      \abs{\sM(\sigma_\epsilon)(\xi)}
      &\lesssim_{\rho,k}
        \epsilon^{-k}\bracket{\xi}^{-k}, \\
        %%% 
      \label{Eq_FamilyOfSmoothingKernel_Approximation}
      \abs{\sM(\sigma_\epsilon)(\xi)-1}
      &\lesssim_{\rho,k}
        \epsilon^k\bracket{\xi}^k, \\
        %%% 
      \label{Eq_FamilyOfSmoothingKernel_Derivative-}
      \abs{\epsilon\del_\epsilon \sM(\sigma_\epsilon)(\xi)}
      &\lesssim_{\chi,\rho,k}
        \epsilon^{-k}\bracket{\xi}^{-k}, \qand \\
        %%% 
      \label{Eq_FamilyOfSmoothingKernel_Derivative+}
      \abs{\epsilon\del_\epsilon \sM(\sigma_\epsilon)(\xi)}
      &\lesssim_{\chi,\rho,k}
        \epsilon^{k}\bracket{\xi}^{k}.
    \end{align}
  \item
    \label{Lem_FamilyOfSmoothingKernels_WeightsLog}
    For every $\epsilon \in (0,1]$, $k \in \Z$, and $\xi \in \R$
    \begin{equation*}
      \abs{\epsilon^{-1}\sM(\log(r) \sigma_\epsilon)(\xi)}
      \lesssim_{\chi,\rho,k}
      \epsilon^k\bracket{\xi}^k \qandq
      \abs{\del_\epsilon \sM(\log(r) \sigma_\epsilon)(\xi)}
      \lesssim_{\chi,\rho,k}
      \epsilon^k\bracket{\xi}^k.
    \end{equation*}    
  \end{enumerate}
\end{lemma}

\begin{proof}
  {\bfseries \autoref{Lem_FamilyOfSmoothingKernels_Estimates}}\quad  
  By direct computation, using $\int_\R \hat \rho = \rho(0) = 1$,  
  \begin{equation*}
    \sM(\sigma_\epsilon) = \chi_\epsilon * \hat\rho, \quad
    \sM(\sigma_\epsilon) -1 = (\chi_\epsilon-1) * \hat\rho, \qandq
    \del_\epsilon\sM(\sigma_\epsilon) = \del_\epsilon\chi_\epsilon * \hat\rho.
  \end{equation*}
  %%% 
  Since $\rho\in C_c^\infty(\R) \subseteq \sS(\R)$, $\hat\rho \in \sS(\R)$ and, therefore,
  \begin{equation*}
    \abs{\hat\rho(\xi)} \lesssim_{\rho,k} \bracket{\xi}^{-k}.
  \end{equation*}
  By Young's inequality and since $\Abs{\chi_\epsilon}_{L^\infty} \leq 1$, $\Abs{\epsilon\del_\epsilon\chi_\epsilon}_{L^\infty} \lesssim_\chi 1$, and $\Abs{\hat\rho}_{L^1} \lesssim_\rho 1$,
  \begin{equation*}
    \Abs{\sM(\sigma_\epsilon)}_{L^\infty} \lesssim_\rho 1, \quad
    \Abs{\sM(\sigma_\epsilon)-1}_{L^\infty} \lesssim_\rho 1 \qandq
    \Abs{\epsilon\del_\epsilon\sM(\sigma_\epsilon)}_{L^\infty} \lesssim_{\chi,\rho} 1.
  \end{equation*}
  This implies \autoref{Eq_FamilyOfSmoothingKernel_Smoothing} and \autoref{Eq_FamilyOfSmoothingKernel_Derivative-} assuming $\epsilon\bracket{\xi} \leq 4$, and
  \autoref{Eq_FamilyOfSmoothingKernel_Approximation} and \autoref{Eq_FamilyOfSmoothingKernel_Derivative+} assuming $\epsilon\bracket{\xi} \geq 1/2$.

  Since $\abs{\bracket{\xi}-\bracket{\eta}} \leq \abs{\xi-\eta}$,
  if $\epsilon\bracket{\xi} \geq 4$ and $\abs{\xi-\eta} \leq 2\epsilon^{-1}$,
  then
  \begin{equation*}
    \bracket{\eta} \geq \bracket{\xi} - \abs{\xi-\eta} \geq \frac12\bracket{\xi}.
  \end{equation*}
  Therefore and since $\supp(\chi) \subseteq [-2,2]$ and $\epsilon \leq 1$,
  if $\epsilon\bracket{\xi} \geq 4$,
  then
  \begin{equation*}
    \abs{\sM(\sigma_\epsilon)(\xi)}
    %%% 
    \leq
    \int_{\abs{\xi - \eta} \leq 2\epsilon^{-1}} \abs{\hat\rho(\eta)} \,\rd \eta
    %%% 
    \lesssim_{\rho,k}
    \epsilon^{-1}\bracket{\xi}^{-k-1}
    %%% 
    \lesssim_k
    \epsilon^{-k}\bracket{\xi}^{-k}.
  \end{equation*}
  Combined with the above observation,
  this proves \autoref{Eq_FamilyOfSmoothingKernel_Smoothing} unconditionally.
  The same argument also proves \autoref{Eq_FamilyOfSmoothingKernel_Derivative-} unconditionally.

  Since $\abs{\xi-\eta} \leq \abs{\xi} + \abs{\eta} \leq \bracket{\xi} + \bracket{\eta}$,
  if $\epsilon\bracket{\xi} \leq 1/2$ and $\abs{\xi-\eta} \geq \epsilon^{-1}$,
  then
  \begin{equation*}
    \bracket{\eta} \geq \abs{\xi-\eta} - \bracket{\xi} \geq \frac12 \epsilon^{-1}.
  \end{equation*}
  Therefore and since $\supp (\chi-1) \cap [-1,1] = \emptyset$ and $\bracket{\xi} \geq 1$,
  if $\epsilon\bracket{\xi} \leq 1/2$,
  then
  \begin{equation*}
    \abs{\sM(\sigma_\epsilon)(\xi)-1}
    %%% 
    \leq
    \int_{\abs{\xi - \eta} \geq \epsilon^{-1}} \abs{\hat\rho(\eta)} \,\rd \eta
    %%% 
    \lesssim_{\rho,k}
    \int_{\abs{\xi - \eta} \geq \epsilon^{-1}} \bracket{\eta}^{-k-2} \,\rd \eta
    %%% 
    \lesssim_k
    \epsilon^k
    %%% 
    \lesssim_k
    \epsilon^k\bracket{\xi}^k.
  \end{equation*}
  Combined with the above observation,
  this proves \autoref{Eq_FamilyOfSmoothingKernel_Approximation} unconditionally.
  The same argument also proves \autoref{Eq_FamilyOfSmoothingKernel_Derivative+} unconditionally.
  
  \medskip
  \noindent
  {\bfseries \autoref{Lem_FamilyOfSmoothingKernels_WeightsLog}}\quad
  By direct computation,
  \begin{equation*}
    \sM(\log(r)\sigma_\epsilon) = \frac{i}{2\pi} \chi_\epsilon' * \hat\rho.
  \end{equation*}
  Since $\Abs{\chi'_\epsilon}_{L^\infty} \lesssim_\chi \epsilon$ and $\Abs{\epsilon\del_\epsilon\chi'_{\epsilon}}_{L^\infty} \lesssim_\chi \epsilon$, and
  $\supp(\chi_\epsilon') \subseteq [-2\epsilon^{-1},2\epsilon^{-1}]$ and  $\supp(\chi_\epsilon') \cap [-\epsilon^{-1},\epsilon^{-1}] = \emptyset$,
  the above arguments prove the asserted estimates.  
\end{proof}

\begin{proof}[Proof of \autoref{Prop_SmoothingOperators_1D}]
  Choose $\chi$ and $\rho$ as in \autoref{Lem_FamilyOfSmoothingKernels} and construct $\sigma_\epsilon$ accordingly.
  Moreover, for every $\nu \in \R$ set $\sigma_\epsilon^{(\nu)} \coloneq r^{-\nu}\sigma_\epsilon$ and $\delta_\epsilon^{(\nu)} \coloneq \log(r) \sigma_\epsilon^{(\nu)}$.
  Observe that $\sigma_\epsilon^{(\nu)}$ arises from $\chi$ and $\rho^{(\nu)}$ defined by $\rho^{(\nu)}(x) \coloneq e^{-\nu x}\rho(x)$.
  In particular, \autoref{Lem_FamilyOfSmoothingKernels} applies to $\sigma_\epsilon^{(\nu)}$ mutatis mutandis.

  By \autoref{Eq_HbkMellin} and \autoref{Lem_FamilyOfSmoothingKernels}~\autoref{Lem_FamilyOfSmoothingKernels_Estimates},
  $\paren{ S_\epsilon \coloneq \sigma_\epsilon *_\sM - : \epsilon \in (0,1) }$
  satisfies the analogues of \autoref{Eq_SmoothingOperators_Smoothing}, \autoref{Eq_SmoothingOperators_Approximation}, \autoref{Eq_SmoothingOperators_Derivative-}, and \autoref{Eq_SmoothingOperators_Derivative+} with respect to $\Abs{-}_{H_b^k}$.
  
  Therefore, it remains to analyse the behaviour of the five differential operators appearing in $\Abs{-}_{\mathring K_\lambda^k}$ with respect to $S_\epsilon$.
  By direct computation,
  \begin{align*}
    r\del_r (\sigma *_\sM \phi)
    &=
      \sigma *_\sM r\del_r\phi, \\
      %%% 
    r^\nu (\sigma  *_\sM \phi)
    &=
      \paren{r^\nu\sigma} *_\sM \paren{r^\nu \phi}, \qand \\
    %%% 
    \log(r) (\sigma *_\sM\phi)
    &=
      \sigma *_\sM \log(r) \phi + \log(r)\sigma *_\sM \phi.
  \end{align*}
  Consequently,
  \begin{align*}
    \del_rS_\epsilon \phi
    &= \sigma_\epsilon^{(1)} *_\sM \del_r\phi, \\
    %%%
    \tfrac{\lambda}{r}S_\epsilon \phi
    &= \sigma_\epsilon^{(1)} *_\sM \tfrac{\lambda}{r}\phi, \\
    %%%
    \log(r)\paren{\del_r - \tfrac{\lambda}{r}}S_\epsilon \phi
    &=
      \sigma_\epsilon^{(1)} *_\sM \log(r)\paren{\del_r - \tfrac{\lambda}{r}}\phi
      + \delta_\epsilon^{(1)} *_\sM \paren{\del_r - \tfrac{\lambda}{r}}\phi, \\
      %%% 
    \del_r\log(r)\paren{\del_r - \tfrac{\lambda}{r}}S_\epsilon \phi
    &=
      \sigma_\epsilon^{(2)} *_\sM \del_r\log(r)\paren{\del_r - \tfrac{\lambda}{r}}\phi
      + \delta_\epsilon^{(2)} *_\sM \del_r \paren{\del_r - \tfrac{\lambda}{r}}\phi, \\
      %%% 
    \tfrac{\lambda}{r}\log(r)\paren{\del_r - \tfrac{\lambda}{r}}S_\epsilon \phi
    &=
      \sigma_\epsilon^{(2)} *_\sM \tfrac{\lambda}{r}\log(r)\paren{\del_r - \tfrac{\lambda}{r}}\phi +
      \delta_\epsilon^{(2)} *_\sM \tfrac{\lambda}{r}\paren{\del_r - \tfrac{\lambda}{r}}\phi.
  \end{align*}
  This combined with
  \begin{align*}
    \Abs{\paren{\del_r - \tfrac{\lambda}{r}}\phi}_{H_b^k}
    &\lesssim
      \Abs{\log(r)\paren{\del_r - \tfrac{\lambda}{r}}\phi}_{H_b^k}, \\
      %%% 
    \Abs{\del_r\paren{\del_r - \tfrac{\lambda}{r}}\phi}_{H_b^k}
    &\lesssim
      \Abs{\del_r\log(r)\paren{\del_r - \tfrac{\lambda}{r}}\phi}_{H_b^k}
      + \Abs{\tfrac{\lambda}{r}\paren{\del_r - \tfrac{\lambda}{r}}\phi}_{H_b^k}, \\
      %%% 
    \Abs{\tfrac{\lambda}{r}\paren{\del_r - \tfrac{\lambda}{r}}\phi}_{H_b^k}
    &\lesssim
      \Abs{\tfrac{\lambda}{r}\log(r)\paren{\del_r - \tfrac{\lambda}{r}}\phi}_{H_b^k},
  \end{align*}
  and \autoref{Lem_FamilyOfSmoothingKernels} proves the assertion.
\end{proof}

%%% Local Variables:
%%% mode: latex
%%% TeX-master: "UniversalModuliSpaceOfZ2ZHarmonicSpinors"
%%% ispell-local-dictionary: "british"
%%% End:

\section{The index formula}
\label{Sec_IndexFormula}

Throughout this appendix,
suppose that $X$ is oriented, $\dim X = 4$, and $\rk S = 8$.
Moreover, 
let $\bp = (g;\gamma,\nabla) \in \SpaceOfDiracBundles$ and $(Z,[\fl]) \in \SpaceOfRamifiedLineBundles$.

\begin{theorem}
  \label{Thm_IndexFormula}
  If $V^\pm \subseteq \ResidueBundle^\pm$ is an $\AlgebraBundle$--linear subbundle with $\rk_\AlgebraBundle V^\pm = 1$,
  then
  \begin{equation*}
    \ind D_{\ResidueCondition_{V^\pm}^\pm}^{\fl,\pm}
    =
    \mp\frac14 \Inner{p_1(S),[X]} \pm \sigma(X) \mp \frac14\Inner{e(NZ),[Z]} \pm \Inner{e(V^\pm),[Z]} + \frac12\chi(Z).
  \end{equation*}
\end{theorem}

\begin{remark}
  \label{Rmk_IndexFormula}
  In the situation of \autoref{Thm_IndexFormula},
  \begin{equation*}
    \frac12\chi(Z)\pm\frac14\Inner{e(NZ),[Z]}
  \end{equation*}
  is an integer.
  Indeed, by the formulae in \autoref{Rmk_IndexInSelfDualCase},
  \begin{equation*}
    \frac12\chi(Z)+\frac14\Inner{e(NZ),[Z]}
    =
    \chi(X)+\sigma(X)
    -\frac12\paren{\chi(\tilde X)+\sigma(\tilde X)}
  \end{equation*}
  and $\chi(X) + \sigma(X)$ is even for every closed oriented $4$--manifold $X$.
\end{remark}

The proof relies on the index computation in \autoref{Rmk_IndexInSelfDualCase} and the following two ingredients:
a relative index formula and a consequence of excision.

\begin{prop}
  \label{Prop_RelativeIndexFormula}
  If $V^+,W^+ \subseteq \ResidueBundle^+$ are $\AlgebraBundle$--linear subbundles with $\rk_\AlgebraBundle V^+ = \rk_\AlgebraBundle W^+ = 1$,
  then
  \begin{equation*}
    \ind D_{\ResidueCondition_{V^+}^+}^+ - \ind D_{\ResidueCondition_{W^+}^+}^+
    =
    \Inner{e(V^+),[Z]} - \Inner{e(W^+),[Z]}.
  \end{equation*}
\end{prop}

\begin{prop}
  \label{Prop_Excision}
  There is a $\AlgebraBundle$--linear subbundle $V_{D}^+ \subseteq \ResidueBundle^+$ with $\rk_\AlgebraBundle V_D^+ = 1$ satisfying
  \begin{equation*}
    e(V_D^+) = e(\ResidueBundle_D^+)
    \qandq
    \ind D_{\ResidueCondition_{V_D^+}^+}^{\fl,+} - \ind    D_{\mathrm{HdR},\ResidueCondition_{\ResidueBundle_D^+}^+}^{\fl,+}
    =
    \ind D^+ - \ind D_{\mathrm{HdR}}^+.
  \end{equation*}
  Here $D_{\mathrm{HdR}}$ denotes the Dirac operator associated with the Dirac bundle structure $\HdR(g) \in \SpaceOfDiracBundles$ and $\ResidueBundle_D^+$ is as in \autoref{Def_OneFormAndSelfDualTwoFormDirichletResidueCondition}.
\end{prop}

\begin{proof}[Proof of \autoref{Thm_IndexFormula} assuming \autoref{Prop_RelativeIndexFormula} and \autoref{Prop_Excision}]
  It suffices to prove the formula in positive chirality,
  because the complementary $V^\pm \subseteq \ResidueBundle^\pm$,
  related by $V^- = (JV^+)^\perp$,
  satisfy
  \begin{equation*}
    e(V^-) = e(V^+) + e(TZ),
  \end{equation*}
  because $J$ is $\AlgebraBundle$--anti-linear and $\overline{\paren{V^+}^\perp} \iso V^+ \otimes_\AlgebraBundle TZ$ via Clifford multiplication.

  By \autoref{Prop_RelativeIndexFormula}, \autoref{Prop_Excision}, and \autoref{Rmk_IndexInSelfDualCase},
  \begin{align*}
    \ind D_{\ResidueCondition_{V^+}^+}^{\fl,+}
    &=
      \ind D_{\ResidueCondition_{V_D^+}^+}^{\fl,+}
      + \Inner{e(V^+),[Z]}
      -  \Inner{e(\ResidueBundle_D^+),[Z]} \\
    &=
      \ind D^+
      + \ind D_{\mathrm{HdR},\ResidueCondition_{\ResidueBundle_D^+}^+}^{\fl,+}
      - \ind D_{\mathrm{HdR}}^+
      - \Inner{e(\ResidueBundle_D^+),[Z]}
      + \Inner{e(V^+),[Z]} \\
    &=
      -\frac14 \Inner{p_1(S),[X]} + \sigma(X) \\
    &\quad
      + \frac12\chi(Z)
      + \frac14 \Inner{e(NZ),[Z]}
      - \Inner{e(\ResidueBundle_D^+),[Z]}
      + \Inner{e(V^+),[Z]}.
  \end{align*}
  Since $\ResidueBundle_D^+ = N^*Z \otimes_\AlgebraBundle NZ^{-1/2}$, and $N^*Z \iso NZ$ and $NZ^{-1/2} \otimes_\bA NZ^{-1/2} \iso \overline{NZ}$ as $\bA$--bundles,
  \begin{equation*}
    2e(\ResidueBundle_D^+) = e(NZ).
  \end{equation*}

  It remains to prove that
  \begin{equation*}
    \ind D^+= -\frac14 \Inner{p_1(S),[X]} + \sigma(X).
  \end{equation*}
  As in \cite[Proof of Theorem 6.1]{Atiyah1978:SelfDuality},
  for the purpose of determining the index formula,
  there is no loss in assuming that $X$ is spin.
  Denote by $\slS$ the \emph{real} spinor bundle.
  There is some quaternionic line bundle $L$ such that $S = \slS \otimes_\H L$.
  By the Atiyah--Singer index formula \cites[Theorem 5.3]{AtiyahSinger1968:Index3}[Part III Theorem 13.10]{Lawson1989},
  \begin{equation*}
    \ind D^+
    = -\Inner{\hat A(X)\ch(L),[X]}
    = \frac14\sigma(X) + \Inner{c_2(L),[X]}.
  \end{equation*}
  The sign is a consequence of the real and complex chirality operators differing by a factor of $-1$.
  It remains to determine $\Inner{c_2(L),[X]}$.
  Since
  $p_1(S^\pm) = -c_2\paren{\slS^\pm \otimes_\C L} = -2\paren{c_2(\slS^\pm) + c_2(L)}$
  and
  $c_2(\slS^\pm) = -\frac14 p_1(TX) \pm \frac12 e(TX)$,
  \begin{align*}
    \Inner{c_2(L),[X]}
    &=
      -\frac12\Inner{p_1(S^\pm),[X]}
      -\Inner{c_2(\slS^\pm),[X]} \\
    &=
      -\frac12\Inner{p_1(S^\pm),[X]}
      +\frac14\Inner{p_1(TX),[X]}
      \mp\frac12\Inner{e(TX),[X]} \\
    &=
      -\frac12\Inner{p_1(S^\pm),[X]}
      +\frac34\sigma(X)
      \mp\frac12\Inner{e(TX),[X]}.
  \end{align*}
  The above combine to the stated contribution.  
\end{proof}

\begin{proof}[Proof of \autoref{Prop_Excision}]
  Denote by $S_{\mathrm{HdR}}$ the Clifford module underlying $D_{\mathrm{HdR}}$.
  The bundle $\Hom_{\Cl}(S,S_{\mathrm{HdR}})$ has rank four and every non-zero homomorphism is an isomorphism, in fact, an isometry after normalisation.
  A simple transversality argument shows that $S$ and $S_{\mathrm{HdR}}$ are isomorphic as Clifford modules outside a finite subset avoiding $Z$.
  Denote by $V_D^+$ the image of $\ResidueBundle_D^+$ under such a choice of isomorphism.
  Up to a lower order perturbation that does not affect the index,
  this also identifies $D_{\ResidueCondition_{V_D^+}^+}^{\fl,+}$ and $D_{\mathrm{HdR},\ResidueCondition_{\ResidueBundle_D^+}^+}^{\fl,+}$ outside the same finite subset.
  The assertion now follows by the usual excision argument \cite[§8]{AtiyahSinger1968:Index1}.
\end{proof}

The remainder of this appendix is devoted to the proof of \autoref{Prop_RelativeIndexFormula}.
The following localises the proof of this relative index formula to $Z$.

\begin{prop}
  \label{Prop_ComparisonWithAPS}
  Denote by $\ResidueCondition_\APS^+ \coloneq \one_{(-\infty,0)}(A^+)H^{1/2}\Gamma\paren{Z,\ResidueBundle^+}$ the positive APS residue condition.
  If $V^+ \subseteq \ResidueBundle^+$ is an $\AlgebraBundle$--linear subbundle with $\rk_\AlgebraBundle V^+ = 1$,
  then
  \begin{equation*}
    \ind D_{\ResidueCondition_{V^+}^+}^{\fl,+} - \ind D_{\ResidueCondition_\APS^+}^{\fl,+}
    =
    \ind\paren{ \one_{(-\infty,0)}(A^+) \co \ResidueCondition_{V^+}^+ \to \ResidueCondition_\APS^+ }.
  \end{equation*}
\end{prop}

\begin{proof}%[Proof of \autoref{Prop_ComparisonWithAPS}]
  By the chiral version of \cite[Proposition 2.24 and Lemma 2.25]{BeraWalpuski2025},
  for every positive Fredholm residue condition $\ResidueCondition^+ \subseteq \GelfandRobbinQuotient^+$,
  \begin{equation*}
    \ind D_{\ResidueCondition^+}^{\fl,+}
    =
    \dim\ker\Dmin^{\fl,+} - \dim\coker\Dmax^{\fl,+}+ \ind  \delta_{\ResidueCondition^+}
  \end{equation*}
  with
  \begin{equation*}
    \delta_{\ResidueCondition^+} \coloneq \pr_{\GelfandRobbinQuotient^+/\Lambda^+}|_{\ResidueCondition^+} \co {\ResidueCondition^+} \to \GelfandRobbinQuotient^+/\Lambda^+
  \end{equation*}
  denoting the restriction of the projection map.
  
  The difference
  \begin{equation*}
    \delta_{\ResidueCondition_{V^+}^+} - \delta_{\ResidueCondition_\APS^+} \one_{(-\infty,0)}(A^+)|_{\ResidueCondition_{V^+}^+}
    =
    \pr_{\GelfandRobbinQuotient^+/\Lambda^+}
    \one_{[0,\infty)}(A^+)|_{\ResidueCondition_{V^+}^+}
  \end{equation*}
  is compact because the right-hand side factors through the compact inclusion
  \begin{equation*}
    \one_{[0,\infty)}(A^+)H^{1/2}\Gamma(Z,\ResidueBundle^+) \into \one_{[0,\infty)}(A^+)H^{-1/2}\Gamma(Z,\ResidueBundle^+).
  \end{equation*}
  Therefore,
  \begin{equation*}
    \ind D_{\ResidueCondition_{V^+}^+}^{\fl,+} - \ind D_{\ResidueCondition_\APS^+}^{\fl,+}
    =
    \ind \delta_{\ResidueCondition_{V^+}^+} - \ind \delta_{\ResidueCondition_\APS^+}
    =
    \ind \one_{(-\infty,0)}(A^+)|_{\ResidueCondition_{V^+}^+}.
    \qedhere
  \end{equation*}
\end{proof}

The above reduces the proof of \autoref{Prop_RelativeIndexFormula} to the following relative index formula.

\begin{prop}
  \label{Prop_SpectralProjectionRelativeIndex}
  If $V^+,W^+ \subseteq \ResidueBundle^+$ are $\AlgebraBundle$--linear subbundles with
  $\rk_\AlgebraBundle V^+ = \rk_\AlgebraBundle W^+ = 1$,
  then
  \begin{multline*}
    \ind\paren*{\one_{(-\infty,0)}(A^+) \co \ResidueCondition_{V^+}^+ \to \ResidueCondition_\APS^+ }
    - \ind\paren*{\one_{(-\infty,0)}(A^+) \co \ResidueCondition_{W^+}^+ \to \ResidueCondition_\APS^+ } \\
    =
    \Inner{e(V^+),[Z]}-\Inner{e(W^+),[Z]}.
  \end{multline*}
\end{prop}

\begin{proof}%[Proof of \autoref{Prop_SpectralProjectionRelativeIndex}]
  The proof consists of four steps.
  
  \begin{step}
    \label{St_ReductionToOrientable}
    Reduction to orientable $Z$.
  \end{step}

  The right-hand side is multiplicative under covers.
  The upcoming argument shows that the same is true for the left-hand side by exhibiting it as the index of a pseudo-differential operator.
  Denote the orthogonal projection onto $W^+$ by $\pr_{W^+}$.
  As explained in \cite[Proof of Theorem 4.40]{BeraWalpuski2025},
  \begin{equation*}
    \pr_{W^+}-\one_{[0,\infty)}(A^+)
  \end{equation*}
  is an elliptic pseudo-differential operator.
  Since
  \begin{equation*}
    \pr_{W^+}-\one_{[0,\infty)}(A^+)
    =
    \one_{(-\infty,0)}(A^+)\pr_{W^+}
    -\one_{[0,\infty)}(A^+)\paren{\one -\pr_{W^+}},
  \end{equation*}
  a pseudo-differential parametrix for $\pr_{W^+}-\one_{[0,\infty)}(A^+)$ induces a pseudo-differential parametrix
  \begin{equation*}
    G_{W^+}\co \ResidueCondition_\APS^+\to \ResidueCondition_{W^+}^+ = H^{1/2}\Gamma\paren{Z,W^+}
  \end{equation*}
  of order zero.
  Therefore,
  \begin{equation*}
    G_{W^+}\one_{(-\infty,0)}(A^+)
    \co H^{1/2}\Gamma(Z,V^+)\to H^{1/2}\Gamma(Z,W^+)
  \end{equation*}
  is an elliptic pseudo-differential operator with
  \begin{multline*}
    \ind\paren{
      G_{W^+}\one_{(-\infty,0)}(A^+)
      \co H^{1/2}\Gamma(Z,V^+)\to H^{1/2}\Gamma(Z,W^+)
    } \\
    =
    \ind\paren{\one_{(-\infty,0)}(A^+) \co \ResidueCondition_{V^+}^+\to \ResidueCondition_\APS^+}
    -
    \ind\paren{\one_{(-\infty,0)}(A^+) \co \ResidueCondition_{W^+}^+\to \ResidueCondition_\APS^+}.
  \end{multline*}
  
  Let $\pi \co \tilde Z\to Z$ be the orientation cover.
  The above setup pulls back along $\pi$.
  By \cite[Theorem 2.12]{AtiyahSinger1968:Index3},
  the index doubles after pulling back.
  Moreover
  \begin{equation*}
    \Inner{e(\pi^*V^+)-e(\pi^*W^+),[\tilde Z]}
    =
    2\Inner{e(V^+)-e(W^+),[Z]}.
  \end{equation*}
  Therefore, it suffices to verify the asserted index formula for oriented $Z$.

  \medskip

  Henceforth, $Z$ is assumed to be oriented.
  As a consequence $\AlgebraBundle=\C$ and $I$ is a complex structure on $\ResidueBundle$.
  The remaining steps prove that,
  in this situation,
  \begin{equation}
    \label{Eq_SpectralProjectionIndex}
    \ind\paren*{\one_{(-\infty,0)}(A^+) \co \ResidueCondition_{V^+}^+ \to \ResidueCondition_\APS^+ }
    =
    \frac12\dim\ker A^+ + \Inner{e(V^+),[Z]} + \frac12\chi(Z),
  \end{equation}
  and analogously for $W^+$.  
  
  \begin{step}
    \label{Step_VariationOnAPS}
    Variations on the APS residue condition.
  \end{step}

  Given a self-adjoint first-order elliptic operator $T$ on $\ResidueBundle^+$ with symbol $\sigma_{T}(\xi) = \sigma_{A^+}(\xi) = -J\gamma(\xi)$ and $TI = -IT$,
  set
  \begin{equation*}
    \ResidueCondition_{\APS,T}^+ \coloneq \one_{(-\infty,0)}(T)H^{1/2}\Gamma\paren{Z,\ResidueBundle^+}.
  \end{equation*}
  Of course, $T = A^+$ recovers the usual positive APS residue condition.  
  Since $TI=-IT$,
  \begin{equation}
    \label{Eq_ConjugationSwapsSign}
    \one_{(0,\infty)}(T)
    =
    -I\one_{(-\infty,0)}(T)I
  \end{equation}
  and, in particular,
  \begin{equation*}    
    H^{1/2}\Gamma\paren{Z,\ResidueBundle^+}
    =
    \ResidueCondition_{\APS,T}^+
    \oplus
    \ker T
    \oplus
    I\ResidueCondition_{\APS,T}^+.
  \end{equation*}
  
  If $T_1,T_2$ are two such operators, then
  \begin{equation}
    \label{Eq_APSRelativeIndexFormula}
    \ind\paren[\big]{
      \one_{(-\infty,0)}(T_2) \co \ResidueCondition_{\APS,T_1}^+ \to \ResidueCondition_{\APS,T_2}^+
    }
    =
    \frac12\paren{\dim\ker T_2-\dim\ker T_1}.
  \end{equation}
  This holds for the following reason.
  Since $T_1$ and $T_2$ have the same principal symbol,
  \begin{equation*}
    \one_{(-\infty,0)}(T_1)-\one_{(-\infty,0)}(T_2)
    \qandq
    \one_{(0,\infty)}(T_1)-\one_{(0,\infty)}(T_2)
  \end{equation*}
  are of order $-1$ and, therefore, are compact on $H^{1/2}\Gamma\paren{Z,\ResidueBundle^+}$.
  Therefore,
  \begin{equation*}
    \one_{(-\infty,0)}(T_2)\one_{(-\infty,0)}(T_1)
    + \one_{\set{0}}(T_2)\one_{\set{0}}(T_1)
    + \one_{(0,\infty)}(T_2)\one_{(0,\infty)}(T_1)
  \end{equation*}
  is a compact perturbation of the identity and thus has index zero.
  By \autoref{Eq_ConjugationSwapsSign},
  With respect to the decompositions
  $H^{1/2}\Gamma\paren{Z,\ResidueBundle^+} = \ResidueCondition_{\APS,T_1}^+ \oplus \ker T_1 \oplus I\ResidueCondition_{\APS,T_1}^+$ $    H^{1/2}\Gamma\paren{Z,\ResidueBundle^+} = \ResidueCondition_{\APS,T_2}^+ \oplus \ker T_2 \oplus I\ResidueCondition_{\APS,T_2}^+$,
  this index zero operator is of the form
  \begin{equation*}
    \begin{pmatrix}
      \one_{(-\infty,0)}(T_2) & 0 & 0 \\
      0 & \one_{\set{0}}(T_2) & 0 \\
      0 & 0 & I\one_{(0,\infty)}(T_2)I^{-1}
    \end{pmatrix}.
  \end{equation*}
  Therefore,
  \begin{align*}
    0
    &=
      \ind\paren{\one_{(-\infty,0)}(T_2) \co \ResidueCondition_{\APS,T_1}^+ \to \ResidueCondition_{\APS,T_2}^+} 
      + \dim \ker T_1 - \dim\ker T_2 \\
    &\quad
      + \ind\paren{I\one_{(-\infty,0)}(T_2)I^{-1} \co I\ResidueCondition_{\APS,T_1}^+ \to I\ResidueCondition_{\APS,T_2}^+} \\
    &=
      2\ind\paren{\one_{(-\infty,0)}(T_2) \co \ResidueCondition_{\APS,T_1}^+ \to \ResidueCondition_{\APS,T_2}^+} 
      + \dim \ker T_1 - \dim\ker T_2.
  \end{align*}

  \begin{step}
    A variation of the APS residue condition adapted to $V^+$.
  \end{step}
  
  Decompose $A^+$ with respect to the decomposition $\ResidueBundle^+ = V^+ \oplus {V^+}^\perp$ as
  \begin{equation*}
    A^+
    =
    \begin{pmatrix}
      \ast & B^*\\
      B & \ast
    \end{pmatrix}.
  \end{equation*}
  A moment's thought shows that the operator
  \begin{equation*}
    T
    \coloneq
    \begin{pmatrix}
      0 & B^*\\
      B & 0
    \end{pmatrix}
  \end{equation*}
  is of the type considered above.
  
  The composition
  \begin{equation*}
    \ResidueCondition_{V^+}^+
    \xrightarrow{\one_{(-\infty,0)}(T)}
    \ResidueCondition_{\APS,T}^+
    \xrightarrow{\one_{(-\infty,0)}(A^+)}
    \ResidueCondition_\APS^+
  \end{equation*}
  differs from the direct projection onto $\ResidueCondition_\APS^+$ by
  \begin{equation*}
    \paren{\one_{(-\infty,0)}(A^+)-\one_{(-\infty,0)}(T)}
    \paren{\one_{(-\infty,0)}(T)-\one}|_{\ResidueCondition_{V^+}^+}.
  \end{equation*}
  The latter is compact because the first factor is of order $-1$.
  Therefore, by \autoref{Eq_APSRelativeIndexFormula},
  \begin{align*}
    \ind\paren[\big]{\one_{(-\infty,0)}(A^+) \co \ResidueCondition_{V^+}^+\to \ResidueCondition_\APS^+}
    &=
      \ind\paren[\big]{\one_{(-\infty,0)}(A^+) \co \ResidueCondition_{\APS,T}^+\to \ResidueCondition_\APS^+} \\
    &\quad
      + \ind\paren[\big]{\one_{(-\infty,0)}(T) \co \ResidueCondition_{V^+}^+ \to  \ResidueCondition_{\APS,T}^+} \\
    &=
      \frac12\paren[\big]{\dim\ker A^+-\dim\ker B - \dim\ker B^*} \\
    &\quad
      + \ind\paren[\big]{\one_{(-\infty,0)}(T) \co \ResidueCondition_{V^+}^+ \to  \ResidueCondition_{\APS,T}^+}.
  \end{align*}

  To compute the index of $\one_{(-\infty,0)}(T) \co \ResidueCondition_{V^+}^+ \to  \ResidueCondition_{\APS,T}^+$,
  consider the polar decomposition $B=U|B|$ with $U$ vanishing on $\ker B$.
  Since  
  \begin{equation*}
    \one_{(0,\infty)}(T)-\one_{(-\infty,0)}(T)
    =
    \begin{pmatrix}
      0 & U^*\\
      U & 0
    \end{pmatrix},
  \end{equation*}
  $U$ is a pseudo-differential operator of order zero.
  Moreover,
  \begin{equation*}
    \ResidueCondition_{\APS,T}^+
    =
    \set[\big]{
      (a,-Ua):
      a\in H^{1/2}\Gamma(Z,V^+) \cap \paren{\ker B}^\perp
    }.
  \end{equation*}
  In particular,
  with $\pi$ denoting the orthogonal projection onto $\ker B$,
  for every $b \in H^{1/2}\Gamma(Z,V^+)$
  \begin{equation*}
    \one_{(-\infty,0)}(T)(b,0)
    =
    \frac12\paren{b - \pi b,-U(b-\pi b)}.
  \end{equation*}
  Consequently,
  \begin{equation*}
    \one_{(-\infty,0)}(T)
    \co \ResidueCondition_{V^+}^+\to \ResidueCondition_{\APS,T}^+
  \end{equation*}
  is surjective and its kernel is $\ker B$.  
  The above, therefore, simplifies to
  \begin{equation*}
    \ind\paren[\big]{\one_{(-\infty,0)}(A^+) \co \ResidueCondition_{V^+}^+\to \ResidueCondition_\APS^+}
    =
    \frac12\paren[\big]{\dim\ker A^+ + \ind B}.
  \end{equation*}  
  
  \begin{step}
    Computation of $\ind B$.
  \end{step}

  The symbol of $B$ is $\sigma_B(\xi) = -J\gamma(\xi) \co V^+ \to \paren{V^+}^\perp$.
  The Clifford multiplication defines a complex anti-linear isomorphism $\paren{V^+}^\perp \iso V^+ \otimes_\C \Wedge^{0,1} T^*Z$.
  Treating this as an identification exhibits $B$ as a complex linear differential operator $\Gamma\paren{Z,V^+} \to \Omega^{0,1}\paren{Z,V^+}$ with the symbol of a Cauchy--Riemann operator $\bar\partial_{V^+}$ (up to a constant factor).
  Therefore,
  by Riemann--Roch, 
  \begin{equation*}
    \ind B
    =
    2\ind_\C\bar\partial_{V^+}
    =
    2\Inner{e(V^+),[Z]}+\chi(Z).
  \end{equation*}
  This combined with the previous step yields the desired \autoref{Eq_SpectralProjectionIndex}.
\end{proof}

\begin{proof}[Proof of \autoref{Prop_RelativeIndexFormula}]
  This is an immediate consequence of \autoref{Prop_ComparisonWithAPS} and \autoref{Prop_SpectralProjectionRelativeIndex}.
\end{proof}

%%% Local Variables:
%%% mode: latex
%%% TeX-master: "UniversalModuliSpaceOfZ2ZHarmonicSpinors"
%%% ispell-local-dictionary: "british"
%%% End:

\printreferences

\end{document}

%%% Local Variables:
%%% mode: latex
%%% TeX-master: t
%%% ispell-local-dictionary: "british"
%%% End: